\documentclass[3p]{elsarticle}
\makeatletter
\def\ps@pprintTitle{%
   \let\@oddhead\@empty
   \let\@evenhead\@empty
   \let\@oddfoot\@empty
   \let\@evenfoot\@oddfoot
}
\makeatother

\usepackage{amsmath,amsfonts,amsthm,amssymb,eucal,mathtools} 
\DeclareMathAlphabet{\bbol}{U}{bbold}{m}{n}
\usepackage{upgreek} 
\usepackage{graphicx,subfig,tikz} 
\usetikzlibrary{shadings,matrix}
\usetikzlibrary{decorations.markings}
\tikzset{->-/.style={decoration={markings,mark=at position .6 with {\arrow[scale=1.5]{stealth}}},postaction={decorate}}}
\usepackage{color} 
\usepackage{appendix} 

\usepackage{fp}
\newsavebox\foobox
\newcommand\slbox[2]{%
  \FPdiv{\result}{#1}{57.296}
  \FPtan{\result}{\result}%
  \slantbox[\result]{#2}%
}%
\newcommand{\slantbox}[2][30]{%
        \mbox{%
        \sbox{\foobox}{#2}%
        \hskip\wd\foobox
        \pdfsave
        \pdfsetmatrix{1 0 #1 1}%
        \llap{\usebox{\foobox}}%
        \pdfrestore
}}
\newcommand\rotslant[3]{\rotatebox{#1}{\slbox{#2}{#3}}}

\newtheorem{thm}{Theorem}[section]
\newtheorem{prop}[thm]{Proposition}

\newtheorem{oss}[thm]{Remark}

\newtheorem{cor}[thm]{Corollary}

\usepackage[norelsize,ruled,vlined,titlenumbered]{algorithm2e}
\SetKwFor{For}{for}{}{end for}

\newcommand{\id}{\text{id}}
\DeclareMathOperator{\grad}{grad}
\DeclareMathOperator{\curl}{curl}
\let\div\relax
\DeclareMathOperator{\div}{div}
\let\id\relax
\DeclareMathOperator{\id}{id}
\DeclareMathOperator{\rk}{rk}
\newcommand{\dR}{\text{dR}}
\newcommand{\cH}{\mathcal{H}}
\newcommand{\im}{\text{im}}
\newcommand{\RR}{\mathbb{R}}
\newcommand{\SSS}{\mathbb{S}}
\newcommand{\spn}{\text{span}\,}
\newcommand{\per}[1]{\mathring{#1}}

\newcommand{\cF}{\mathcal{F}}
\newcommand{\cM}{\mathcal{M}}
\newcommand{\cR}{\mathcal{R}}

\newcommand{\acpatch}{\scalebox{1.5}{$\pmb{\circlearrowleft}$}}
\newcommand{\cpatch}{\scalebox{1.5}{$\pmb{\circlearrowright}$}}
\newcommand{\gS}{\mathfrak{S}}
\newcommand{\R}{\Rightarrow}
\newcommand{\cA}{\mathcal{A}}
\newcommand{\m}[1]{\mathbf{#1}}
\renewcommand{\phi}{\varphi}
\newcommand{\defeq}{\coloneqq}

\begin{document}

\begin{frontmatter}
\title{Isogeometric de Rham complex discretization\\in solid toroidal domains}

\author[unifi]{Francesco Patrizi\corref{cor1}}
\ead{francesco.patrizi@unifi.it}
\author[tudelft]{Deepesh Toshniwal}
\ead{d.toshniwal@tudelft.nl}
\address[unifi]{Department of Mathematics and Computer Science ``Ulisse Dini'',\\ University of Florence, Viale Giovanni Battista Morgagni 67/a, 50134 Florence, Italy}
\address[tudelft]{Delft Institute of Applied Mathematics, Delft University of Technology, Delft, The Netherlands}
\cortext[cor1]{Corresponding Author}

\begin{abstract}
In this work we define a spline complex preserving the cohomological structure of the continuous de Rham complex when the underlying physical domain is a toroidal solid. In the spirit of the \emph{isogeometric analysis}, the spaces involved will be defined as pushforward of suitable spline spaces on a parametric domain.
The singularity of the parametrization of the solid will demand the imposition of smoothness constraints on the full tensor product spline spaces in the parametric domain to properly set up the discrete complex on the physical domain. 
\end{abstract}

\end{frontmatter}

\section{Introduction}
Isogeometric Analysis (IgA), introduced in \cite{IgA}, is a technique to perform numerical simulations on complicated geometries. As opposed to the Finite Elements Method (FEM), where the domain is approximated, usually by simplices, to allow the construction of numerical solutions in spaces of piecewise polynomials, the IgA approach gives priority to the geometric description of the problem. The numerical solution is then constituted by means of the functions used for the domain modeling. Nowadays, challenging geometries are usually expressed in terms of Computer Aided Design (CAD) shape functions, as B-splines, Non-Uniform Rational B-Splines (NURBS) and their generalization to address adaptive refineability \cite{HBsplines,Tsplines,PHTsplines,THBsplines,beirao2013,paperLR,PBsplines}.

The IgA approach has several advantages over FEM. Firstly, the geometry of the problem is exactly represented regardless of the fineness of the discretization. Secondly, while the numerical solution in FEM has $C^0$ global continuity, in IgA the global smoothness of the approximation can be controlled by choosing meshline multiplicities and degrees. As a consequence, IgA improves the stability of the approximation (fewer nonphysical oscillations \cite{IgA,oscillation}) and often it reaches a required accuracy using a much smaller number of degrees of freedom ($\sim 90\%$ less, \cite{grossmann,espen}).
Furthermore, the existing bottleneck in engineering is the creation of analysis suitable models from geometric models in CAD. This involves many steps that lead to the final domain approximation for FEM. This conversion is estimated to consume more than 80\% of the overall analysis time. By adopting the IgA approach, all these conversion steps are not needed anymore and the model development is much faster and simplified.

In this paper we describe an IgA discretization for electromagnetics. For such physical phenomena, the reproduction of essential topological and homological structures at the discrete level is required to ensure accuracy and stability of the numerical method (see, e.g., \cite{Arnold1,Arnold2,Boffi,Hiptmair}). More precisely, we are interested in preserving the cohomological structure of the following de Rham complex:
\begin{equation}\label{continuousderham}
0 \xrightarrow{\id} H^1(\Omega^{pol}) \xrightarrow{\grad} H(\curl;\Omega^{pol}) \xrightarrow{\curl} H(\div;\Omega^{pol}) \xrightarrow{\div} L^2(\Omega^{pol}) \xrightarrow{0} 0,
\end{equation}
where $$
\begin{array}{l}
H(\curl;\Omega^{pol}) = \{ \pmb{f} \in (L^2(\Omega^{pol}))^3 ~|~ \curl\pmb{f}\in (L^2(\Omega^{pol}))^3\},\\\\
H(\div; \Omega^{pol}) = \{\pmb{f} \in (L^2(\Omega^{pol}))^3 ~|~ \div\pmb{f} \in L^2(\Omega^{pol})\},
\end{array}
$$
and $\Omega^{pol}$ is a solid toroidal domain parametrized by a polar map. Such a map collapses a face of the parametric domain $\Omega$ to a closed polar curve running around the handle of the toroidal domain. 
It is well known, see e.g. \cite{hatcher}, that the dimension of the $k$th cohomology $\cH_{\dR}^k$ of the de Rham complex \eqref{continuousderham} is equal to the $k$th Betti's number $b_k$ of the domain. In our case, we have: \begin{itemize}
\item $\dim \cH_{\dR}^0=\dim \ker (\grad) = b_0 = 1$, the number of connected components forming $\Omega^{pol}$,
\item $\dim \cH_{\dR}^1= \dim (\ker(\curl)/\im(\grad)) = b_1= 1$, the number of handles in $\Omega^{pol}$,
\item $\dim \cH_{\dR}^2=\dim (\ker(\div) /\im(\curl)) =b_2 = 0$, the number of cavities enclosed in $\Omega^{pol}$,
\item $\dim \cH_{\dR}^3 = \dim (L^2(\Omega^{pol})/\im(\div)) = b_3 = 0$, by definition for any 3D domain. 
\end{itemize}
The purpose of this work is to discretize the de Rham complex \eqref{continuousderham} in a spline complex on $\Omega^{pol}$ preserving the above cohomology dimensions. 
The spline spaces involved in such discretization will be obtained by pushforward operators from spline spaces on the parametric domain $\Omega$. The singularity introduced by the polar map makes the standard tensor product spline spaces not suitable for this operation. In fact, the pushforward of the B-spline basis functions would be multivalued at the polar curve. As a consequence, the obtained discrete spaces on $\Omega^{pol}$ would not be subspaces of the continuous counterparts in the de Rham complex \eqref{continuousderham}. For this reason, we will define extraction operators leading to restricted spaces, in the tensor product spline spaces on $\Omega$, for which the pushforward operators can be applied to obtain appropriate approximant spaces of $H^1(\Omega^{pol}), H(\curl;\Omega^{pol}),H(\div;\Omega^{pol})$ and $L^2(\Omega^{pol})$ forming, moreover, a spline complex preserving the cohomological structure of \eqref{continuousderham}.

We are interested in toroidal domains because these geometries are of particular relevance in the problem of controlled thermonuclear reactions, especially in fusion reactors based on magnetic confinement. In these devices, plasma is confined by the magnetic fields created by a fixed set of external current-carrying conductors located at the boundary walls. In order to prevent dramatic losses of energy and plasma density, the magnetic field lines should not intersect such conductors or even the vacuum chamber located between plasma and material walls. The study of ideal magnetic hydrodynamics equilibria leads to relatively complicated toroidal geometries, such as tokamaks and stellarators, see Figure \ref{tokamakstellaratorelmo}, which can be employed to confine and isolate plasma from material walls and keep it stable for relatively large values of the ratio energy produced/energy used. Further details can be found, e.g., in the books \cite{freidberg} and \cite{hazeltine}. 
\begin{figure}
\centering
\subfloat[tokamak and a cross-section]{
\includegraphics[trim=270 180 300 0, clip, width=.28\textwidth]{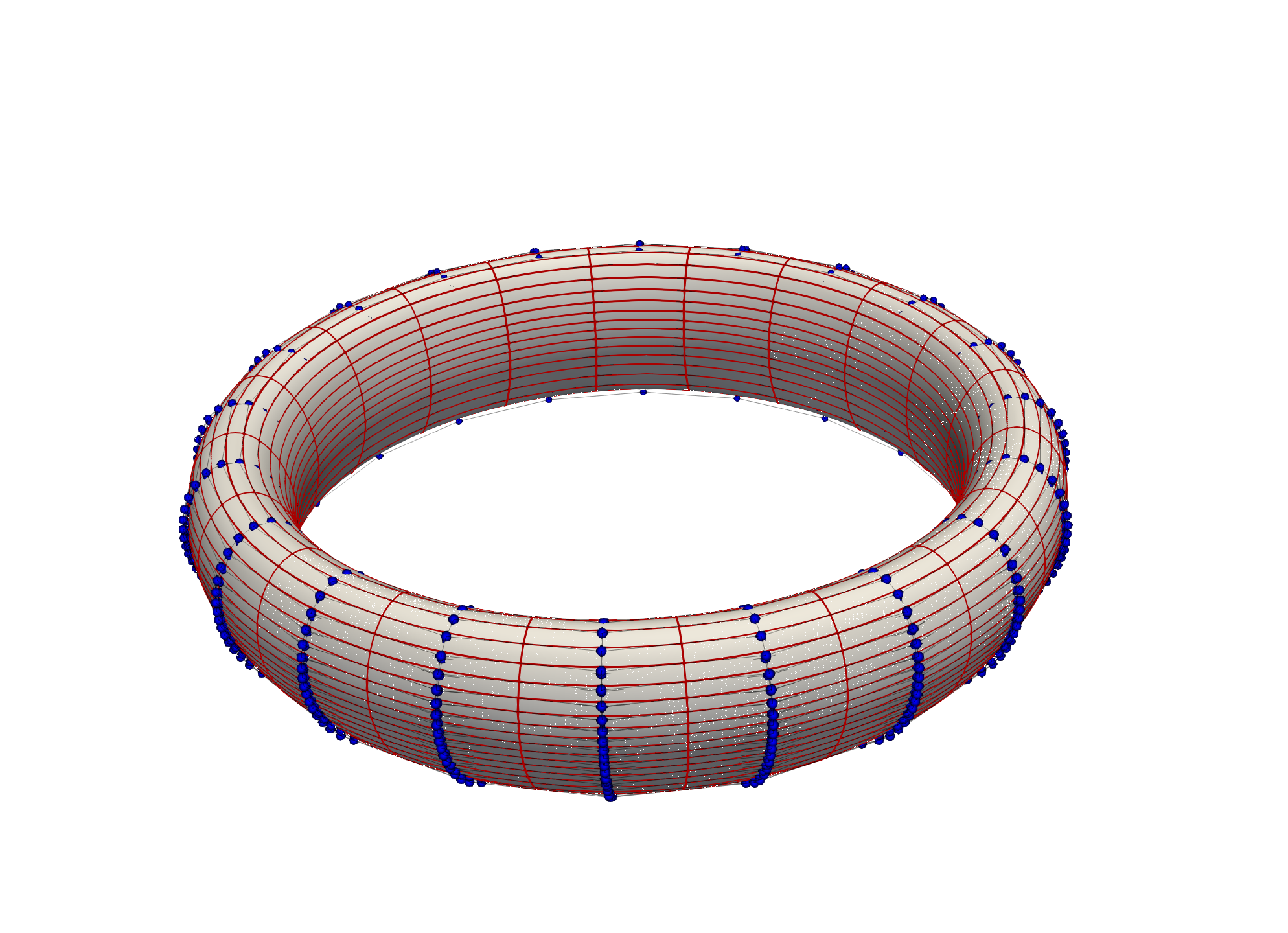}
\includegraphics[trim=500 0 500 0, clip, width=.2\textwidth]{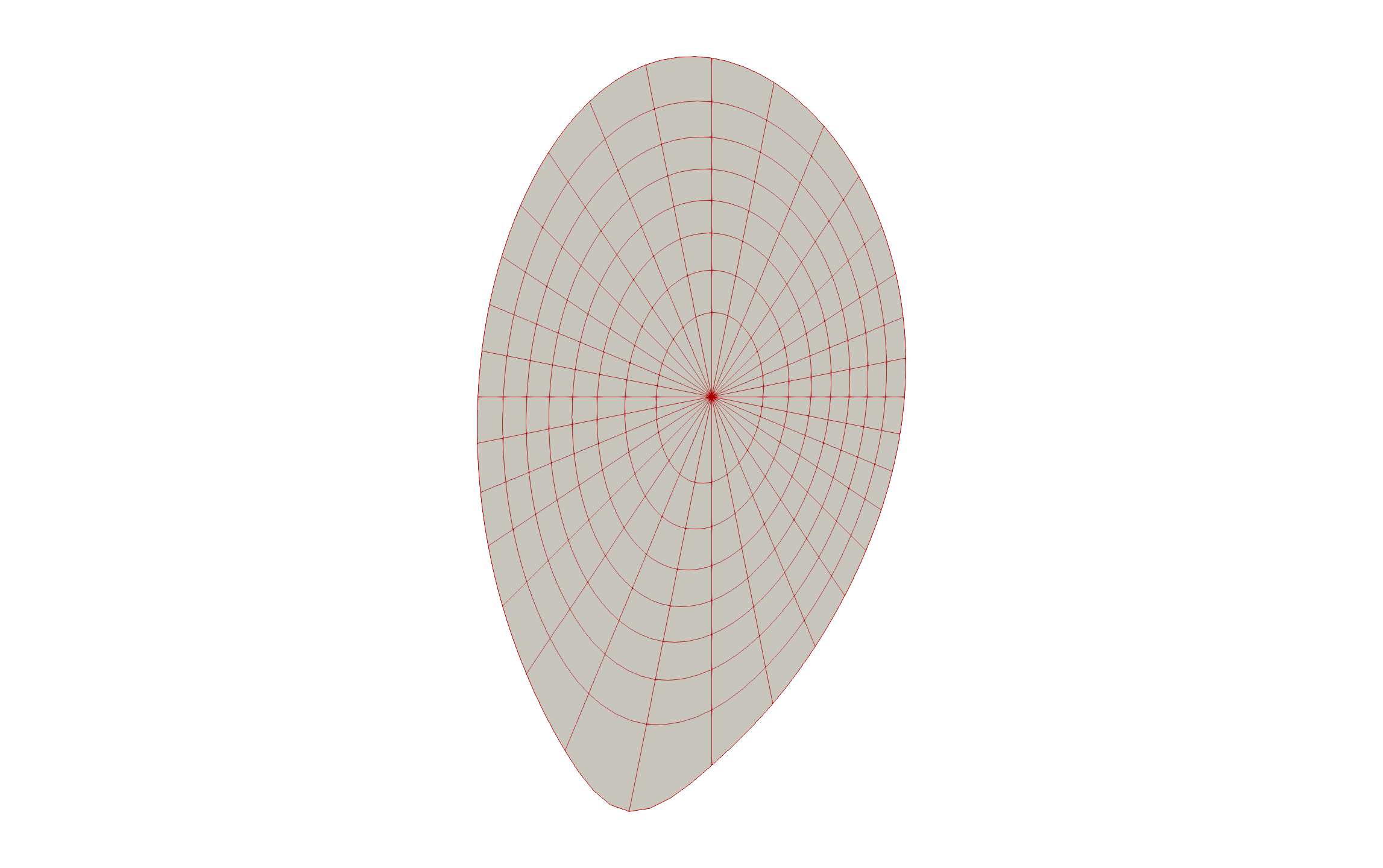}
}
\subfloat[stellarator and a cross-section]{
\includegraphics[trim=300 180 270 0, clip, width=.28\textwidth]{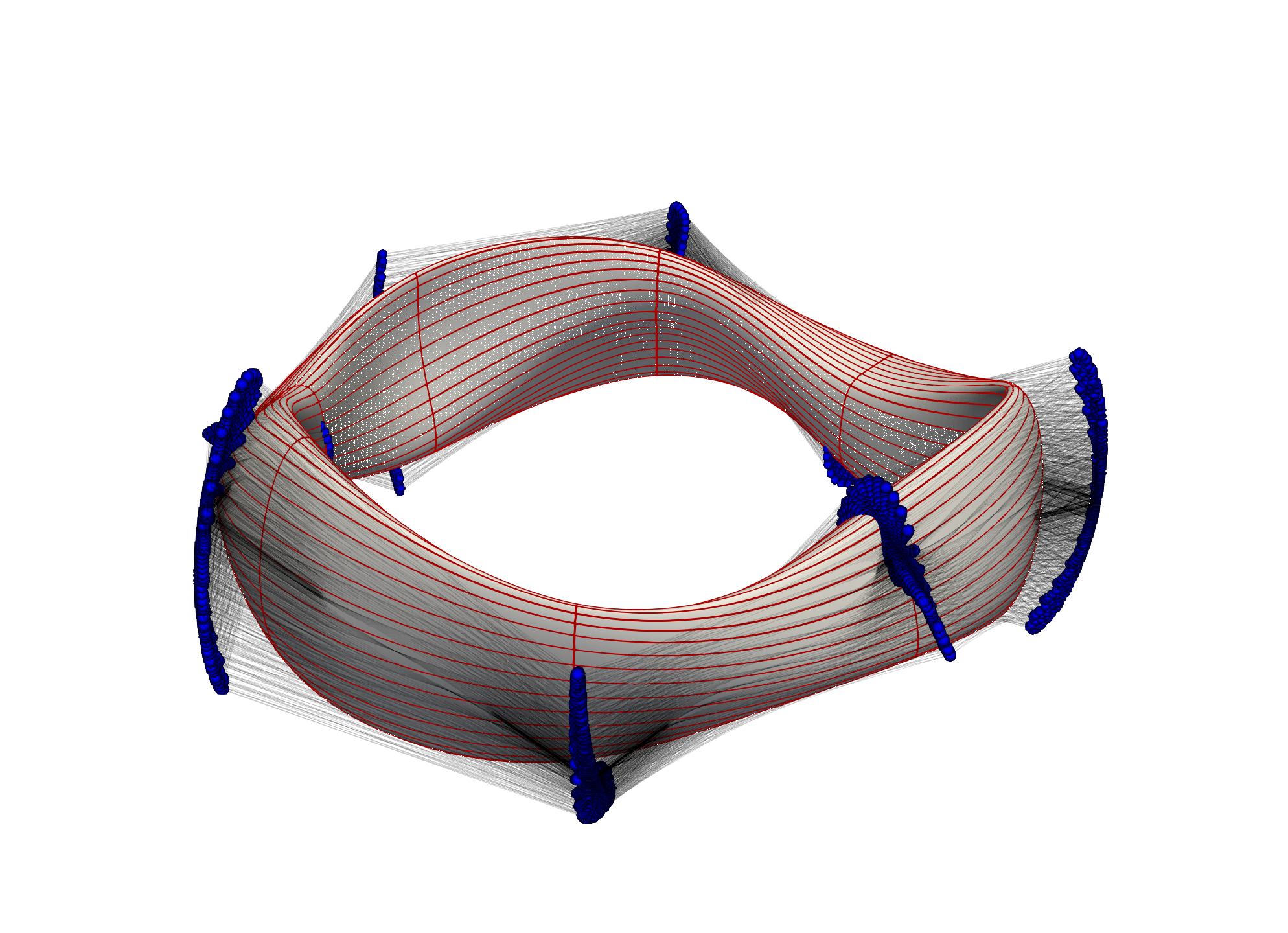}
\includegraphics[trim=600 0 600 0, clip, width=.2\textwidth]{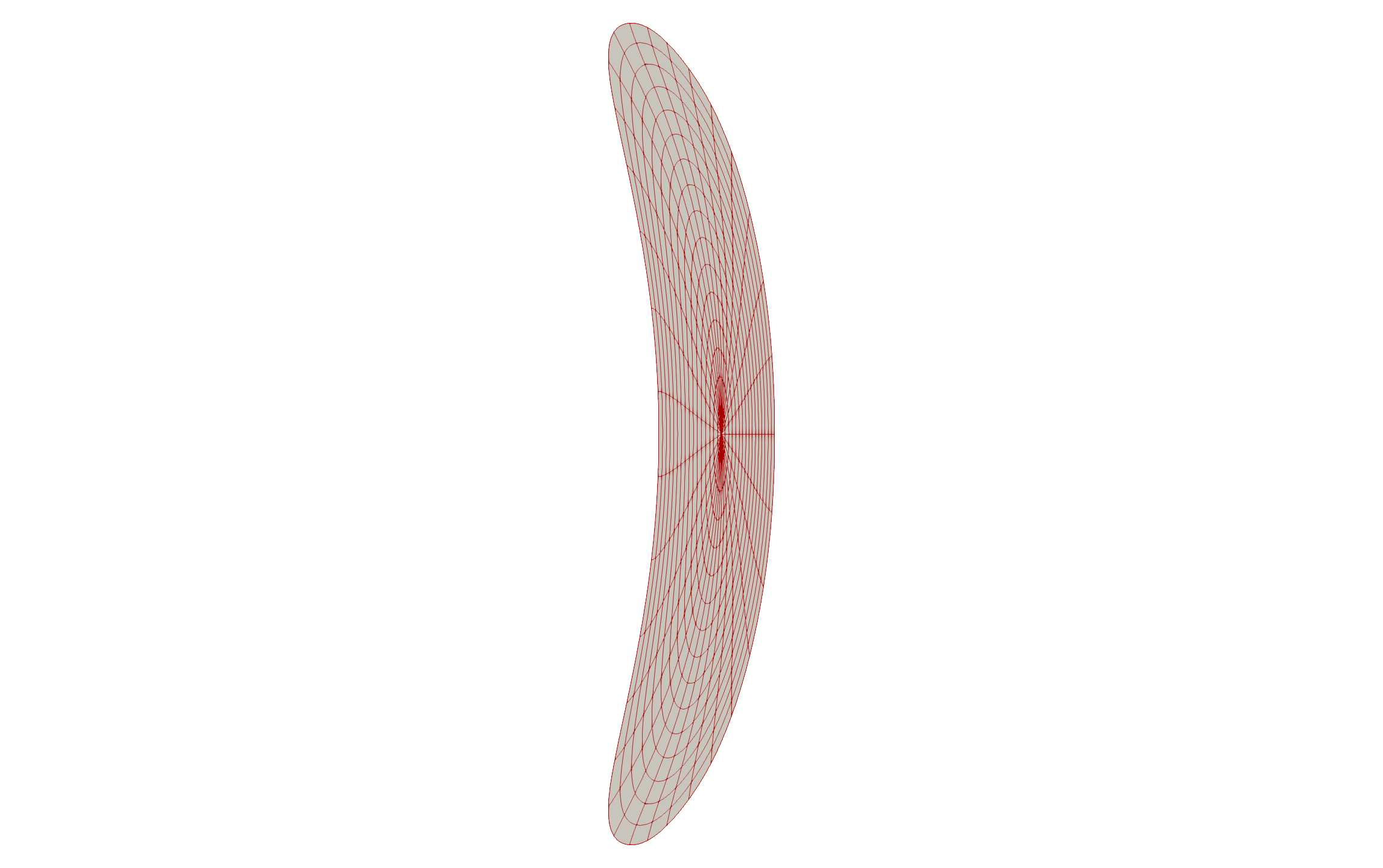}
}
\caption{$C^1$ smooth volumetric geometric models of a tokamak (a) and a stellarator (b) created using polar splines. The blue dots are the control points of the mappings, obtained via interpolation at the Greville points of such magnetic hydrodynamics equilibria using the Python libraries Struphy \cite{struphy} and Psydac \cite{psydac}.}\label{tokamakstellaratorelmo}
\end{figure}
Numerical methods for the simulation of magnetically confined plasma have been studied and developed since the early '60s. Nevertheless, only recently spline spaces for grid-based discretizations have been adopted, see, e.g., \cite{back,gempic,holderied,yaman,pinto}. Such numerical models rely on the approximations of electromagnetic fields in the IgA framework built using tensor-product splines \cite{Buffa1,Buffa2} and, more recently, adaptively-refined splines \cite{Buffa3,johannessen,evans,thbderham}. However, the establishing of plasma simulations within physical domains parametrized by means of polar maps still constitute a major difficulty, due to the singularity introduced at the polar curves. In \cite{yaman,possanner}  models providing $C^1$ continuous approximation are proposed on a 2D disk-like domain using the so called \emph{polar splines} \cite{MDB1,polarC1}. Polar splines have also been used for the  discretization of the 2D de Rham complex on disk-like domains and more in genereal on bivariate surfaces of genus 0 \cite{deepesh}. This work can indeed be reckon as an extension of \cite{deepesh} to a 3D solid toroidal domain. However, as opposed to it, we have chosen to use the standard functional analysis terminology instead of introducing the more abstract, but perhaps more elegant, formalism of the finite element exterior calculus, for the sake of briefness and concreteness.

The next sections are organized as follows. In Section \ref{preliminariesanddiagram} we first introduce the notations used in this paper and recall some basic concepts of the spline theory. Then we set the framework for our construction and we provide an overview of the desired discretization by means of a diagram. In Section \ref{reducedspaces} we define the restricted spline spaces employed in the parametric domain for the construction of the spline complex on the physical domain. In Section \ref{preservation} we prove the preservation of the cohomological structure of \eqref{continuousderham} at the discrete level and in Section \ref{conclusion} we draw the conclusions. 

We end this introductory part with some notation remarks. In what follows, vectorial objects related to physical quantities such as points in the space $\RR^3$ or vector-valued functions will be written in bold italic letters. All the other unphysical vectorial entities, such as ordered sets, multi-degrees or sequences of parameter values, will be denoted with bold non-italic letters. $\bbol{I}_n$ will denote the identity matrix of size $n$. Finally, the indexing related to periodic quantities is cyclic, that is, if $i \in \{1, \ldots, n\}$ is related to periodic entities, then whenever we encounter $i = n + 1$ we mean $i = 1$ and, similarly, if $i = 0$ we mean $i = n$.

\section{Preliminaries \& discretization overview}\label{preliminariesanddiagram}
In this section we present the framework and tools we need for the discretization. We further picture the process in a diagram. In this paper we have chosen to set it up with spline spaces generated by B-splines for the sake of simplicity and briefness. However, an analogous construction could be done starting from, e.g., \emph{multi-degree spline spaces} (see \cite{MDB1,MDB2,MDB3,polarC1}), i.e., collections of NURBS spaces (of different degrees and weights) glued together with some prescribed smoothness. 

We assume the reader to be familiar with the definition and main properties of B-splines. An introduction to this topic can be found, e.g., in the review papers \cite{lychemanni,cetraro} or in the classical books \cite{deboor} and \cite{schumaker}. What follows only introduces the notations adopted in this work.

\subsection{Preliminaries}\label{preliminaries}
Given a degree $p\geq 0$ and an open knot vector $\m{t}= (t_1, \ldots, t_{n+p+1}) \subset \RR$, we indicate as $B_j^p$ the \textbf{$j$th B-spline of degree $p$} defined on $\m{t}$. The \textbf{spline space generated by all the B-splines of degree $p$ on $\m{t}$} will be denoted as $\SSS^p_{\m{t}} =\spn\{B_j^p\}_{j=1}^n$. 
A spline space $\SSS_{\m{t}}^p$ \textbf{is $C^k$}, for $k \geq 0$, if its elements have global smoothness greater or equal to $C^k$ in the interval spanned by the knot vector. Given a $C^k$ spline space $\SSS_{\m{t}}^p$, we define its \textbf{periodic subspace}, $\mathring{\SSS}_{\m{t}}^p$, as the space of splines $f \in \SSS^p_{\m{t}}$ which further satisfy
$$
f^{(j)}(t_1) = f^{(j)}(t_{n+p+1}) \quad \text{for }j=0, \ldots, k.
$$
For $C^1$ spaces, if $\m{B}^{(0)}$ is the vector containing the B-splines spanning $\SSS_{\m{t}}^p$, then $\mathring{\SSS}_{\m{t}}^p$ is spanned by the functions contained in the vector 
\begin{equation}\label{periodicmatrix}
\per{\m{B}}{}^{(0)} := \bbol{H}^{(0)}\m{B}^{(0)}
\end{equation}
where $\bbol{H}^{(0)}$ is the rectangular matrix of size $(n-2)\times n$ defined as $\bbol{H}^{(0)}=[\m{c}|\bbol{I}_{n-2}|\m{c}]$ with $$
\m{c} = \left[\begin{array}{c}c_1\\0\\\vdots\\0\\c_2\end{array}\right] \quad\text{for }(c_1,c_2) = \frac{(t_{n+p+1}-t_n,t_{p+2}-t_1)}{t_{n+p+1}-t_n + t_{p+2} -t_1}.
$$
The construction and understanding of matrix $\bbol{H}^{(0)}$ is detailed in the setting of multi-degree splines in \cite{polarC1}. A more general assemble algorithm, for $C^k$ periodicity with any $k\geq 0$, is illustrated in \cite{MDB1,MDB3}. Figure \ref{periodicsplinebasis} visually compares the B-spline basis of $\SSS_{\m{t}}^p$ with the basis functions of $\mathring{\SSS}_{\m{t}}^p$, for $p=2$ and $\m{t}$ an uniform open knot vector with five distinct knots. 
\begin{figure}
\centering
\subfloat[]{
\includegraphics[width=.45\textwidth]{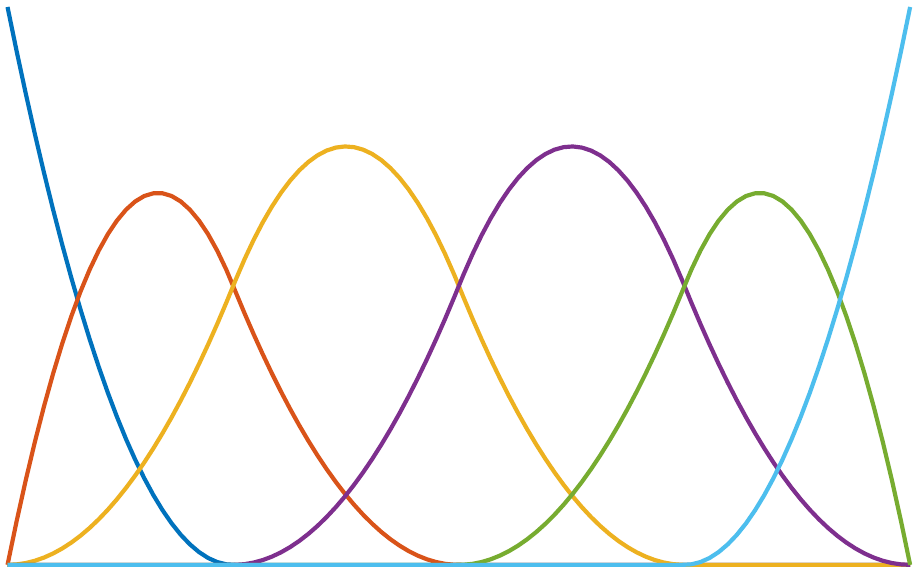}

}\quad
\subfloat[]{
\includegraphics[width=.45\textwidth]{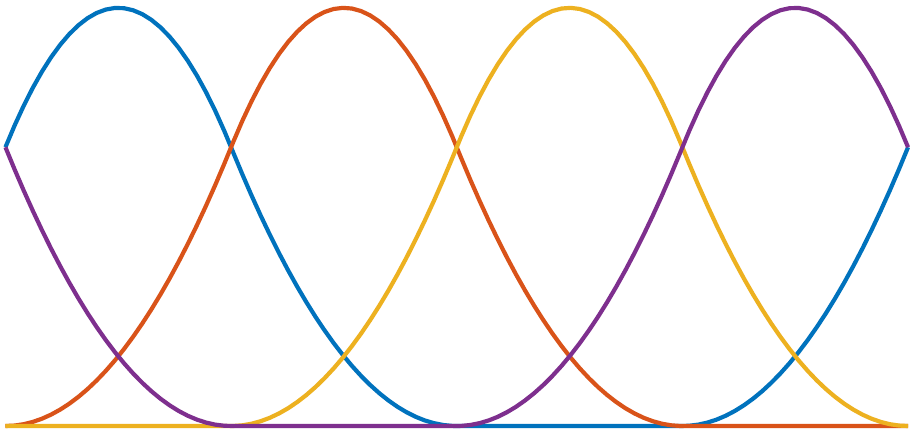}
}
\caption{The B-spline basis of the quadratic $C^1$ spline space (a) and the basis functions of the quadratic $C^1$ periodic spline space (b) on the uniform open knot vector of five distinct knots $\m{t}=[0,0,0,1,2,3,4,4,4]$. We stress that the basis functions in (b) tie in at the endpoints of the interval with $C^1$ continuity.}\label{periodicsplinebasis}
\end{figure}
Note that the first and last basis functions in $\mathring{\m{B}}{}^{(0)}$ are not B-splines. However they preserve all the B-spline properties when identifying the endpoints of the interval. Therefore we will not stress this fact again and, with an abuse of terminology for the sake of simplicity, we will call them B-splines as well and use the same notation.

A useful relation we will make large use of is the \textbf{derivative formula for splines}. Let $\SSS_{\m{t}}^p = \spn\{B_j^p\}_{j=1}^n$ be a $C^1$ spline space. Given $f \in \SSS_{\m{t}}^p$, $$
f(t) = \sum_{j=1}^n f_j B_j^p(t)
$$
we have that
\begin{equation}\label{splinederivative}
f'(t) = \sum_{j=1}^{n-1} (f_{j+1}-f_j) D_j^p(t) \quad \text{where } D_j^p(t) = \frac{p}{t_{j+p+1}-t_{j+1}}B_j^{p-1}(t),
\end{equation}
with $B_j^{p-1}$ defined on the knot vector $\hat{\m{t}}=(\hat{t}_1, \ldots, \hat{t}_{n+p-1}) = (t_2, \ldots, t_{n+p})\subset \m{t}$ for all $j$. Let again $\m{B}^{(0)}$ be the vector of the B-spline basis $\{B_j^p\}_{j=1}^n$. If $\m{f}$ is the vector of coefficients $(f_1, \ldots, f_n)^T$, then $
f = \m{B}^{(0)} \cdot \m{f}
$
and the derivative formula \eqref{splinederivative} can be written as a matrix expression:
\begin{equation}\label{splinederivativematrix}
f'(t) = \m{B}^{(1)} \cdot \bbol{D}_n\m{f}  
\end{equation}
where $\m{B}^{(1)}$ is the vector of functions $\{D_j^p\}_{j=1}^{n-1}$ and $\bbol{D}_n$ the following rectangular matrix of size $(n-1)\times n$
\begin{equation}\label{Dn}
\bbol{D}_n = \left[\begin{array}{ccccc}-1 & 1 & & \\  & \ddots & \ddots&&\raisebox{.25cm}{\makebox(0,0){\scalebox{3}{$\circ\quad$}}}\\ & & \ddots & \ddots\\\raisebox{.5cm}{\makebox(0,0){\scalebox{3}{$~~\circ$}}}&& & -1 & 1\end{array}\right].
\end{equation}
The span of the functions in $\m{B}^{(1)}$ is the $C^0$ spline space $\SSS_{\hat{\m{t}}}^{p-1}$ (see, e.g., \cite[Theorems 6--7]{lychemanni}). 

In the periodic spline space a derivative formula similar to Equation \eqref{splinederivativematrix} holds true. Let $\mathring{\SSS}_{\m{t}}^p$ be the $C^1$ periodic spline space spanned by the functions in the vector $\mathring{\m{B}}{}^{(0)}$. Then a spline $\mathring{f} \in \per{\SSS}_{\m{t}}^p$ can be expressed as $
\per{f} = \per{\m{B}}{}^{(0)} \cdot \per{\m{f}}
$
for a vector of coefficients $\per{\m{f}}$. The derivative of $\per{f}$ is 
\begin{equation*}
\per{f}'(t) = \per{\m{B}}{}^{(1)} \cdot \per{\bbol{D}}_{n-2}\per{\m{f}}  
\end{equation*}
where $\per{\bbol{D}}_{n-2}$ is the square matrix of size $n-2$
\begin{equation}\label{Dnperiodic}
\per{\bbol{D}}_{n-2} = \left[\begin{array}{ccccc}-1 & 1 & & \\  & \ddots & \ddots&&\raisebox{.25cm}{\makebox(0,0){\scalebox{3}{$\circ\quad$}}}\\ & & \ddots & \ddots\\\raisebox{.5cm}{\makebox(0,0){\scalebox{3}{$~~\circ$}}}&& & -1 & 1\\1 & 0 & \cdots &0& -1\end{array}\right],
\end{equation}
and $\per{\m{B}}{}^{(1)}$ is the vector of splines $\{\per{D}_{j}^p\}_{j=1}^{n-2}$ spanning  the $C^0$ periodic space $\per{\SSS}_{\hat{\m{t}}}^{p-1}$ given by
\begin{equation}\label{periodicderivativebasis}
\per{\m{B}}{}^{(1)} = \bbol{H}^{(1)}\m{B}^{(1)}
\end{equation}
where $\bbol{H}^{(1)}$ is the following rectangular matrix of size $(n-2)\times (n-1)$: $$\bbol{H}^{(1)}= \left[
\begin{minipage}{.19\textwidth}
\begin{tikzpicture}[scale=3]
\draw (.2,0) -- (.2,1);
\draw (.8,0) -- (.8,1);
\draw (0,.2) -- (1,.2);
\node at (.1,.6) {$\begin{array}{c} 0 \\ \vdots \\ \vdots \\ 0 \end{array}$};
\node at (.9,.6) {$\begin{array}{c} 0 \\ \vdots \\ \vdots \\ 0 \end{array}$};
\node at (.1,.1) {$c_2$};
\node at (.925,.1) {$c_1$};
\node at (.5,.1) {$0~\cdots~0$};
\node at (.5,.6) {\scalebox{1.5}{$\bbol{I}_{n-3}$}};
\end{tikzpicture}
\end{minipage}
\right].
$$
\begin{figure}
\centering
\subfloat[]{
\includegraphics[width=.45\textwidth]{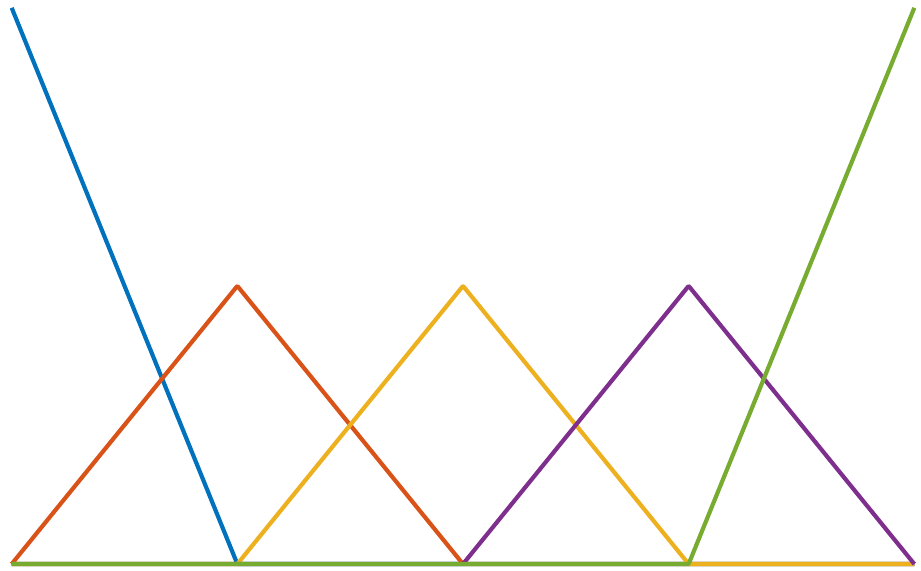}

}\quad
\subfloat[]{
\includegraphics[width=.45\textwidth]{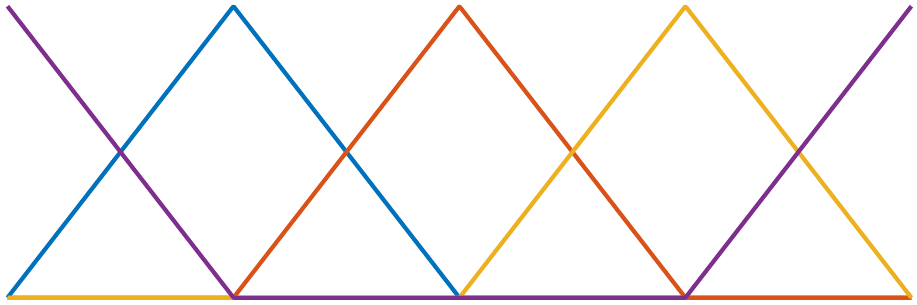}
}
\caption{Comparison of the linear spline functions spanning the derivatives of the quadratic splines generated by the bases reported in Figure \ref{periodicsplinebasis} (a)-(b) respectively. These derivatives are $C^0$ (a) and $C^0$ periodic (b) respectively as the linear splines in (b) tie in at the endpoints of the interval with $C^0$ continuity.}\label{periodicderivatives}
\end{figure}
In Figure \ref{periodicderivatives} we visually compare the functions in $\m{B}^{(1)}$ and $\per{\m{B}}{}^{(1)}$ when $\m{B}^{(0)}$ and $\per{\m{B}}^{(0)}$ contain the splines in Figure \ref{periodicsplinebasis} (a)-(b) respectively. 

Another notion we need is the definition of \textbf{design-through-analysis compatible matrix}, introduced in \cite{polarC1}. Consider a set of B-spline like functions, that is, a set of functions that are linearly independent, locally supported and that form a positive partition of unity. When a design-through-analysis compatible, or DTA-compatible, matrix is applied to such a collection, it provides another set of B-spline like functions. More precisely, a matrix $\bbol{E}$ is DTA-compatible, if
\begin{itemize}
\item $\bbol{E}$ has full rank (preservation of the linear independence),
\item each column of $\bbol{E}$ sums to 1 (preservation of the partition of unity),
\item each entry of $\bbol{E}$ is non-negative (preservation of the positivity),
\item $\bbol{E}$ preserves the locality of the supports through sparsity.
\end{itemize}
We have already seen an example of such matrices (aside the identity matrix). Matrix $\bbol{H}^{(0)}$ of Equation \eqref{periodicmatrix}, which transforms the B-spline basis of $\SSS_{\m{t}}^p$ in the basis of $\per{\SSS}_{\m{t}}^p$ is a DTA-compatible matrix.

Finally, we make the following simple remark which, nevertheless, will be fundamental to prove commutation in the discretization diagram.
\begin{oss}\label{linearalgebraremark}
Let $A$ be a functional vector space spanned by $\{\alpha_j\}_{j=1}^n$ and let $B\subseteq A$ be a subspace spanned by $\{\beta_i\}_{i=1}^m$. If the vectors $\pmb{\upalpha}$ and $\pmb{\upbeta}$ contain the bases of $A$ and $B$ respectively, let $\bbol{M}$ be the restriction matrix, $\pmb{\upbeta} = \bbol{M}\pmb{\upalpha}$. Then, given $f \in B$, $f = \pmb{\upbeta} \cdot \m{b}$ with $\m{b} \in \RR^m$ its vector of coefficients, we can always write $f$ in terms of the basis functions of $A$, $f = \pmb{\upalpha} \cdot \m{a}$ with $\m{a} \in \RR^n$ the vector of coefficients given by $\m{a} = \bbol{M}^T\m{b}$.
\end{oss}

\subsection{Problem setting \& discretization overview}\label{ssec:discretizationoverview}
We now present the process leading to a spline sub-complex of the de Rham complex \eqref{continuousderham} preserving the cohomology dimensions, in a solid toroidal domain $\Omega^{pol}$. 

For the sake of simplicity, we drop the subscripts in the spline spaces we are going to define hereafter, as their specific knot vectors will be not relevant anymore.

Given a multidegree $\m{p}=(p^r,p^\theta,p^\phi)$, let  \begin{itemize}
\item $\SSS^{p^r}$ be a $C^1$ spline space defined on the interval $[0, R]$ spanned by the B-splines $\{B_i^{p^r}\}_{i=1}^{n^r}$,
\item $\per{\SSS}^{p^\theta}$ be a $C^1$ periodic spline space defined on the interval $[0, 2\pi]$ spanned by the B-splines $\{\per{B}_j^{p^\theta}\}_{j=1}^{n^\theta}$,
\item $\per{\SSS}^{p^\phi}$ be a $C^1$ periodic spline space defined on the interval $[0, 2\pi]$ spanned by the B-splines $\{\per{B}_k^{p^\varphi}\}_{k=1}^{n^\varphi}$.
\end{itemize}
We define the \textbf{parametric domain} $\Omega$ as the Cartesian product of the intervals, $\Omega := [0, R]\times[0, 2\pi]\times[0, 2\pi]$ and \textbf{the spline space in} $\Omega$, $\SSS^{p^r,p^\theta,p^\varphi}$, as the tensor product of the considered univariate spline spaces on such intervals:
$$
\SSS^{p^r,p^\theta,p^\phi} := \SSS^{p^r}\otimes \per{\SSS}^{p^\theta} \otimes \per{\SSS}^{p^\phi} = \spn\{B_{ijk}^{\m{p}}:=B_i^{p^r}\per{B}_j^{p^\theta}\per{B}_k^{p^\phi}\}_{i,j,k=1}^{n^r,n^\theta,n^\varphi}.
$$
Let $\pmb{F}: \Omega\to \Omega^{pol}$ be \textbf{the polar map},
$$
\Omega \ni (r,\theta,\phi) \mapsto \pmb{F}(r,\theta,\phi) = (F_u(r,\theta,\phi),F_v(r,\theta,\phi), F_w(r,\theta,\phi)) = (u,v,w) \in \Omega^{pol},
$$
which transforms $\Omega$ into the toroidal domain $\Omega^{pol}$, see Figure \ref{polarmap}.
\begin{figure}
\centering
\begin{tikzpicture}[scale=.75]
\foreach \x in{0,...,4}
{\draw (0,\x ,4) -- (4,\x ,4);
  \draw (\x ,0,4) -- (\x ,4,4);
}
\draw[blue,ultra thick] (0,0,4) -- (4,0,4);
\draw[ultra thick,red] (0,0,4) -- (0,4,4);
\draw[ultra thick] (4,0,4) -- (4,4,4);
\draw[ultra thick,blue] (4,4,4) -- (0,4,4);
\draw[ultra thick] (4,0,0) -- (4,4,0);
\draw[ultra thick,blue] (4,4,0) -- (0,4,0);
\foreach \x in{0,1,2,3,4}
{\draw[dashed] (4,\x ,2) -- (4,\x ,1);
  \draw[dashed] (\x ,4,2) -- (\x ,4,1);
  \draw (4,\x ,4) -- (4,\x ,2);
  \draw (\x ,4,4) -- (\x ,4,2);
  \draw (4,\x ,1) -- (4,\x ,0);
  \draw (\x ,4,0) -- (\x ,4,1);
  \draw (4,0,\x ) -- (4,4,\x );
  \draw (0,4,\x ) -- (4,4,\x );
}
\draw[-stealth] (4.75,3) --node[below]{$r$} (5.75,3);
\draw[-stealth] (4.75,3) --node[left]{$\theta$} (4.75,4);
\draw[-stealth](4.75,3) -- (5.25,3.5) node[above]{$\phi$};
\node at (3.75,-1) {$\Omega$};
\draw[->] (4.75,1.5) -- node[above]{$\pmb{F}$} (6.25,1.5);
\node at (10.35,1.5) {\includegraphics{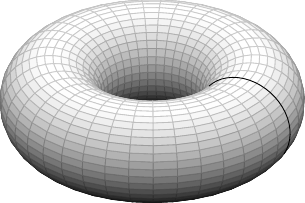}};
\node at (8,-1) {$\Omega^{pol}$};
\node at (17.5,1.5) {\includegraphics{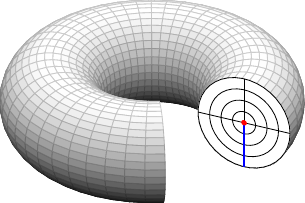}};
\node at (15.15,-1) {$\Omega^{pol}$};
\end{tikzpicture}
\caption{The transformation of the parametric domain $\Omega$ into the toroidal domain $\Omega^{pol}$ by the polar map $\pmb{F}$. The three-dimensional tensor mesh defined on $\Omega$ is mapped to a three-dimensional polar mesh in $\Omega^{pol}$. The latter can be seen in the right-most figure.}\label{polarmap}
\end{figure}
We assume $\pmb{F}$ in the spline space $\SSS^{p^r,p^\theta,p^\phi}$, that is,
\begin{equation*}
\pmb{F}(r,\theta,\phi) := \sum_{k=1}^{n^\phi}\sum_{j=1}^{n^\theta} \sum_{i=1}^{n^r} \pmb{F}_{ijk} B_{ijk}^{\m{p}}(r,\theta,\phi)
\end{equation*} 
with $\pmb{F}_{ijk} = \left((\bar{\rho} + \rho_i\cos\theta_j)\cos\varphi_k,(\bar{\rho} + \rho_i\cos\theta_j)\sin\varphi_k,\rho_i \sin\theta_j\right)\in \RR^3$, for $\bar{\rho}>2$ a fixed number and
$$
\begin{array}{ll}
\rho_i = \frac{i-1}{n^r-1} \in [0,1]&\text{for }i=1, \ldots, n^r,\\\\
\theta_j = \frac{j-1}{n^\theta}2\pi \in [0,2\pi)&\text{for }j=1, \ldots, n^\theta,\\\\
\varphi_k = \frac{k-1}{n^\phi}2\pi \in [0,2\pi) &\text{for }k=1, \ldots, n^\phi.
\end{array}
$$
With reference to Figure \ref{polarmap}, the map $\pmb{F}$ identifies in $\Omega^{pol}$ the bottom side with the top side and the front side with the back side of $\Omega$. The left side is instead collapsed to a \textbf{polar curve} running around the handle of $\Omega^{pol}$. Let us introduce the following \textbf{radial, toroidal and poloidal directions} in $\Omega^{pol}$:
$$
\begin{tikzpicture}[scale=1.25]
\fill (0,0) circle (.05);
\draw (0,0) circle (1);
\draw (4,0) circle (1);
\draw (3,0) arc(0:180:1 and .5);
\draw (5,0) arc(0:180:3 and 2);
\draw[-stealth] (0,0) --node[below]{radial~~} (1,0);
\draw[-stealth] (4.5,-1) arc(-66.5:0:1);
\node at (5.65,-.7) {poloidal};
\draw[-stealth] (5.1,.3) arc(0:45:2.2 and 1.5);
\node at (5.65, 1) {toroidal};
\end{tikzpicture}
$$
with the radial direction generating from the polar curve.
We further define the \textbf{pushforward B-splines} as the images via $\pmb{F}$ of the B-splines spanning $\SSS^{p^r,p^\theta,p^\phi}$, i.e., the collection of functions $B_{ijk}^{\m{p},pol}: \Omega^{pol}\to \RR$ defined as $$
B_{ijk}^{\m{p},pol} := B_{ijk}^{\m{p}} \circ \pmb{F}^{-1} \quad\forall\, i,j,k.
$$
Except at the polar curve, where $\pmb{F}^{-1}$ is not well-defined and the pushforward B-splines are multivalued, the $B_{ijk}^{\m{p},pol}$ are $C^1$-continuous in the radial direction. Such smoothness is drawn from the basis elements in $\SSS^{p^r}$. In the same way and thanks to the periodicity of the univariate spline spaces $\per{\SSS}^{p^\theta}$ and $\per{\SSS}^{p^\phi}$, the $C^1$-continuity of the pushforward B-splines is also guaranteed in the toroidal and poloidal directions, far from the polar curve. 
However, our strategy for the discretization of the de Rham complex \eqref{continuousderham} is to construct a space $V_0^{pol}\subseteq H^1(\Omega^{pol})$ of splines which are $C^1$-differentiable in $\Omega^{pol}$, in particular at the polar curve, and then look for spline spaces $V_1^{pol},V_2^{pol}, V_3^{pol}$, in $H(\curl;\Omega^{pol}), H(\div;\Omega^{pol})$ and $L^2(\Omega^{pol})$ respectively, such that the differential operators ``$\grad$'', ``$\curl$'' and ``$\div$'' are locally exact near the polar curve. In this manner, every element of $V_1^{pol}$ would be the gradient of an element in $V_0^{pol}$, every element of $V_2^{pol}$ would be the curl of an element in $V_1^{pol}$ and every element of $V_3^{pol}$ would be the divergence of an element in $V_2^{pol}$, near the polar curve in $\Omega^{pol}$. As a consequence, the $C^1$-smoothness in $V_0^{pol}$ would ensure the $C^0$-smoothness of the functions in $V_1^{pol}$ and $V_2^{pol}$ and the boundedness of the elements of $V_3^{pol}$, despite the global $C^{-1}$-continuity. For this reason, we shall extract from the spline space a suitable subspace where $\pmb{F}$ is invertible. 

All the spaces $V_0^{pol},V_1^{pol},V_2^{pol}$ and $V_3^{pol}$ will be obtained by pushforward operators from reduced spline spaces on the parametric domain $\Omega$. The full discretization process is schematized in the diagram of Figure \ref{discretization}. 
\begin{figure}
\centering
\scalebox{.85}{
\begin{tikzpicture}
\matrix (m)[matrix of math nodes,column sep=2em,row sep=3em]{ 
\SSS^{p^r,p^\theta,p^\phi} \pgfmatrixnextcell \SSS^{p^r-1,p^\theta,p^\phi}\times \SSS^{p^r,p^\theta-1,p^\phi}\times \SSS^{p^r,p^\theta,p^\phi-1} \pgfmatrixnextcell \SSS^{p^r,p^\theta-1,p^\phi-1}\times\SSS^{p^r-1,p^\theta,p^\phi-1}\times \SSS^{p^r-1,p^\theta-1,p^\phi} \pgfmatrixnextcell \SSS^{p^r-1,p^\theta-1,p^\phi-1}\\
V_0 \pgfmatrixnextcell V_1 \pgfmatrixnextcell V_2 \pgfmatrixnextcell V_3\\
V_0^{pol} \pgfmatrixnextcell V_1^{pol} \pgfmatrixnextcell V_2^{pol}\pgfmatrixnextcell V_3^{pol}\\
H^1(\Omega^{pol}) \pgfmatrixnextcell H(\curl;\Omega^{pol}) \pgfmatrixnextcell H(\div; \Omega^{pol})\pgfmatrixnextcell L^2(\Omega^{pol})\\};

\draw[-stealth] (m-4-1) --node[above]{$\grad$} (m-4-2);
\draw[-stealth] (m-3-1) --node[above]{$\grad$} (m-3-2);
\draw[-stealth] (m-2-1) --node[above]{$\grad$} (m-2-2);
\draw[-stealth] (m-1-1) --node[above]{$\grad$} (m-1-2);
\draw[-stealth] (m-1-2) --node[above]{$\curl$} (m-1-3);
\draw[-stealth] (m-2-2) --node[above]{$\curl$} (m-2-3);
\draw[-stealth] (m-3-2) --node[above]{$\curl$} (m-3-3);
\draw[-stealth] (m-4-2) --node[above]{$\curl$} (m-4-3);
\draw[-stealth] (m-4-3) --node[above]{$\div$} (m-4-4);
\draw[-stealth] (m-3-3) --node[above]{$\div$} (m-3-4);
\draw[-stealth] (m-2-3) --node[above]{$\div$} (m-2-4);
\draw[-stealth] (m-1-3) --node[above]{$\div$} (m-1-4);

\draw[-stealth] (m-2-1) --node[right]{$\text{id}$} (m-1-1);
\draw[-stealth] (m-2-2) --node[right]{$\text{id}$} (m-1-2);
\draw[-stealth] (m-2-3) --node[right]{$\text{id}$} (m-1-3);
\draw[-stealth] (m-2-4) --node[right]{$\text{id}$} (m-1-4);
\draw[-stealth] (m-2-1) --node[right]{$\cF^0$} (m-3-1);
\draw[-stealth] (m-2-2) --node[right]{$\cF^1$} (m-3-2);
\draw[-stealth] (m-2-3) --node[right]{$\cF^2$} (m-3-3);
\draw[-stealth] (m-2-4) --node[right]{$\cF^3$} (m-3-4);
\draw[-stealth] (m-3-1) --node[right]{$\text{id}$} (m-4-1);
\draw[-stealth] (m-3-2) --node[right]{$\text{id}$} (m-4-2);
\draw[-stealth] (m-3-3) --node[right]{$\text{id}$} (m-4-3);
\draw[-stealth] (m-3-4) --node[right]{$\text{id}$} (m-4-4);
\end{tikzpicture}
}
\caption{The scheme of the different complexes we consider. The second bottom row is the spline sub-complex on $\Omega^{pol}$ we use to discretize the de Rham complex \eqref{continuousderham} (bottom row) and which preserves the cohomological structure of the latter. The two top rows are complexes on the parametric domain $\Omega$. The top one is the spline complex of the full tensor product spline space. The second top is the complex of the reduced spline space in which the pushforward operators $\cF^i$ related to the polar map $\pmb{F}$ can be applied.}\label{discretization}
\end{figure} 
We start with the complex of the full spline space on the parametric domain $\Omega$ (top row). This complex has been analyzed in \cite{Buffa2} and it can be employed for the discretization of the de Rham complex on non-singular physical domain.
However, as we explained, for polar maps we need to define reduced spline spaces on $\Omega$ in which we can apply the corresponding pushforward. 
The definition of such reduced spaces is the content of Section \ref{reducedspaces}. The \textbf{pushforward operators} $\cF^0, \cF^1, \cF^2, \cF^3$ can then be employed to obtain our final discretization on $\Omega^{pol}$. These are defined as follows:
\begin{equation}\label{pushforward}
\begin{array}{l}
V_0 \ni f \mapsto \cF^0(f) :=f \circ \pmb{F}^{-1} =: f^{pol} \in V_0^{pol},\\\\
V_1 \ni \pmb{g} \mapsto \cF^1(\pmb{g}) := (D\pmb{F}^{-T}\pmb{g}) \circ \pmb{F}^{-1} =: \pmb{g}^{pol} \in V_1^{pol},\\\\
V_2 \ni \pmb{h} \mapsto \cF^2(\pmb{h}) := ((\det D\pmb{F})^{-1}D\pmb{F}\pmb{h}) \circ \pmb{F}^{-1} =: \pmb{h}^{pol} \in V_2^{pol},\\\\
V_3 \ni m \mapsto \cF^3(m) := ((\det D\pmb{F})^{-1}m) \circ \pmb{F}^{-1} =: m^{pol} \in V_3^{pol},
\end{array}
\end{equation}
where $D\pmb{F}$ is the Jacobian matrix of $\pmb{F}$. The space $V_0^{pol}$ in the definition of $\cF^0$ is referred to as \textbf{polar spline space} in the literature \cite{MDB1,polarC1,deepesh}.
We shall show that the discretization process is well-defined by proving commutation of the diagram and we shall verify that the spline complex on the physical domain at the second last row of the diagram preserves the cohomology dimensions of the continuous de Rham complex at the last row.

\begin{oss}
We have considered a specific polar map $\pmb{F}$. This choice has been done in the interest of standardization. Namely, the chosen domain $\Omega^{pol}$ is considered just to identify the right spline spaces to deal with the polar singularity. Given a different toroidal solid $\tilde{\Omega}^{pol}$, such those in Figure \ref{tokamakstellaratorelmo}, assume that it is image of a polar map $\pmb{G}: \Omega \to \tilde{\Omega}^{pol}$.  The spaces of the reduced spline complex we shall define on $\Omega$ can be safely pushed-forward on $\tilde{\Omega}^{pol}$ through $\pmb{G}$, instead that on $\Omega^{pol}$ through $\pmb{F}$, provided that $G $ can be written as $\pmb{G} = \bar{\pmb{G}} \circ \pmb{F}$ with $\bar{\pmb{G}}$ a diffeomorphism between the polar domains $\Omega^{pol}$ and $\tilde{\Omega}^{pol}$. This would preserve the smoothness of the splines in the spline complex on $\Omega^{pol}$ as $\bar{\pmb{G}}$ does not introduce new singularities. In other words, $\Omega^{pol}$ plays the role of a \textbf{polar parametric domain} rather than the physical domain in practice. It is an intermediate domain considered to remove the polar singularity hurdle. 
\end{oss}

\section{Extraction operators and reduced spaces in $\Omega$}\label{reducedspaces}
In this section we identify subspaces of the full tensor product spline spaces in the parametric domain, for which the pushforward operators can be applied in order to create appropriate approximant spaces for $H^1(\Omega^{pol}), H(\curl;\Omega^{pol}), H(\div;\Omega^{pol})$ and $L^2(\Omega^{pol})$. Extraction matrices will be defined to determine such reduced spline spaces. We further show the commutation of such extraction process with the differential operators $\grad, \curl$ and $\div$.

\subsection{The space $V_0$}\label{spaceV0}
Let $f \in \SSS^{p^r,p^\theta,p^\phi}$ and $f^{pol}:=\cF^0(f)$ be the corresponding pushforward spline on $\Omega^{pol}$, that is, $$
f^{pol}(\pmb{F}(r,\theta,\phi)) \defeq f(r,\theta,\phi).
$$
We have that $f^{pol} \in C^1(\Omega^{pol})$ if $f^{pol}$ is well-defined and $C^1$ in the radial, poloidal and toroidal directions everywhere in $\Omega^{pol}$. 
The polar map $\pmb{F}$ collapses the left side of $\Omega$, $\{(r,\theta,\phi) \in \Omega~|~ r = 0\}$, to a polar curve lying on the plane $\{(u,v,w) \in \RR^3~|~ w=0\}$. Indeed, by using the fact that only $B_1^{p^r}(0)\neq 0$, with $B_1^{p^r}(0) = 1$, $\rho_1 = 0$ and that the sum over the $\per{B}_i^{p^\theta}(\theta)$ is one for any $\theta \in [0, 2\pi]$, we have
\begin{equation*}
\begin{split}
\pmb{F}(0,\theta, \phi) &= \sum_{k=1}^{n^\phi}\sum_{j=1}^{n^\theta}\left((\bar{\rho} + \rho_1\cos\theta_j)\cos\varphi_k,(\bar{\rho} + \rho_1\cos\theta_j)\sin\varphi_k,\rho_1 \sin\theta_j\right)B_1^{p^r}(0)\per{B}_j^{p^\theta}(\theta)\per{B}_k^{p^\phi}(\phi)\\
&=\sum_{k=1}^{n^\phi}\sum_{j=1}^{n^\theta}\left(\bar{\rho}\cos\varphi_k,\bar{\rho}\sin\varphi_k,0\right)\per{B}_j^{p^\theta}(\theta)\per{B}_k^{p^\phi}(\phi)\\
&=\sum_{k=1}^{n^\phi}(\bar{\rho}\cos\varphi_k,\bar{\rho}\sin\varphi_k,0)\per{B}_k^{p^\phi}(\phi).
\end{split}
\end{equation*}
Thus, by imposing that $f(0,\theta, \phi)$ is constant in the $\theta$ direction, that is
\begin{equation}\label{c1}
f(0,\theta, \phi) \in \SSS^{p^\phi}, f(0,\theta,\phi) = \sum_{k=1}^{n^\phi} \alpha_k \per{B}_k^{p^\phi}(\phi),
\end{equation}
with $\alpha_k \in \RR$ for $k=1, \ldots, n^\phi$, we ensure that $f^{pol}$ is well-defined and that its restriction to the toroidal direction is $C^1$-continuous in $\Omega^{pol}$, as $\SSS^{p^\varphi}$ is $C^1$ periodic.

The conditions to have $f^{pol}$ $C^1$-smooth also in the radial and poloidal directions have been analyzed in \cite{polarC1}. These enforce the following further relation to be satisfied together with Equation \eqref{c1}: there must exist $\beta_k, \gamma_k \in \RR$ for $k=1, \ldots, n^\phi$ such that 
\begin{equation}\label{c2}
\frac{\partial}{\partial r} f(0,\theta,\phi) = \sum_{k=1}^{n^\phi} \sum_{j=1}^{n^\theta} \sum_{i=1}^2  [\pmb{F}_{ijk}\cdot (\beta_k,\gamma_k, 0)] \frac{\partial}{\partial r} B_{ijk}^{\m{p}}(0,\theta,\phi),
\end{equation}
with $\pmb{F}_{ijk}$ the control points of the polar map.
By following the construction of \cite{polarC1}, we have that Equations \eqref{c1}-\eqref{c2} are verified if $f$ is taken in the restricted space $V_0$, defined as follows. Let $\m{B}^{(0,0,0)}$ be the vectorization of the B-spline basis $\{B_{ijk}^{\m{p}}\}_{i,j,k=1}^{n^r,n^\theta,n^\phi}$ of $\SSS^{p^r,p^\theta,p^\phi}$ obtained by flattening the multi-index $ijk$ as $ijk \to j+(i-1)n^\theta+(k-1)n^rn^\theta$. Then $V_0$ is spanned by the splines in the collection $\m{N}^{(0)}$ given by 
\begin{equation}\label{N0}
\m{N}^{(0)} = \bbol{E}^{(0)}\m{B}^{(0,0,0)}
\end{equation}
where $\bbol{E}^{(0)}$ is the restriction matrix, of size $n^\phi(n^\theta(n^r-2)+3)\times n^rn^\theta n^\phi$ defined as Kronecker's product $\bbol{E}^{(0)} \defeq \bbol{I}_{n^\phi} \otimes \bbol{E}$ with $\bbol{E}$ the following block matrix, introduced in \cite{MDB1}:
\begin{equation}\label{E0}
\bbol{E} \defeq\left[\begin{array}{cc} \bar{\bbol{E}} & \scalebox{1.5}{$\circ$}\\\scalebox{1.5}{$\circ$} & \bbol{I}_{n^\theta(n^r-2)}\end{array}\right], \quad
\bar{\bbol{E}} = \left[\begin{array}{cccccccc}
\bar{E}_{1,1} & \cdots & \bar{E}_{1,n^\theta} & \bar{E}_{1,n^\theta + 1} & \cdots  & \bar{E}_{1,2n^\theta}\\
\bar{E}_{2,1} & \cdots & \bar{E}_{2,n^\theta} & \bar{E}_{2,n^\theta + 1} & \cdots  & \bar{E}_{2,2n^\theta}\\
\bar{E}_{3,1} & \cdots & \bar{E}_{3,n^\theta} & \bar{E}_{3,n^\theta + 1} & \cdots  & \bar{E}_{3,2n^\theta}\\
\end{array}\right]
\end{equation}
where $\bar{\bbol{E}}$ the rectangular matrix of size $3\times 2n^\theta$ whose elements are:
\begin{equation}\label{eq:Ebar}
\begin{array}{l}
\bar{E}_{\ell,j} = \frac{1}{3} \qquad\forall\,\ell=1,2,3\text{ and }\forall\,j=1, \ldots, n^\theta,\\\\
\left[\begin{array}{c}\bar{E}_{1,n^\theta + j}\\\bar{E}_{2,n^\theta + j}\\\bar{E}_{3,n^\theta + j}\end{array}\right] = \left[\begin{array}{c}\frac{1}{3}\\\frac{1}{3}\\\frac{1}{3}\end{array}\right]+\left[\begin{array}{cc}\frac{1}{3} & 0\\ -\frac{1}{6} & \frac{\sqrt{3}}{6}\\-\frac{1}{6}&-\frac{\sqrt{3}}{6}\end{array}\right]\left[\begin{array}{c}\cos\theta_j\\\sin\theta_j\end{array}\right]\qquad\forall\,j=1, \ldots, n^\theta.
\end{array}
\end{equation}
As $\bbol{E}$ is DTA-compatible \cite[Theorem 3.2]{polarC1}, also $\bbol{E}^{(0)}$ is DTA-compatible. This means that the functions in $\m{N}^{(0)}$, $\{N_\ell^{(0)}\}_{\ell=1}^{n^0}$ with $n^0 \defeq n^\phi(n^\theta(n^r-2)+3)$, are linearly independent, locally supported and form a convex partition of unity on $\Omega$. Moreover their pushforward counterpart, $N_\ell^{(0),pol}:\Omega^{pol} \to \RR$, defined as $$
N_\ell^{(0),pol} :=\cF^0(N_\ell^{(0)}) = N_\ell^{(0)} \circ \pmb{F}^{-1}, 
$$
are well-defined, $C^1$-smooth, linearly independent and form a convex partition of unity on $\Omega^{pol}$. 

\subsection{Geometric interpretation of the degrees of freedom}\label{geometricinterpretation}
A geometric interpretation to the degrees of freedom (DOFs) defining the functions in $\SSS^{p^r,p^\theta,p^\phi}$, $V_0$ and in the spaces we are going to consider in the next sections will help us to understand the action of the differential operators on such discrete spaces and will make the discretization scheme reported in Figure \ref{discretization} clearer. This interpretation was introduced in \cite{Buffa3} and adopted later in, e.g., \cite{evans, deepesh}. 
Suppose to have a spline complex on the parametric domain $\Omega$:
$$
\mathfrak{X}:\quad 0 \xrightarrow{\text{id}} X_0 \xrightarrow{\grad} X_1 \xrightarrow{\curl} X_2 \xrightarrow{\div} X_3 \xrightarrow{0} 0.
$$
$X_0$ could be, for instance, the full spline space $\SSS^{p^r,p^\theta,p^\phi}$ or the restricted space $V_0$. Suppose also to have a geometry map $\pmb{G}:\Omega\to \Omega^G$, from the parametric domain $\Omega$ to some physical domain $\Omega^G$ and let $\pmb{G}$ belong to $X_0$. From $\pmb{G}$ one can define the control mesh $\cM^G$ by connecting its control points. Finally, let us assume the pushforward operators for $\pmb{G}$ with respect to $\mathfrak{X}$, $\mathcal{G}_\ell$ with $\ell =0,1,2,3$, to be well defined. 
Then \cite[Proposition 4.4]{Buffa3} ensures that the spline complex $\mathfrak{X}$ is isomorphic to the \textbf{cochain complex on} $\cM^G$. The latter can be defined modifying the definition of  simplicial cochain complexes \cite[Section 2.5]{FEECbook} by replacing the uderlying triagulation with the control mesh $\cM^G$. Such cochain complex on $\cM^G$ has the following form
$$
0 \xrightarrow{\text{id}} \cA^0(\cM^G) \xrightarrow{\partial^0} \cA^1(\cM^G) \xrightarrow{\partial^1} \cA^2(\cM^G) \xrightarrow {\partial^2} \cA^3(\cM^G) \xrightarrow{0} 0.
$$ 
What interests us in this complex is that the DOFs in $\cA^0(\cM^G)$ are assigned to the vertices of $\cM^G$, in $\cA^1(\cM^G)$ to the (oriented) edges, in $\cA^2(\cM^G)$ to the (oriented) faces and in $\cA^3(\cM^G)$ to the (oriented) volumes.

Let now $\pmb{G}$ be the identity map and $\mathfrak{X}$ be the complex of the full spline space on $\Omega$ reported in the first row of the scheme in Figure \ref{discretization}. 
The control mesh $\cM^G$ in this case is a tensor mesh $\cM$ with vertices of indices $\{(i,j,k): i=1, \ldots,n^r;~j=1,\ldots,n^\theta;~k=1, \ldots, n^\phi\}$ called \textbf{Greville mesh} \cite[Section 3.2.1]{Buffa3}. Then, there is a one-to-one correspondence between\begin{itemize}
\item the DOFs in $\SSS^{p^r,p^\theta,p^\phi}$ and the vertices of $\cM$,
\item the DOFs in $\SSS^{p^r-1,p^\theta,p^\phi}\times\SSS^{p^r,p^\theta-1,p^\phi}\times\SSS^{p^r,p^\theta,p^\phi-1}$ and the (oriented) edges in $\cM$
\item the DOFs in $\SSS^{p^r,p^\theta-1,p^\phi-1}\times\SSS^{p^r-1,p^\theta,p^\phi-1}\times\SSS^{p^r-1,p^\theta-1,p^\phi}$ and the (oriented) faces in $\cM$,
\item the DOFs in $\SSS^{p^r-1,p^\theta-1,p^\phi}$ and the (oriented) volumes, or elements, in $\cM$.
\end{itemize} 

A similar matching of DOFs to geometric entities holds when considering the reduced spaces $V_0, V_1, V_2, V_3$. This time we take $\pmb{G}$ as the following $C^1$ parametrization of the toroidal domain:
\begin{equation}\label{geometrymap}
\pmb{G}(r, \theta, \varphi) := \sum_{\ell=1}^{n_0} \pmb{G}_{\ell} N_\ell^{(0)}(r,\theta,\varphi) \quad \text{for }(r,\theta,\varphi) \in \Omega
\end{equation}
where the $\pmb{G}_\ell$ have the follwing expressions:
$$
\begin{array}{l}
\pmb{G}_{1 +(k-1)(n^\theta(n^r-2)+3)}= ((\bar{\rho} + \rho_2)\cos\varphi_k, (\bar{\rho} + \rho_2)\sin\varphi_k, 0)\\\\
\pmb{G}_{2 +(k-1)(n^\theta(n^sr-2)+3)} = ((\bar{\rho} - \frac{1}{2}\rho_2)\cos\varphi_k, (\bar{\rho} - \frac{1}{2}\rho_2)\sin\varphi_k, \frac{\sqrt{3}}{2}\rho_2)\\\\
\pmb{G}_{3 +(k-1)(n^\theta(n^r-2)+3)} = ((\bar{\rho} - \frac{1}{2}\rho_2)\cos\varphi_k, (\bar{\rho} - \frac{1}{2}\rho_2)\sin\varphi_k, -\frac{\sqrt{3}}{2}\rho_2)\\\\
\pmb{G}_{3+j + (i-3)n^\theta+(k-1)(n^\theta(n^r-2)+3)} = \pmb{F}_{ijk}\quad \text{for } i = 3, \ldots, n^r;\, j = 1, \ldots, n^\theta;\, k = 1, \ldots, n^\varphi;
\end{array}
$$
This time \cite[Proposition 4.4]{Buffa3} guarantees that the DOFs in the complex of the reduced spline spaces (reported in the second row of the scheme in Figure \ref{discretization}) are related to the DOFs in the cochain complex of the control mesh $\cR:=\cM^G$, represented in Figure \ref{controlmesh}. 
\begin{figure}
\centering
\includegraphics[width=.5\textwidth]{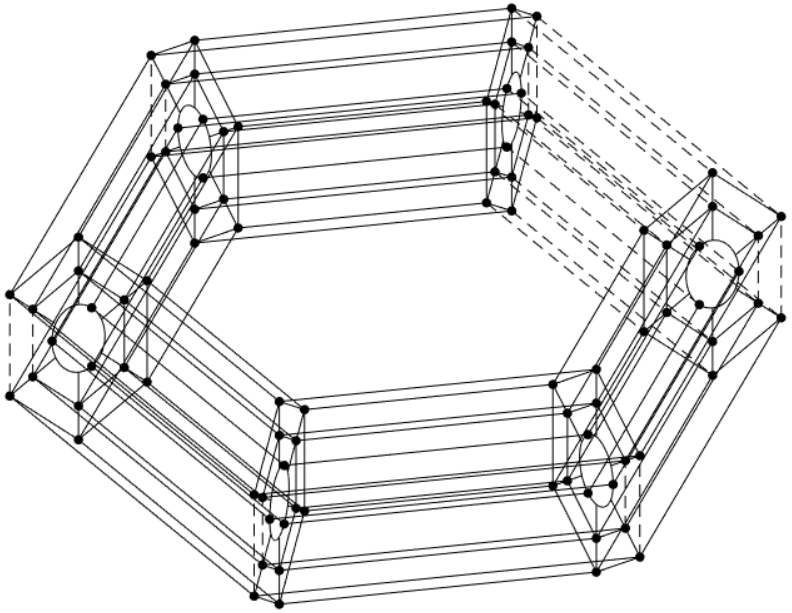}
\caption{Part of the control mesh $\cR$ associated to the geometry map $\pmb{G}$ of Equation \eqref{geometrymap}. We show only the first two rounds of control points in the radial direction for the readability of the figure.}\label{controlmesh}
\end{figure}
$\cR$ is a polygonal ring composed of $n^\varphi$ sides connecting at joints which have the shape of a net made of $n^r-2$ concentric polygons of $n^\theta$ vertices plus three further equidistant vertices lying on a circumference at the center. In particular the DOFs in $V_0$ are linked to the $n_0$ vertices of $\cR$ and the DOFs in the spaces we are going to define in the next sections, $V_1, V_2$ and $V_3$, will be related to the oriented edges, faces and volumes in $\cR$, respectively. 
Figure \ref{dof0} shows this geometric interpretation of the DOFs when taking a function $f \in V_0$ and in $\SSS^{p^r,p^\theta,p^\varphi}$. As we have observed in Remark \ref{linearalgebraremark}, the DOFs in $V_0$ are linked to the DOFs in $\SSS^{p^r,p^\theta,p^\varphi}$ via the transpose of $\bbol{E}^{(0)}$. More precisely, if $f = \m{N}^{(0)} \cdot \m{f} \in V_0$ for a set of coefficients $\{f_\ell\}_{\ell = 1}^{n_0}$ then the corresponding coefficients $\{f_{ijk}\}_{i,j,k = 1}^{n^r, n^\theta, n^\varphi}$ will be
\begin{equation}\label{eq:fijk}
\begin{minipage}{.7\textwidth}
\begin{algorithm}[H]
\For{$k = 1, \ldots, n^\varphi$}{
\For{$j = 1, \ldots, n^\theta$}{
\For{$i = 1, 2$}{
$\displaystyle f_{ijk} = \sum_{\ell=1}^3 \bar{E}_{\ell, (i-1)n^\theta + j} f_{\ell + (k-1)(n^\theta(n^r-2) + 3)}$\;
}
\For{$i = 3, \ldots, n^r$}{
$f_{ijk} = f_{3 + j + (i-3)n^\theta + (k-1)(n^\theta(n^r - 2) + 3)}$\; 
}}}
\end{algorithm}
\end{minipage}
\end{equation}
\begin{figure}
\begin{tikzpicture}[scale=3.5]
\node at (2.55,1.25) {\includegraphics[width=.6\textwidth]{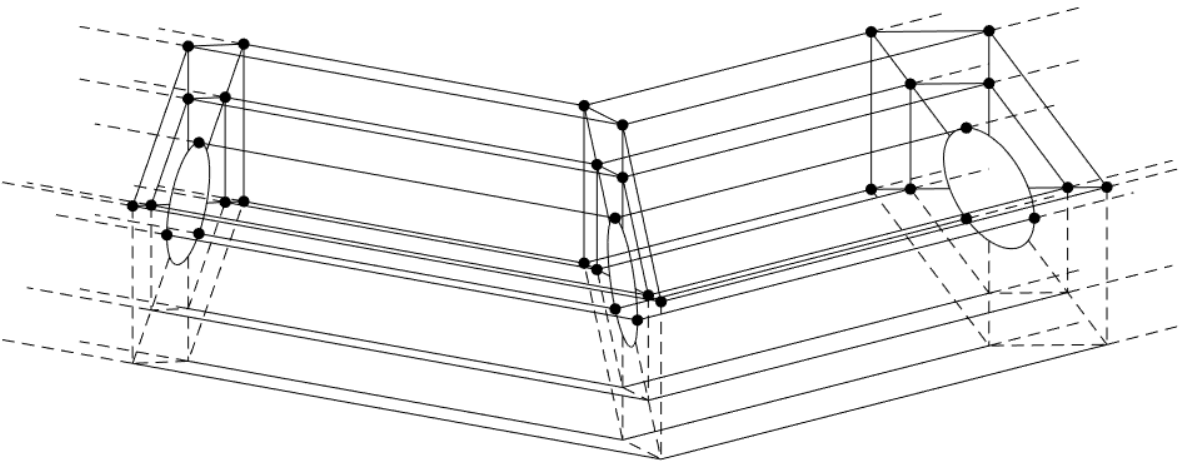}};
\foreach \x in{0,...,4}
{\foreach \y in{0,...,4}{
\fill (\x*.25,\y*.25,1) circle (.02);
\fill (\x*.25,1,\y*.25) circle (.02);
\fill (1,\x*.25,\y*.25) circle (.02);}
\draw (0,\x*.25,1) -- (.75,\x*.25 ,1);
\draw[dashed] (.75,\x*.25 ,1)--(1,\x*.25 ,1);
\draw (\x*.25 ,0,1) -- (\x*.25 ,.75,1);
\draw[dashed] (\x*.25 ,.75,1)--(\x*.25 ,1,1);
\draw (1,\x*.25 ,1) -- (1,\x*.25 ,.25);
\draw[dashed] (1,\x*.25 ,.25)--(1,\x*.25 ,0);
\draw (\x*.25 ,1,1) -- (\x*.25 ,1,.25);
\draw[dashed] (\x*.25 ,1,.25)--(\x*.25 ,1,0);
\draw (1,0,\x*.25 ) -- (1,.75,\x*.25 );
\draw[dashed] (1,.75,\x*.25 )--(1,1,\x*.25 );
\draw (0,1,\x*.25 ) -- (.75,1,\x*.25 );
\draw[dashed] (.75,1,\x*.25 )--(1,1,\x*.25 );
}

\draw (0,0,1) node[below]{{\scriptsize $1,1,1\qquad~~$}};
\draw (.25,0,1) node[below]{{\scriptsize $2,1,1\quad$}};
\draw (.5,0,1) node[below]{{\scriptsize $3,1,1\quad$}};
\draw (.75,0,1) node[below]{{\scriptsize $4,1,1\quad$}};
\draw (1,0,1) node[below]{{\scriptsize $n^r,1,1$}};

\draw (0,.25,1) node[below]{{\scriptsize $1,2,1\qquad~~$}};
\draw (0,.5,1) node[below]{{\scriptsize $1,3,1\qquad~~$}};
\draw (0,.75,1) node[below]{{\scriptsize $1,4,1\qquad~~$}};
\draw (0,1,1) node[below]{{\scriptsize $1,n^\theta,1\qquad~~~$}};

\draw (0,1,.75) node[left]{{\scriptsize $1,n^\theta,2$}};
\draw (0,1,.5) node[left]{{\scriptsize $1,n^\theta,3$}};
\draw (0,1,.25) node[left]{{\scriptsize $1,n^\theta,4$}};
\draw (0,1,0) node[left]{{\scriptsize $1,n^\theta,n^\varphi$}};
\draw[-stealth]  (1.25,1.5) to [in=80,out=180] (.5,1.15);
\node at (.575,1.5) {$(\bbol{E}^{(0)})^T$};
\draw[-stealth] (2.75,.65) to [out=330,in=0] (2.75,.15);
\node at (2.25,.1) {
\begin{tikzpicture}[scale=3]
\draw (.303,.175) -- (0,.35) -- (-.303,.175) -- (-.303,-.175);
\draw[dashed] (-.303,-.175) -- (0,-.35) -- (.303,-.175) -- (.303,.175);
\draw (0,0) circle (.135);
\draw (.519,.3) -- (0,.6) -- (-.519,.3) -- (-.519,-.3);
\draw[dashed] (-.519,-.3) -- (0,-.6) -- (.519,-.3) -- (.519,.3);
\draw (0,.135) -- (0,.6);
\draw[dashed] (0,.7) -- (0,.6);
\draw[dashed] (0,-.135) -- (0,-.7);
\draw (-.519,.3) -- (-.116,.067);
\draw[dashed] (-.596,.345)--(-.519,.3);
\draw (.519,.3) -- (.116,.067);
\draw[dashed] (.596,.345)--(.519,.3);
\draw (-.519,-.3) -- (-.116,-.067);
\draw[dashed] (-.596,-.345)--(-.519,-.3);
\draw[dashed] (.596,-.345) -- (.116,-.067);
\fill (.135,0)node[left]{{\scriptsize 1}} circle (.02);
\fill (-.067, .116)node[below]{{\scriptsize \,\,\,\,2}} circle (.02); 
\fill (-.067, -.116)node[above]{{\scriptsize \,\,\,\,3}} circle (.02);
\fill (-.303,.175)node[above]{{\scriptsize 5}} circle (.02);
\fill (0,.35)node[above]{{\scriptsize \,\,\,\,4}} circle (.02);
\fill (.303,.175)node[right]{{\scriptsize $n^\theta+3$}} circle (.02);
\fill (-.303,-.175)node[below]{{\scriptsize 6}} circle (.02);
\fill (-.519,.3)node[left]{{\scriptsize $n^\theta+5\,\,$}} circle (.02);
\fill (0,.55)node[above]{{\scriptsize $n^\theta+4\qquad\quad$}} circle (.02);
\fill (.519,.3)node[right]{{\scriptsize $2n^\theta+3$}} circle (.02);
\fill (-.519,-.3)node[left]{{\scriptsize $n^\theta+6$}} circle (.02);
\draw (.75,-.55) node{{\scriptsize $+(k-1)(n^\theta(n^r-2)+3)$}};
\end{tikzpicture}
};
\node at (1.85,1.8) {$k-1$};
\node at (2.5,1.65) {$k$};
\node at (3.1,1.8) {$k+1$};
\end{tikzpicture}
\caption{Geometric interpretation of the DOFs of a function in $\SSS^{p^r,p^\theta,p^\varphi}$ (left) and $V_0$ (right). In the former space, the DOFs are associated to the vertices of a tensor mesh. In the latter they are associated to the vertices of a polygonal ring, partially represented here, in which the joints between two sides are nets of concentric polygons. In the middle figure we show the numbering of the vertices corresponding to the DOFs in $V_0$ for the $k$th joint, with $k\in \{1, \ldots, n^\varphi\}$. After the first three vertices in the center, we count anticlockwise the vertices of the concentric polygons from any radial halfline starting at the circumference. When considering a spline in $V_0$, one can get its representation in terms of the basis functions of $\SSS^{p^r,p^\theta,p^\varphi}$ by applying the transpose of $\bbol{E}^{(0)}$ defined in \eqref{N0}.}\label{dof0}
\end{figure}

\subsection{The gradient of functions in $V_0$ and the space $V_1$}
Let $f \in \SSS^{p^r,p^\theta,p^\varphi}$, $f=\m{B}^{(0,0,0)}\cdot \m{f}$ for $\m{f}=(f_1, \ldots, f_{n^rn^\theta n^\varphi})^T$. In order to simplify the notation, let again $f_{ijk}=f_{j+(i-1)n^\theta+(k-1)n^rn^\theta}$. The partial derivatives of $f$ have the following expressions:
$$
\begin{array}{l}
\displaystyle\frac{\partial}{\partial r} f = \sum_{k=1}^{n^\varphi}\sum_{j=1}^{n^\theta}\sum_{i=1}^{n^r-1} (f_{(i+1)jk} - f_{ijk})D_i^{p^r}\per{B}_j^{p^\theta}\per{B}_k^{p^\varphi},\\\\
\displaystyle\frac{\partial}{\partial \theta} f = \sum_{k=1}^{n^\varphi}\sum_{j=1}^{n^\theta}\sum_{i=1}^{n^r} (f_{i(j+1)k} - f_{ijk})B_i^{p^r}\per{D}_j^{p^\theta}\per{B}_k^{p^\varphi},\\\\
\displaystyle\frac{\partial}{\partial \varphi} f = \sum_{k=1}^{n^\varphi}\sum_{j=1}^{n^\theta}\sum_{i=1}^{n^r} (f_{ij(k+1)} - f_{ijk})B_i^{p^r}\per{B}_j^{p^\theta}\per{D}_k^{p^\varphi},
\end{array}
$$
where the functions $D_i^{p^r}, \per{D}_j^{p^\theta}, \per{D}_k^{p^\varphi}$, for any $i,j,k$, have been defined in Section \ref{preliminaries}. Let $\m{B}^{(1,0,0)}, \m{B}^{(0,1,0)}$ and $\m{B}^{(0,0,1)}$ be the vectorizations of $\{D_i^{p^r}\per{B}_j^{p^\theta}\per{B}_k^{p^\varphi}\}_{i,j,k=1}^{n^r-1,n^\theta,n^\varphi}, \{B_i^{p^r}\per{D}_j^{p^\theta}B_k^{p^\varphi}\}_{i,j,k=1}^{n^r,n^\theta,n^\varphi}$ and $\{B_i^{p^r}\per{B}_j^{p^\theta}\per{D}_k^{p^\varphi}\}_{i,j,k=1}^{n^r,n^\theta,n^\varphi}$, respectively, running the indices on $j$, then on $i$ and then on $k$. The gradient operator $\grad: \SSS^{p^r,p^\theta,p^\varphi} \to \SSS^{p^r-1,p^\theta,p^\varphi}\times \SSS^{p^r,p^\theta-1,p^\varphi}\times \SSS^{p^r,p^\theta,p^\varphi-1}$ has the following matrix representation:
$$
\grad f = \left(\begin{array}{c}\m{B}^{(1,0,0)} \cdot \bbol{D}^{(1,0,0)}\m{f}\\\\\m{B}^{(0,1,0)} \cdot \bbol{D}^{(0,1,0)}\m{f}\\\\\m{B}^{(0,0,1)} \cdot \bbol{D}^{(0,0,1)}\m{f}\end{array}\right) \quad \forall\,f \in \SSS^{p^r,p^\theta,p^\varphi},
$$
where $\bbol{D}^{(1,0,0)}, \bbol{D}^{(0,1,0)}$ and $\bbol{D}^{(0,0,1)}$ are matrices of size $(n^r-1)n^\theta n^\varphi\times n^rn^\theta n^\varphi, n^rn^\theta n^\varphi\times n^rn^\theta n^\varphi$ and $n^rn^\theta n^\varphi\times n^rn^\theta n^\varphi$, respectively, defined as
\begin{equation}\label{D100}
\begin{array}{l}
\bbol{D}^{(1,0,0)} = \bbol{I}_{n^\varphi}\otimes \bbol{D}_{n^r}\otimes \bbol{I}_{n^\theta},\\\\
\bbol{D}^{(0,1,0)} = \bbol{I}_{n^\varphi}\otimes \bbol{I}_{n^r} \otimes \per{\bbol{D}}_{n^\theta},\\\\
\bbol{D}^{(0,0,1)} = \per{\bbol{D}}_{n^\varphi} \otimes \bbol{I}_{n^r} \otimes \bbol{I}_{n^\theta},
\end{array}
\end{equation}
with $\bbol{D}_q$ the matrix of size $(q-1)\times q$ whose structure and entries are presented in Equation \eqref{Dn} and $\per{\bbol{D}}_{q}$ is the square matrix of size $q$ whose structure and entries are presented in Equation \eqref{Dnperiodic}. 

We can give a geometric interpretation of the action of the gradient in $\SSS^{p^r,p^\theta,p^\varphi}$.  We have seen that the DOFs in $\SSS^{p^r,p^\theta,p^\varphi}$ and $\SSS^{p^r-1,p^\theta,p^\varphi}\times \SSS^{p^r,p^\theta-1,p^\varphi}\times \SSS^{p^r,p^\theta,p^\varphi-1}$ can be associated respectively to the vertices and the (oriented) edges of a tensor mesh $\cM$. In particular the DOFs in $\im(\grad)$ have form \begin{equation}\label{dof1img}
\left\{\begin{array}{ll}g_{ijk}^1=f_{(i+1)jk}-f_{ijk}&\text{for }i=1, \ldots, n^r-1;~j=1, \ldots, n^\theta;~k=1, \ldots, n^\varphi;\\\\g_{ijk}^2=f_{i(j+1)k}-f_{ijk}&\text{for }i=1, \ldots, n^r;~j=1, \ldots, n^\theta;~k=1, \ldots, n^\varphi;\\\\ g_{ijk}^3=f_{ij(k+1)}-f_{ijk}&\text{for }i=1, \ldots, n^r;~j=1, \ldots, n^\theta;~k=1, \ldots, n^\varphi;\end{array}\right.
\end{equation}
for some set of DOFs in $\SSS^{p^r,p^\theta,p^\varphi}$, $\{f_{ijk}\}_{i,j,k=1}^{n^r,n^\theta,n^\varphi}$. If these latter are assigned to the vertices of $\cM$, then the differences in \eqref{dof1img} identify particular oriented edges on $\cM$. For example $g_{ijk}^1$ is referred to the edge between the vertices $(i,j,k)$ and $(i+1,j,k)$ oriented from the former towards the latter. This interpretation is represented in Figure \ref{tensorgradient}.
\begin{figure}
\centering
\begin{tikzpicture}[scale=3.5]
\foreach \x in{0,...,4}
{\foreach \y in{0,...,4}{
\fill (\x*.25,\y*.25,1) circle (.02);
\fill (\x*.25,1,\y*.25) circle (.02);
\fill (1,\x*.25,\y*.25) circle (.02);}
\draw (0,\x*.25,1) -- (.75,\x*.25 ,1);
\draw[dashed] (.75,\x*.25 ,1)--(1,\x*.25 ,1);
\draw (\x*.25 ,0,1) -- (\x*.25 ,.75,1);
\draw[dashed] (\x*.25 ,.75,1)--(\x*.25 ,1,1);
\draw (1,\x*.25 ,1) -- (1,\x*.25 ,.25);
\draw[dashed] (1,\x*.25 ,.25)--(1,\x*.25 ,0);
\draw (\x*.25 ,1,1) -- (\x*.25 ,1,.25);
\draw[dashed] (\x*.25 ,1,.25)--(\x*.25 ,1,0);
\draw (1,0,\x*.25 ) -- (1,.75,\x*.25 );
\draw[dashed] (1,.75,\x*.25 )--(1,1,\x*.25 );
\draw (0,1,\x*.25 ) -- (.75,1,\x*.25 );
\draw[dashed] (.75,1,\x*.25 )--(1,1,\x*.25 );
}

\draw (0,0,1) node[below]{{\scriptsize $1,1,1\qquad~~$}};
\draw (.25,0,1) node[below]{{\scriptsize $2,1,1\quad$}};
\draw (.5,0,1) node[below]{{\scriptsize $3,1,1\quad$}};
\draw (.75,0,1) node[below]{{\scriptsize $4,1,1\quad$}};
\draw (1,0,1) node[below]{{\scriptsize $n^r,1,1$}};

\draw (0,.25,1) node[below]{{\scriptsize $1,2,1\qquad~~$}};
\draw (0,.5,1) node[below]{{\scriptsize $1,3,1\qquad~~$}};
\draw (0,.75,1) node[below]{{\scriptsize $1,4,1\qquad~~$}};
\draw (0,1,1) node[below]{{\scriptsize $1,n^\theta,1\qquad~~~$}};

\draw (0,1,.75) node[left]{{\scriptsize $1,n^\theta,2$}};
\draw (0,1,.5) node[left]{{\scriptsize $1,n^\theta,3$}};
\draw (0,1,.25) node[left]{{\scriptsize $1,n^\theta,4$}};
\draw (0,1,0) node[left]{{\scriptsize $1,n^\theta,n^\varphi$}};
\draw[-stealth] (1.4,0,1) --node[below]{$r$} (1.7,0,1);
\draw[-stealth] (1.4,0,1) --node[left]{$\theta$} (1.4,.3,1);
\draw[-stealth](1.4,0,1) -- (1.4,0,.5) node[above]{$\varphi$};
\end{tikzpicture}\hspace{-1cm}
\begin{tikzpicture}[scale=3]
\draw[-stealth] (0,.8) --node[above]{${\footnotesize\begin{array}{c}\textcolor{red}{\bbol{D}^{(1,0,0)}}\\\textcolor{green!50!black}{\bbol{D}^{(0,1,0)}}\\ \textcolor{blue}{\bbol{D}^{(0,0,1)}} \end{array}}$} (.4,.8);
\draw[white] (0,0) -- (.4,0);
\end{tikzpicture}\hspace{-.5cm}
\begin{tikzpicture}[scale=3.5]
\foreach \x in{0,...,4}
{\foreach \y in{0,...,2}{
   \draw[red,->-] (\y*.25,\x*.25,1) -- (\y*.25+.25,\x*.25 ,1);
   \draw[red,->-] (\y*.25,1,\x*.25 ) -- (\y*.25+.25,1,\x*.25 );
}
\draw[red,dashed] (.75,\x*.25 ,1)--(1,\x*.25 ,1);
\draw[red,dashed] (.75,1,\x*.25)--(1,1,\x*.25);
}
\foreach \x in{0,...,4}
{\foreach \y in{0,...,2}{
\draw[green!50!black,->-] (\x*.25 ,\y*.25,1) -- (\x*.25 ,\y*.25+.25,1);
\draw[green!50!black,->-] (1,\y*.25,\x*.25 ) -- (1,\y*.25+.25,\x*.25 );
\draw[green!50!black,->-] (\x*.25,1,\y*.25 ) -- (\x*.25,1.25,\y*.25 );
\draw[green!50!black,->-] (\x*.25,1,\y*.25+.5) -- (\x*.25,1.25,\y*.25+.5);
}
\draw[green!50!black,dashed] (\x*.25 ,.75,1)--(\x*.25 ,1,1);
\draw[green!50!black,dashed] (1 ,.75,\x*.25)--(1,1,\x*.25);
}
\foreach \x in{0,...,4}
{\foreach \y in{-1,1,2,3}{
    \draw[blue,->-] (1,\x*.25 ,\y*.25+.25) -- (1,\x*.25 ,\y*.25);
    \draw[blue,->-] (\x*.25 ,1,\y*.25+.25) -- (\x*.25 ,1,\y*.25);
    }
\draw[blue,dashed] (\x*.25,1,0) -- (\x*.25,1,.25);    
\draw[blue,dashed] (1,\x*.25,0) -- (1,\x*.25,.25);
    }
\draw[red] (.125,0,1) node[below]{{\scriptsize $1,1,1$}};
\draw[red] (.375,0,1) node[below]{{\scriptsize $2,1,1$}};
\draw[red] (.625,0,1) node[below]{{\scriptsize $3,1,1$}};

\draw[green!50!black] (0,.125,1) node[left]{{\scriptsize $\hspace{-.05cm}1,1,1$}};
\draw[green!50!black] (0,.375,1) node[left]{{\scriptsize $\hspace{-.05cm}1,2,1$}};
\draw[green!50!black] (0,.625,1) node[left]{{\scriptsize $\hspace{-.05cm}1,3,1$}};
\draw[green!50!black] (0,1.125,1) node[left]{{\scriptsize $\hspace{-.05cm}1,n^\theta,1$}};

\draw[blue] (1,0,.825) node[right]{{\scriptsize $n^r,1,1~$}};
\draw[blue] (1,0,.6) node[right]{{\scriptsize $n^r,1,2~$}};
\draw[blue] (1,0,.325) node[right]{{\scriptsize $n^r,1,3~$}};
\draw[blue] (1,0,-.125) node[right]{{\scriptsize $n^r,1,n^\varphi$}};
\end{tikzpicture}
\caption{The action of the gradient operator when associating the DOFs in $\SSS^{p^r,p^\theta,p^\varphi}$ to the vertices of a tensor mesh $\cM$. In the gradient space $\SSS^{p^r-1,p^\theta,p^\varphi}\times \SSS^{p^r,p^\theta-1,p^\varphi}\times \SSS^{p^r,p^\theta,p^\varphi-1}$ the DOFs correspond to oriented edges. In particular in $\im(\grad)$ the first components of the DOFs correspond to the edges in the $r$ direction, the second components to the edges in the $\theta$ direction and the third components to the edges in the $\varphi$ direction. Because of the periodicity of $\SSS^{p^\theta}$ and $\SSS^{p^\varphi}$, there should be oriented edges connecting the top side with the bottom side of $\cM$ and the back side with the front side of $\cM$. These are instead represented in the right figure as arrows going outward of $\cM$.}\label{tensorgradient}
\end{figure}

Let now $f \in V_0$, $f=\m{N}^{(0)} \cdot \m{f}$ for $\m{f} = (f_1, \ldots, f_{n_0})^T$. In order to simplify the notation in what follows, let $\bar{n}_0 := \frac{n_0}{n^\varphi} = n^\theta(n^r-2)+3$ and $\bar{n}_1 \defeq 2(\bar{n}_0 -2) = 2n^\theta(n^r-2)+2$. We want the action of the gradient on the DOFs in $V_0$ to mimic the action we had on the DOFs in $\SSS^{p^r,p^\theta,p^\varphi}$. This time the collection $\{f_\ell\}_{\ell=1}^{n_0}$ corresponds to the vertices of the polygonal ring $\cR$ introduced in Section \ref{geometricinterpretation}. The action of the gradient should then associated them to oriented edges in $\cR$, as shown in Figure \ref{polargradient}. The numbers of vertices and edges in a joint of $\cR$ are $\bar{n}_0$ and $\bar{n}_1$ respectively. Hence, the total number of edges in $\cR$ is equal to $n_1\defeq n^\varphi(\bar{n}_0 + \bar{n}_1)=n^\varphi(3n^\theta(n^r-2)+5)$, as each vertex in a joint corresponds to an edge of $\cR$ connecting such joint to the next joint.  The degrees of freedom in $V_1$ have been ordered as follows. First we have all those associated to edges in the joints of $\cR$ and then those associated to edges in the sides of $\cR$. For each joint, first we consider the two edges in the central circumference, then we count all the radial edges and then all the poloidal edges anti-clockwise, see again Figure \ref{polargradient}. Instead, the edges in the sides of $\cR$ are ordered as the vertices in $V_0$. 
\begin{figure}
\centering
\scalebox{.9}{
\begin{tikzpicture}[scale=3]
\node at (0,1.75) {\includegraphics[width=.5\textwidth]{dof0}};
\node at (3,1.75) {\includegraphics[width=.5\textwidth]{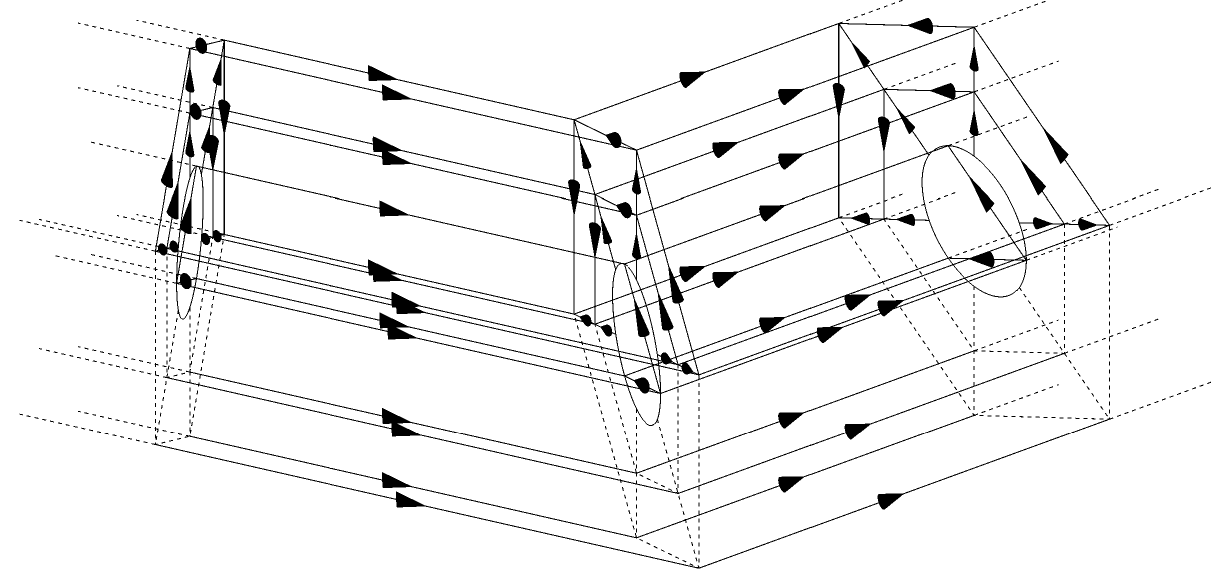}};
\node at (-.75,2.3) {$k-1$};
\node at (0,2.175) {$k$};
\node at (.55,2.3) {$k+1$};
\node at (2.2,2.4) {$k-1$};
\node at (2.925,2.225) {$k$};
\node at (3.45,2.4) {$k+1$};
\draw[-stealth] (1,2.3) to [out=45,in=135] (2,2.3);
\node at (1.5,2.4) {$\bbol{D}^{(0)}$};
\node at (-.5,-1) {
\begin{tikzpicture}[scale=4]
\draw (.303,.175) -- (0,.35) -- (-.303,.175) -- (-.303,-.175);
\draw[dashed] (-.303,-.175) -- (0,-.35) -- (.303,-.175) -- (.303,.175);
\draw (0,0) circle (.135);
\draw (.519,.3) -- (0,.6) -- (-.519,.3) -- (-.519,-.3);
\draw[dashed] (-.519,-.3) -- (0,-.6) -- (.519,-.3) -- (.519,.3);
\draw (0,.135) -- (0,.6);
\draw[dashed] (0,.7) -- (0,.6);
\draw[dashed] (0,-.135) -- (0,-.7);
\draw (-.519,.3) -- (-.116,.067);
\draw[dashed] (-.596,.345)--(-.519,.3);
\draw (.519,.3) -- (.116,.067);
\draw[dashed] (.596,.345)--(.519,.3);
\draw (-.519,-.3) -- (-.116,-.067);
\draw[dashed] (-.596,-.345)--(-.519,-.3);
\draw[dashed] (.596,-.345) -- (.116,-.067);
\fill (.135,0)node[left]{{\scriptsize 1}} circle (.02);
\fill (-.067, .116)node[below]{{\scriptsize \,\,\,\,2}} circle (.02); 
\fill (-.067, -.116)node[above]{{\scriptsize \,\,\,\,3}} circle (.02);
\fill (-.303,.175)node[above]{{\scriptsize 5}} circle (.02);
\fill (0,.35)node[above]{{\scriptsize \,\,\,\,4}} circle (.02);
\fill (.303,.175)node[right]{{\scriptsize $n^\theta+3$}} circle (.02);
\fill (-.303,-.175)node[below]{{\scriptsize 6}} circle (.02);
\fill (-.519,.3)node[left]{{\scriptsize $n^\theta+5\,\,$}} circle (.02);
\fill (0,.55)node[above]{{\scriptsize $n^\theta+4\qquad\quad$}} circle (.02);
\fill (.519,.3)node[right]{{\scriptsize $2n^\theta+3$}} circle (.02);
\fill (-.519,-.3)node[left]{{\scriptsize $n^\theta+6$}} circle (.02);
\draw (.5,-.55) node{{\scriptsize $+(k-1)\bar{n}_0$}};
\end{tikzpicture}
};
\node at (2.5,0) {
\begin{tikzpicture}[scale=4]
\draw[->-] (.135,0) -- (-.067,.116);
\draw[->-] (.135,0) -- (-.067,-.116);
\draw (0,.04) node{{\scriptsize 1}};
\draw (0,.-.04) node{{\scriptsize 2}};
\draw (0,0) circle (.135);
\draw (-.303,.175) -- (0,.35);
\draw (0,.35) -- (.303,.175);
\draw (-.303,-.175) -- (-.303,.175);
\draw[dashed] (-.303,-.175) -- (0,-.35) -- (.303,-.175) -- (.303,.175);
\draw (-.519,.3) -- (0,.6);
\draw (0,.6) -- (.519,.3);
\draw (-.519,-.3) -- (-.519,.3);
\draw[dashed] (-.519,-.3) -- (0,-.6) -- (.519,-.3) -- (.519,.3);
\draw[->-] (0,.135) --node[right]{{\scriptsize 3}} (0,.35);
\draw[->-] (0,.35) --node[right, yshift = -0.15cm]{{\scriptsize $n^\theta + 3$}} (0,.6);
\draw[dashed] (0,.7) -- (0,.6);
\draw[dashed] (0,-.135) -- (0,-.7);
\draw[->-] (-.116,.067) --node[yshift = 0.15cm, xshift = 0.15cm]{{\scriptsize 4}} (-.303,.175);
\draw[->-] (-.303,.175) --node[yshift = 0.15cm, xshift = 0.5cm]{{\scriptsize $n^\theta + 4$}} (-.519,.3);
\draw[dashed] (-.596,.345)--(-.519,.3);
\draw[->-] (.116,.067) --node[yshift = -0.15cm, xshift = 0.25cm]{{\scriptsize $n^\theta + 2$}} (.303,.175);
\draw[->-] (.303,.175) --node[yshift = -0.25cm, xshift = 0.25cm]{{\scriptsize $2n^\theta + 2$}}  (.519,.3);
\draw[dashed] (.596,.345)--(.519,.3);
\draw[->-] (-.116,-.067) --node[yshift = 0.25cm]{{\scriptsize 5}} (-.303,-.175);
\draw[->-] (-.303,-.175) --node[yshift = 0.35cm]{{\scriptsize $n^\theta + 5$}} (-.519,-.3);
\draw[dashed] (-.596,-.345)--(-.519,-.3);
\draw[dashed] (.596,-.345) -- (.116,-.067);
\node at (.425,-.65) {{\scriptsize $+(k-1)\bar{n}_1$}};
\end{tikzpicture}\hspace{.5cm}
\begin{tikzpicture}[scale=4]
\draw (0,0) circle (.135);
\draw[->-] (0,.35) -- node[xshift = -0.25, yshift = 0.25cm]{{\scriptsize\rotatebox{30}{2}}}(-.303,.175);
\draw[->-] (.303,.175) --node[xshift = 0.25, yshift =  0.25cm]{{\scriptsize \rotatebox{-30}{1}}} (0,.35);
\draw[->-] (-.303,.175) --node[left]{{\scriptsize \rotatebox{90}{3}}} (-.303,-.175);
\draw[dashed] (-.303,-.175) -- (0,-.35) -- (.303,-.175) -- (.303,.175);
\draw[->-]  (0,.6) --node[xshift = -0.25, yshift = 0.25cm]{{\scriptsize\rotatebox{30}{$n^\theta + 2$}}} (-.519,.3);
\draw[->-]  (.519,.3) --node[xshift = 0.25, yshift =  0.25cm]{{\scriptsize \rotatebox{-30}{$n^\theta + 1$}}} (0,.6);
\draw[->-] (-.519,.3) --node[left]{{\scriptsize \rotatebox{90}{$n^\theta + 3$}}} (-.519,-.3);
\draw[dashed] (-.519,-.3) -- (0,-.6) -- (.519,-.3) -- (.519,.3);
\draw (0,.135) -- (0,.35);
\draw (0,.35) -- (0,.6);
\draw[dashed] (0,.7) -- (0,.6);
\draw[dashed] (0,-.135) -- (0,-.7);
\draw (-.116,.067) -- (-.303,.175);
\draw (-.303,.175) -- (-.519,.3);
\draw[dashed] (-.596,.345)--(-.519,.3);
\draw (.116,.067) -- (.303,.175);
\draw (.303,.175) --  (.519,.3);
\draw[dashed] (.596,.345)--(.519,.3);
\draw (-.116,-.067) -- (-.303,-.175);
\draw (-.303,-.175) -- (-.519,-.3);
\draw[dashed] (-.596,-.345)--(-.519,-.3);
\draw[dashed] (.596,-.345) -- (.116,-.067);
\node at (.425,-.55) {{\scriptsize $+(k-1)\bar{n}_1$}};
\node at (.425,-.65) {{\scriptsize $+2 + (n^r - 2)n^\theta$}};
\end{tikzpicture}
};
\node at (2.5, -2) {
\begin{tikzpicture}[scale=4]
\draw (.303,.175) -- (0,.35) -- (-.303,.175) -- (-.303,-.175);
\draw[dashed] (-.303,-.175) -- (0,-.35) -- (.303,-.175) -- (.303,.175);
\draw (0,0) circle (.135);
\draw (.519,.3) -- (0,.6) -- (-.519,.3) -- (-.519,-.3);
\draw[dashed] (-.519,-.3) -- (0,-.6) -- (.519,-.3) -- (.519,.3);
\draw (0,.135) -- (0,.6);
\draw[dashed] (0,.7) -- (0,.6);
\draw[dashed] (0,-.135) -- (0,-.7);
\draw (-.519,.3) -- (-.116,.067);
\draw[dashed] (-.596,.345)--(-.519,.3);
\draw (.519,.3) -- (.116,.067);
\draw[dashed] (.596,.345)--(.519,.3);
\draw (-.519,-.3) -- (-.116,-.067);
\draw[dashed] (-.596,-.345)--(-.519,-.3);
\draw[dashed] (.596,-.345) -- (.116,-.067);
\fill (.135,0)node[left]{{\scriptsize 1}} circle (.02);
\fill (-.067, .116)node[below]{{\scriptsize \,\,\,\,2}} circle (.02); 
\fill (-.067, -.116)node[above]{{\scriptsize \,\,\,\,3}} circle (.02);
\fill (-.303,.175)node[above]{{\scriptsize 5}} circle (.02);
\fill (0,.35)node[above]{{\scriptsize \,\,\,\,4}} circle (.02);
\fill (.303,.175)node[right]{{\scriptsize $n^\theta+3$}} circle (.02);
\fill (-.303,-.175)node[below]{{\scriptsize 6}} circle (.02);
\fill (-.519,.3)node[left]{{\scriptsize $n^\theta+5\,\,$}} circle (.02);
\fill (0,.596)node[above]{{\scriptsize $n^\theta+4\qquad\quad$}} circle (.02);
\fill (.519,.3)node[right]{{\scriptsize $2n^\theta+3$}} circle (.02);
\fill (-.519,-.3)node[left]{{\scriptsize $n^\theta+6$}} circle (.02);
\node at (.425,-.55) {{\scriptsize $+n^\varphi\bar{n}_1$}};
\node at (.425,-.65) {{\scriptsize $+(k-1)\bar{n}_0$}};
\end{tikzpicture}
};
\draw[-stealth] (3.1, 1.2) to [out=-100, in=45] (1.75, 0.75); 
\draw[-stealth] (3.1, 1.2) to [out=-80, in=140] (3.3, 0.85);
\draw[-stealth] plot [smooth] coordinates {(3.1, 1.2) (4.35, 0.5) (4.25, -1.75) (3.25, -2)};
\draw[-stealth] (0.25, -1) to [out=0, in=225]node[above]{$\bbol{D}_J^{(0)}$} (1.25, -0.75);
\draw[-stealth] (0.25, -1) to [out=0, in=180]node[below]{$\bbol{I}_{\bar{n}_0}$} (1.75, -2);
\draw[-stealth] (0.25, -1) to [out=0, in=245]node[below]{$\bbol{D}_J^{(0)}$} (2.75, -0.5);
\draw[-stealth] (0.1, 1.2) to [out=225, in=90] (-0.5, 0);
\end{tikzpicture}}
\caption{The action of the gradient operator when associating the DOFs in $V_0$ to the vertices of the polygonal ring $\cR$. In the gradient space the DOFs correspond to the oriented edges. On the top of the figure we see part of $\cR$. On the bottom we see the numbering of the vertices (left) and of the edges (right) when considering the $k$th joint in $\cR$, for $k \in \{1, \ldots, n^\varphi\}$. Note that the edges running around the hole of $\cR$ are represented as dots in the lower-right figure when looking at a joint. $\bbol{D}^{(0)}$ is a matrix mapping the vertices in $\cR$ to the oriented edges in $\cR$ and encodes the geometric interpretation of the gradient action on the DOFs. In particular the blocks $\bbol{D}_J^{(0)}$ of $\bbol{D}^{(0)}$ provide the relation between the vertices and the radial and poloidal edges in a joint. Instead the identity blocks $\bbol{I}_{\bar{n}_0}$ in $\bbol{D}^{(0)}$ map the vertices in the edges in the sides of $\cR$.}\label{polargradient}
\end{figure}
The gradient action on the DOFs of $V_0$ is therefore encoded by the matrix $\bbol{D}^{(0)}$ of size $n_1\times n_0$ with the following block structure
\begin{equation}\label{eq:D0}
\bbol{D}^{(0)} \defeq 
\left[\begin{array}{c}
\bbol{I}_{n^\varphi} \otimes \bbol{D}_J^{(0)}\\
\per{\bbol{D}}_{n^\varphi} \otimes \bbol{I}_{\bar{n}_0} \end{array}\right]
\end{equation}
where the upper block provides the DOFs related to edges in the joints of $\cR$ and the lower block the DOFs related to edges in the sides of $\cR$. The $J$ subscript in $\bbol{D}_J^{(0)}$ indicates that this matrix will provide the DOFs in a joint. This matrix, identified in \cite[Equations (78)-(84)]{deepesh}, has size $\bar{n}_1 \times \bar{n}_0$ and the following structure and entries
\begin{equation*}
\bbol{D}_J^{(0)} \defeq \raisebox{-2cm}{\begin{tikzpicture}
\draw (0, 0) rectangle (3, 4);
\draw[step = 1] (0, 2) grid (2, 4);
\draw (1,1) -- (3,1);
\draw (1,0) -- (1,2);
\draw (2, 2) -- (3, 2);
\node at (0.45, 3.75) {{\fontsize{6}{6} \selectfont -1$\,\,\,$1$\,\,\,$0}};
\node at (0.45, 3.25) {{\fontsize{6}{6} \selectfont -1$\,\,\,$0$\,\,\,$1}};
\node at (1.5, 3.5) {\Large{$\circ$}};
\node at (0.5, 2.5) {\scalebox{0.75}{$-(\bar{\bbol{E}}^R)^T$}};
\node at (1.5, 2.5) {$\bbol{I}_{n^\theta}$};
\node at (2.5, 3) {\Large{$\circ$}};
\node at (0.5, 1) {\Large{$\circ$}};
\node at (2, 1.5) {$\bbol{D}_{n^r-2} \otimes \bbol{I}_{n^\theta}$};
\node at (2, 0.5) {$\bbol{I}_{n^r-2} \otimes \per{\bbol{D}}_{n^\theta}$};
\draw[|-|] (0,4.2) --node[above]{3} (1,4.2);
\draw[-|] (1,4.2) --node[above]{$n^\theta(n^r-2)$} (3, 4.2);
\draw[|-|] (3.2, 4) --node[right]{2} (3.2, 3);
\draw[-|] (3.2, 3) --node[right]{$n^\theta$} (3.2, 2);
\draw[-|] (3.2, 2) --node[right]{$n^\theta(n^r-3)$} (3.2, 1);
\draw[-|] (3.2, 1) --node[right]{$n^\theta(n^r-2)$} (3.2, 0);
\end{tikzpicture}}
\end{equation*} 
with $\bar{\bbol{E}}^R$ the right part of $\bar{\bbol{E}}$ reported in Equation \eqref{E0}, that is, the matrix whose entries are $\bar{E}_{\ell, n^\theta + j}$ for $\ell = 1, 2, 3$ and $j = 1, \ldots, n^\theta$.

When applied to $\m{f}$, $\bbol{D}^{(0)}$ provides a vector $\m{g}=\bbol{D}^{(0)}\m{f}$ with the entries $\{g_\ell\}_{\ell=1}^{n_1}$ given by
\begin{equation}\label{eq:gl}
\begin{minipage}{.925\textwidth}
\begin{algorithm}[H]
\For{$k = 1, \ldots, n^\varphi$}{
$g_{1 + (k-1)\bar{n}_1} = f_{2 + (k-1)\bar{n}_0} - f_{1 + (k-1)\bar{n}_0}$\;
$g_{2 + (k-1)\bar{n}_1} = f_{3 + (k-1)\bar{n}_0} - f_{1 + (k-1)\bar{n}_0}$\;
\For{$j = 1, \ldots, n^\theta$}{
$\displaystyle g_{2 + j + (k-1)\bar{n}_1} = f_{3 + j + (k-1)\bar{n}_0} - \sum_{\ell=1}^3 \bar{E}_{\ell,n^\theta + j}f_{\ell + (k-1)\bar{n}_0}$\;
\For{$i = 3, \ldots, n^r -1$}{
$g_{2 + j + (i-2)n^\theta + (k-1) \bar{n}_1} = f_{3 + j + (i-2)n^\theta + (k-1)\bar{n}_0} - f_{3 + j + (i-3)n^\theta + (k-1)\bar{n}_0}$\;
$g_{2 + (n^r-2)n^\theta + j + (i-3)n^\theta + (k-1)\bar{n}_1} = f_{3 + (j + 1) + (i-3)n^\theta + (k-1)\bar{n}_0} -  f_{3 + j+ (i-3)n^\theta + (k-1)\bar{n}_0}$\;
}
$g_{2 + (n^r-2)n^\theta + j + (n^r-3)n^\theta + (k-1)\bar{n}_1} = f_{3 + (j + 1) + (n^r-3)n^\theta + (k-1)\bar{n}_0} -  f_{3 + j+ (n^r-3)n^\theta + (k-1)\bar{n}_0}$\;
}
\For{$\ell = 1, \ldots, \bar{n}_0$}{
$g_{n^\varphi\bar{n}_1 + \ell + (k-1)\bar{n}_0} = f_{\ell + k\bar{n}_0} - f_{\ell + (k-1)\bar{n}_0}$\;
}
}
\end{algorithm}
\end{minipage}
\end{equation}
All of such entries are differences of two DOFs in $V_0$, corresponding to the endpoints of an edge in $\cR$, except those whose geometric interpretation forms the first round of radial edges in a joint of $\cR$ from the central circumference, namely $g_{2+j+(k-1)\bar{n}_1}$ for $j=1, \ldots, n^\theta$, for which a suitable combination of the three vertices $f_{\ell+(k-1)\bar{n}_0}$ for $\ell=1,2,3$ is invoked. 

The extraction matrix definig the basis of the space $V_1$ from the basis of $ \SSS^{p^r-1,p^\theta,p^\varphi}\times\SSS^{p^r,p^\theta-1,p^\varphi}\times\SSS^{p^r,p^\theta,p^\varphi-1}$ is the following
\begin{equation}\label{eq:E1}
\bbol{E}^{(1)} \defeq \begin{bmatrix} \bbol{I}_{n^\varphi} \otimes \bbol{E}_J^{(1), L} & \bbol{I}_{n^\varphi} \otimes \bbol{E}_J^{(1), R} & \scalebox{1.5}{$\circ$}\\\scalebox{1.5}{$\circ$} & \scalebox{1.5}{$\circ$} & \bbol{I}_{n^\varphi} \otimes \bbol{E}\end{bmatrix} = \begin{bmatrix} \bbol{I}_{n^\varphi} \otimes \bbol{E}_J^{(1), L} & \bbol{I}_{n^\varphi} \otimes \bbol{E}_J^{(1), R} & \scalebox{1.5}{$\circ$}\\\scalebox{1.5}{$\circ$} & \scalebox{1.5}{$\circ$} & \bbol{E}^{(0)}\end{bmatrix}
\end{equation}
with $\bbol{E}_J^{(1), L}, \bbol{E}_J^{(1), R}$ the following matrices of sizes $\bar{n}_1 \times (n^r-1)n^\theta$, $\bar{n}_1 \times n^rn^\theta$, respectively, appearing as left and right blocks of matrix $\bbol{E}_J^{(1)} \defeq \begin{bmatrix}\bbol{E}_J^{(1), L} & \bbol{E}_J^{(1), R}\end{bmatrix}$ identified in \cite[Equations (78)-(84)]{deepesh}, which provides the relations between DOFs associated to edges of the Greville tensor mesh in the $r$ and $\theta$ directions, respectively, i.e., $\{g_{ijk}^1\}_{i,j = 1}^{n^r - 1, n^\theta}$ and $\{g_{ijk}^2\}_{i,j=1}^{n^r, n^\theta}$ for fixed $k$, and the edges in the $k$th joint of $\cR$, 
\begin{equation}\label{eq:EJ1}
\bbol{E}_J^{(1), L} \defeq \raisebox{-1.5cm}{
\begin{tikzpicture}
\draw (0, 0) rectangle (4, 3);
\draw (0, 0) rectangle (2, 2);
\draw (2, 1) -- (4, 1);
\draw (2, 2) -- (2, 3);
\draw (2, 2) -- (4, 2);
\node at (1, 1) {{\Huge$\circ$}};
\node at (3, 0.5) {{\Huge$\circ$}};
\node at (3, 1.5) {$\bbol{I}_{(n^r-2)n^\theta}$};
\node at (3, 2.5) {{\Huge$\circ$}};
\node at (1, 2.5) {$\bar{\bbol{E}}^{(1), L}$};
\draw[|-|] (0, 3.2) --node[above]{$n^\theta$} (2, 3.2);
\draw[-|] (2, 3.2) --node[above]{$(n^r-2)n^\theta$} (4, 3.2);
\draw[|-|] (4.2, 3) --node[right]{$2$} (4.2, 2);
\draw[-|] (4.2, 2) --node[right]{$(n^r-2)n^\theta$} (4.2, 1);
\draw[-|] (4.2, 1) --node[right]{$(n^r-2)n^\theta$} (4.2, 0);
\end{tikzpicture}}\quad
\bbol{E}_J^{(1), R} \defeq \raisebox{-1.5cm}{
\begin{tikzpicture}
\draw (0, 0) rectangle (4, 3);
\draw (0, 0) rectangle (2, 2);
\draw (2, 1) -- (4, 1);
\draw (2, 2) -- (2, 3);
\draw (2, 2) -- (4, 2);
\node at (1, 1) {{\Huge$\circ$}};
\node at (3, 0.5) {$\bbol{I}_{(n^r-2)n^\theta}$};
\node at (3, 1.5) {{\Huge$\circ$}};
\node at (3, 2.5) {{\Huge$\circ$}};
\node at (1, 2.5) {$\bar{\bbol{E}}^{(1), R}$};
\draw[|-|] (0, 3.2) --node[above]{$2n^\theta$} (2, 3.2);
\draw[-|] (2, 3.2) --node[above]{$(n^r-2)n^\theta$} (4, 3.2);
\draw[|-|] (4.2, 3) --node[right]{$2$} (4.2, 2);
\draw[-|] (4.2, 2) --node[right]{$(n^r-2)n^\theta$} (4.2, 1);
\draw[-|] (4.2, 1) --node[right]{$(n^r-2)n^\theta$} (4.2, 0);
\end{tikzpicture}}
\end{equation}
$\bar{\bbol{E}}^{(1), L}$ and $\bar{\bbol{E}}^{(1), R}$ are matrices $2 \times n^\theta$ and $2 \times 2n^\theta$ respectively, with entries
$$
\begin{array}{l}
\bar{E}^{(1), L}_{\ell,j} = \bar{E}_{1 + \ell,n^\theta + j} - \bar{E}_{1+\ell,j}\\\\
\bar{E}^{(1), R}_{\ell, j} =  \bar{E}_{1 + \ell,j+1} - \bar{E}_{1+\ell,j} = 0\\\\
\bar{E}^{(1), R}_{\ell,n^\theta + j} = \bar{E}_{1 + \ell,n^\theta + j+1} - \bar{E}_{1+\ell,n^\theta + j}
\end{array}
\quad \ell = 1, 2; ~j = 1, \ldots, n^\theta,
$$
that is, suitable differences of entries of $\bar{\bbol{E}}$.

Matrix $(\bbol{E}^{(1)})^T$ will provide the relation between the coefficients of splines in $V_1$ and $\SSS^{p^r-1,p^\theta,p^\varphi}\times \SSS^{p^r,p^\theta-1,p^\varphi}\times \SSS^{p^r,p^\theta,p^\varphi-1}$, according to Remark \ref{linearalgebraremark}. Namely, given a set of coefficients in $V_1$, $\m{g}=\{g_\ell\}_{\ell=1}^{n_1}$, we will have that 
\begin{equation}\label{dof1relation}
\begin{bmatrix}
\m{g}^1 \\\\ \m{g}^2 \\\\ \m{g}^3
\end{bmatrix} = (\bbol{E}^{(1)})^T \m{g}
\end{equation} 
where $\m{g}^1, \m{g}^2, \m{g}^3$ are the vectorizations of the coefficients in the first, second and third components of an element $\pmb{g}$ in $\SSS^{p^r-1,p^\theta,p^\varphi}\times \SSS^{p^r,p^\theta-1,p^\varphi}\times \SSS^{p^r,p^\theta,p^\varphi-1}$, i.e.,
$$
\pmb{g} = \left(\begin{array}{c}
\m{B}^{(1,0,0)} \cdot \m{g}^1\\\\\m{B}^{(0,1,0)}\cdot \m{g}^2\\\\\m{B}^{(0,0,1)}\cdot \m{g}^3
\end{array}\right) = \left(\begin{array}{c}g^1\\\\g^2\\\\g^3\end{array}\right) \in \SSS^{p^r-1,p^\theta,p^\varphi}\times \SSS^{p^r,p^\theta-1,p^\varphi}\times \SSS^{p^r,p^\theta,p^\varphi-1}.
$$
More precisely, we have
\begin{equation}\label{eq:gijk}
\begin{minipage}{.925\textwidth}
\begin{algorithm}[H]
\For{$k = 1, \ldots, n^\varphi$}{
\For{$j = 1, \ldots, n^\theta$}{
$\begin{array}{ll} \displaystyle g_{1jk}^1 = \sum_{\ell = 1}^2 \bar{E}_{\ell,j}^{(1), L} g_{\ell + (k - 1)\bar{n}_1}; & \\
\displaystyle g_{1jk}^2 = \sum_{\ell = 1}^2 \bar{E}_{\ell,j}^{(1), R} g_{\ell +  (k - 1)\bar{n}_1} = 0; & \displaystyle g_{2jk}^2 = \sum_{\ell = 1}^2 \bar{E}_{\ell,n^\theta + j}^{(1), R} g_{\ell + (k - 1)\bar{n}_1};\\
\displaystyle g_{1jk}^3 = \sum_{\ell = 1}^3 \bar{E}_{\ell, j} g_{n^\varphi\bar{n}_1 + \ell + (k - 1)\bar{n}_0}; & \displaystyle g_{2jk}^3 = \sum_{\ell = 1}^3 \bar{E}_{\ell, n^\theta + j} g_{n^\varphi\bar{n}_1 + \ell + (k - 1)\bar{n}_0};\end{array}$\\
\For{$i = 3, \ldots, n^r$}{
$g_{(i-1)jk}^1 = g_{2 + j + (i - 3)n^\theta + (k-1)\bar{n}_1}$\;
$g_{ijk}^2 = g_{2 + (n^r - 2)n^\theta + j + (i - 3)n^\theta + (k - 1)\bar{n}_1}$\;
$g_{ijk}^3 = g_{n^\varphi\bar{n}_1 + 3 + j + (i - 3)n^\theta + (k - 1)\bar{n}_0}$\;
}}}
\end{algorithm}
\end{minipage}
\end{equation}
The geometric interpretation of this link between the DOFs in $V_1$ and in $\SSS^{p^r-1,p^\theta,p^\varphi}\times \SSS^{p^r,p^\theta-1,p^\varphi}\times \SSS^{p^r,p^\theta,p^\varphi-1}$ is pictured in Figure \ref{dof1}.
\begin{figure}
\centering
\scalebox{.9}{
\begin{tikzpicture}[scale=3.5]
\node at (2.15,1.25) {\includegraphics[width=.55\textwidth]{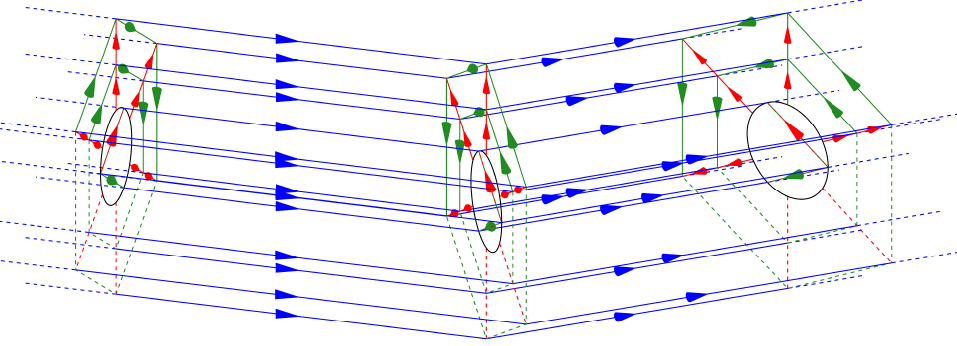}};
\node at (-0.5, -1.15) {
\begin{tikzpicture}[scale=3.5]
\foreach \x in{0,...,4}
{\foreach \y in{0,...,2}{
   \draw[red,->-] (\y*.25,\x*.25,1) -- (\y*.25+.25,\x*.25 ,1);
   \draw[red,->-] (\y*.25,1,\x*.25 ) -- (\y*.25+.25,1,\x*.25 );
}
\draw[red,dashed] (.75,\x*.25 ,1)--(1,\x*.25 ,1);
\draw[red,dashed] (.75,1,\x*.25)--(1,1,\x*.25);
}
\foreach \x in{0,...,4}
{\foreach \y in{0,...,2}{
\draw[green!50!black,->-] (\x*.25 ,\y*.25,1) -- (\x*.25 ,\y*.25+.25,1);
\draw[green!50!black,->-] (1,\y*.25,\x*.25 ) -- (1,\y*.25+.25,\x*.25 );
\draw[green!50!black,->-] (\x*.25,1,\y*.25 ) -- (\x*.25,1.25,\y*.25 );
\draw[green!50!black,->-] (\x*.25,1,\y*.25+.5) -- (\x*.25,1.25,\y*.25+.5);
}
\draw[green!50!black,dashed] (\x*.25 ,.75,1)--(\x*.25 ,1,1);
\draw[green!50!black,dashed] (1 ,.75,\x*.25)--(1,1,\x*.25);
}
\foreach \x in{0,...,4}
{\foreach \y in{-1,1,2,3}{
    \draw[blue,->-] (1,\x*.25 ,\y*.25+.25) -- (1,\x*.25 ,\y*.25);
    \draw[blue,->-] (\x*.25 ,1,\y*.25+.25) -- (\x*.25 ,1,\y*.25);
    }
\draw[blue,dashed] (\x*.25,1,0) -- (\x*.25,1,.25);    
\draw[blue,dashed] (1,\x*.25,0) -- (1,\x*.25,.25);
    }
\draw[red] (.125,0,1) node[below]{{\scriptsize $1,1,1$}};
\draw[red] (.375,0,1) node[below]{{\scriptsize $2,1,1$}};
\draw[red] (.625,0,1) node[below]{{\scriptsize $3,1,1$}};

\draw[green!50!black] (0,.125,1) node[left]{{\scriptsize $\hspace{-.05cm}1,1,1$}};
\draw[green!50!black] (0,.375,1) node[left]{{\scriptsize $\hspace{-.05cm}1,2,1$}};
\draw[green!50!black] (0,.625,1) node[left]{{\scriptsize $\hspace{-.05cm}1,3,1$}};
\draw[green!50!black] (0,1.125,1) node[left]{{\scriptsize $\hspace{-.05cm}1,n^\theta,1$}};

\draw[blue] (1,0,.825) node[right]{{\scriptsize $n^r,1,1~$}};
\draw[blue] (1,0,.6) node[right]{{\scriptsize $n^r,1,2~$}};
\draw[blue] (1,0,.325) node[right]{{\scriptsize $n^r,1,3~$}};
\draw[blue] (1,0,-.125) node[right]{{\scriptsize $n^r,1,n^\varphi$}};
\end{tikzpicture}};
\node at (2,-.25) {
\begin{tikzpicture}[scale=4]
\draw[green!50!black,->-] (.135,0) -- (-.067,.116);
\draw[red,dashed] (.135,0) -- (-.067,.116);
\draw[red,->-] (.135,0) -- (-.067,-.116);
\draw[green!50!black,dashed] (.135,0) -- (-.067,-.116);
\draw (0,.04) node{{\scriptsize \textcolor{red}{1}}};
\draw (0,.-.04) node{{\scriptsize \textcolor{red}{2}}};
\draw (0,0) circle (.135);
\draw (-.303,.175) -- (0,.35);
\draw (0,.35) -- (.303,.175);
\draw (-.303,-.175) -- (-.303,.175);
\draw[dashed] (-.303,-.175) -- (0,-.35) -- (.303,-.175) -- (.303,.175);
\draw (-.519,.3) -- (0,.6);
\draw (0,.6) -- (.519,.3);
\draw (-.519,-.3) -- (-.519,.3);
\draw[dashed] (-.519,-.3) -- (0,-.6) -- (.519,-.3) -- (.519,.3);
\draw[red,->-] (0,.135) --node[right]{{\scriptsize 3}} (0,.35);
\draw[red,->-] (0,.35) --node[right, yshift = -0.15cm]{{\scriptsize $n^\theta + 3$}} (0,.6);
\draw[red,dashed] (0,.7) -- (0,.6);
\draw[red,dashed] (0,-.135) -- (0,-.7);
\draw[red,->-] (-.116,.067) --node[yshift = 0.15cm, xshift = 0.15cm]{{\scriptsize 4}} (-.303,.175);
\draw[red,->-] (-.303,.175) --node[yshift = 0.15cm, xshift = 0.5cm]{{\scriptsize $n^\theta + 4$}} (-.519,.3);
\draw[red,dashed] (-.596,.345)--(-.519,.3);
\draw[red,->-] (.116,.067) --node[yshift = -0.15cm, xshift = 0.25cm]{{\scriptsize $n^\theta + 2$}} (.303,.175);
\draw[red,->-] (.303,.175) --node[yshift = -0.25cm, xshift = 0.25cm]{{\scriptsize $2n^\theta + 2$}}  (.519,.3);
\draw[red,dashed] (.596,.345)--(.519,.3);
\draw[red,->-] (-.116,-.067) --node[yshift = 0.25cm]{{\scriptsize 5}} (-.303,-.175);
\draw[red,->-] (-.303,-.175) --node[yshift = 0.35cm]{{\scriptsize $n^\theta + 5$}} (-.519,-.3);
\draw[red,dashed] (-.596,-.345)--(-.519,-.3);
\draw[red,dashed] (.596,-.345) -- (.116,-.067);
\node at (.425,-.65) {{\scriptsize $+(k-1)\bar{n}_1$}};
\end{tikzpicture}\hspace{.5cm}
\begin{tikzpicture}[scale=4]
\draw (0,0) circle (.135);
\draw[green!50!black,->-] (0,.35) -- node[xshift = -0.25, yshift = 0.25cm]{{\scriptsize\rotatebox{30}{2}}}(-.303,.175);
\draw[green!50!black,->-] (.303,.175) --node[xshift = 0.25, yshift =  0.25cm]{{\scriptsize \rotatebox{-30}{1}}} (0,.35);
\draw[green!50!black,->-] (-.303,.175) --node[left]{{\scriptsize \rotatebox{90}{3}}} (-.303,-.175);
\draw[green!50!black,dashed] (-.303,-.175) -- (0,-.35) -- (.303,-.175) -- (.303,.175);
\draw[green!50!black,->-]  (0,.6) --node[xshift = -0.25, yshift = 0.25cm]{{\scriptsize\rotatebox{30}{$n^\theta + 2$}}} (-.519,.3);
\draw[green!50!black,->-]  (.519,.3) --node[xshift = 0.25, yshift =  0.25cm]{{\scriptsize \rotatebox{-30}{$n^\theta + 1$}}} (0,.6);
\draw[green!50!black,->-] (-.519,.3) --node[left]{{\scriptsize \rotatebox{90}{$n^\theta + 3$}}} (-.519,-.3);
\draw[green!50!black,dashed] (-.519,-.3) -- (0,-.6) -- (.519,-.3) -- (.519,.3);
\draw (0,.135) -- (0,.35);
\draw (0,.35) -- (0,.6);
\draw[dashed] (0,.7) -- (0,.6);
\draw[dashed] (0,-.135) -- (0,-.7);
\draw (-.116,.067) -- (-.303,.175);
\draw (-.303,.175) -- (-.519,.3);
\draw[dashed] (-.596,.345)--(-.519,.3);
\draw (.116,.067) -- (.303,.175);
\draw (.303,.175) --  (.519,.3);
\draw[dashed] (.596,.345)--(.519,.3);
\draw (-.116,-.067) -- (-.303,-.175);
\draw (-.303,-.175) -- (-.519,-.3);
\draw[dashed] (-.596,-.345)--(-.519,-.3);
\draw[dashed] (.596,-.345) -- (.116,-.067);
\node at (.425,-.55) {{\scriptsize $+(k-1)\bar{n}_1$}};
\node at (.425,-.65) {{\scriptsize $+2 + (n^r - 2)n^\theta$}};
\end{tikzpicture}
};
\node at (2, -2) {
\begin{tikzpicture}[scale=4]
\draw (.303,.175) -- (0,.35) -- (-.303,.175) -- (-.303,-.175);
\draw[dashed] (-.303,-.175) -- (0,-.35) -- (.303,-.175) -- (.303,.175);
\draw (0,0) circle (.135);
\draw (.519,.3) -- (0,.6) -- (-.519,.3) -- (-.519,-.3);
\draw[dashed] (-.519,-.3) -- (0,-.6) -- (.519,-.3) -- (.519,.3);
\draw (0,.135) -- (0,.6);
\draw[dashed] (0,.7) -- (0,.6);
\draw[dashed] (0,-.135) -- (0,-.7);
\draw (-.519,.3) -- (-.116,.067);
\draw[dashed] (-.596,.345)--(-.519,.3);
\draw (.519,.3) -- (.116,.067);
\draw[dashed] (.596,.345)--(.519,.3);
\draw (-.519,-.3) -- (-.116,-.067);
\draw[dashed] (-.596,-.345)--(-.519,-.3);
\draw[dashed] (.596,-.345) -- (.116,-.067);
\fill[blue] (.135,0)node[left]{{\scriptsize 1}} circle (.02);
\fill[blue] (-.067, .116)node[below]{{\scriptsize \,\,\,\,2}} circle (.02); 
\fill[blue] (-.067, -.116)node[above]{{\scriptsize \,\,\,\,3}} circle (.02);
\fill[blue] (-.303,.175)node[above]{{\scriptsize 5}} circle (.02);
\fill[blue] (0,.35)node[above]{{\scriptsize \,\,\,\,4}} circle (.02);
\fill[blue] (.303,.175)node[right]{{\scriptsize $n^\theta+3$}} circle (.02);
\fill[blue] (-.303,-.175)node[below]{{\scriptsize 6}} circle (.02);
\fill[blue] (-.519,.3)node[left]{{\scriptsize $n^\theta+5\,\,$}} circle (.02);
\fill[blue] (0,.596)node[above]{{\scriptsize $n^\theta+4\qquad\quad$}} circle (.02);
\fill[blue] (.519,.3)node[right]{{\scriptsize $2n^\theta+3$}} circle (.02);
\fill[blue] (-.519,-.3)node[left]{{\scriptsize $n^\theta+6$}} circle (.02);
\node at (.425,-.55) {{\scriptsize $+n_1$}};
\node at (.425,-.65) {{\scriptsize $+(k-1)\bar{n}_0$}};
\end{tikzpicture}};
\node at (1.15,1.75) {$k-1$};
\node at (2.15,1.65) {$k$};
\node at (3,1.8) {$k+1$};
\draw[-stealth]  (.75,1.25) to [in=80,out=180]node[left]{$(\bbol{E}^{(1)})^T$} (-.5,-.15);
\draw[-stealth] (2.2,.7) to [out=270,in=90] (1.25,.5);
\draw[-stealth] (2.2,.7) to [out=270,in=90] (2.7,.5);
\draw[-stealth] plot [smooth] coordinates {(2.2,.7) (2.2,.5) (2,.25) (2, -1.15)};
\end{tikzpicture}}
\caption{Geometric interpretation of the DOFs in the gradient spaces $\SSS^{p^r-1,p^\theta,p^\varphi}\times\SSS^{p^r,p^\theta-1,p^\varphi}\times \SSS^{p^r,p^\theta,p^\varphi-1}$ (left) and $V_1$ (right). In the former space, the DOFs are associated to the oriented edges of a Greville tensor mesh $\cM$. In the latter they are associated to the oriented edges of a polygonal ring $\cR$, defined in Section \ref{geometricinterpretation} and partially represented here. In the middle figure we show the numbering of the edges corresponding to the DOFs in $V_1$ from the point of view of the $k$th joint, with $k \in \{1, \ldots, n^\varphi\}$. The edges moving from the $k$th joint to the $(k+1)$th joint are represented as dots in the central-bottom figure. When considering a gradient function in $V_1$, one can get its representation in terms of the basis functions in $\SSS^{p^r-1,p^\theta,p^\varphi}\times\SSS^{p^r,p^\theta-1,p^\varphi}\times \SSS^{p^r,p^\theta,p^\varphi-1}$ by applying the transpose of the matrix $\bbol{E}^{(1)}$ defined in \eqref{eq:E1}. More precisely, $(\bbol{E}^{(1)})^T$ links the DOFs corresponding to poloidal, radial and toroidal edges in $\cR$ with the DOFs relative to the edges in the $r,\theta$ and $\varphi$ directions, respectively, of the tensor mesh $\cM$. The first two edges in a joint of $\cR$ are neither radial nor poloidal. They are linked to edges of $\cM$ both in the $r$ and in the $\theta$ directions.}\label{dof1}
\end{figure}
\begin{prop}\label{commutation1}
For every set of DOFs $\{g_{ijk}^1\}_{i,j,k=1}^{n^r-1,n^\theta,n^\varphi}, \{g_{ijk}^2\}_{i,j,k=1}^{n^r,n^\theta,n^\varphi}, \{g_{ijk}^3\}_{i,j,k=1}^{n^r,n^\theta,n^\varphi}$ in $\im(\grad) \subseteq\SSS^{p^r-1,p^\theta,p^\varphi}\times\SSS^{p^r,p^\theta-1,p^\varphi}\times \SSS^{p^r,p^\theta,p^\varphi-1}$, there exist a set of DOFs $\{f_{ijk}\}_{i,j,k=1}^{n^r,n^\theta,n^\varphi}$ in $\SSS^{p^r,p^\theta,p^\varphi}$, a set of DOFs $\{f_\ell\}_{\ell=1}^{n_0}$ in $V_0$ and a collection of numbers $\{g_\ell\}_{\ell=1}^{n_1}$ such that Equation \eqref{dof1relation} holds true, for which the following diagram commutes:
$$
\begin{tikzpicture}
\matrix (m)[matrix of math nodes,column sep=4em,row sep=3em]{
\{f_{ijk}\}_{i,j,k} \pgfmatrixnextcell \left\{\begin{array}{c}g_{ijk}^1\\\\g_{ijk}^2\\\\g_{ijk}^3\end{array}\right\}_{i,j,k}\\ \{f_\ell\}_{\ell} \pgfmatrixnextcell \{g_\ell\}_{\ell}\\
};
\draw[-stealth] (m-1-1) --node[above]{$\begin{bmatrix} \bbol{D}^{(1,0,0)}\\ \bbol{D}^{(0,1,0)}\\ \bbol{D}^{(0,0,1)}\end{bmatrix}$} (m-1-2);
\draw[-stealth] (m-2-1) --node[below]{$\bbol{D}^{(0)}$} (m-2-2);
\draw[-stealth] (m-2-1) --node[left]{$(\bbol{E}^{(0)})^T$} (m-1-1);
\draw[-stealth] (m-2-2) --node[right]{$(\bbol{E}^{(1)})^T$} (m-1-2);
\draw[-stealth, dashed] plot [smooth] coordinates {(-1.75, -1.5) (-1.25, 0.35) (0.4, 0.65)};
\draw[-stealth, dashed] plot [smooth] coordinates {(-1.35, -1.65) (0.75, -1.5) (1.25, -0.4)};
\end{tikzpicture}
$$
\end{prop}
\begin{proof} 
For all $\ell, i,j,k$, we make a loop of identities from $g_{ijk}^\ell$ back to itself that can hold only if the diagram is commutative. Let us start with $\ell = 1$. For $i = 1$, by using Equation \eqref{dof1img}, \eqref{eq:fijk}, \eqref{eq:gl}, \eqref{eq:gijk} and the fact that matrix $\bbol{E}^{(0)}$ is DTA-compatible, so that in particular its columns sum to one, we have 
\begin{equation*}
\begin{split}
g_{1jk}^1 &\stackrel{\bbol{D}^{(1,0,0)}}{=} f_{2jk} - f_{1jk}\\
&\stackrel{(\bbol{E}^{(0)})^T}{=} \sum_{\ell=1}^3 (\bar{E}_{\ell,n^\theta + j} - \bar{E}_{\ell, j})f_{\ell+(k-1)\bar{n}_0}\\
&= \sum_{\ell=1}^3 (\bar{E}_{\ell,n^\theta + j} - \bar{E}_{\ell, j})f_{\ell+(k-1)\bar{n}_0} + f_{1+(k-1)\bar{n}_0}-f_{1+(k-1)\bar{n}_0}\\
&\stackrel{\text{DTA}}{=} \sum_{\ell=1}^3 (\bar{E}_{\ell,n^\theta + j} - \bar{E}_{\ell, j})f_{\ell+(k-1)\bar{n}_0} + \sum_{\ell=1}^3 \bar{E}_{\ell,j} f_{1+(k-1)\bar{n}_0} - \sum_{\ell=1}^3 \bar{E}_{\ell,n^\theta + j} f_{1+(k-1)\bar{n}_0}\\ 
&=\sum_{\ell=1}^2 \bar{E}_{\ell, j}^{(1), L}(f_{(\ell+1)+(k-1)\bar{n}_0} - f_{1+(k-1)\bar{n}_0})\\
&\stackrel{\bbol{D}^{(0)}}{=} \sum_{\ell=1}^2 \bar{E}_{\ell, j}^{(1), L}g_{\ell+(k-1)\bar{n}_1}\stackrel{(\bbol{E}^{(1)})^T}{=} g_{1jk}^1.
\end{split}
\end{equation*}
For $i = 2$, it holds
\begin{equation*}
\begin{split}
g_{2jk}^1 &\stackrel{\bbol{D}^{(1,0,0)}}{=} f_{3jk} - f_{2jk}\\
&\stackrel{(\bbol{E}^{(0)})^T}{=}  f_{3+j+(k-1)\bar{n}_0} -\sum_{\ell=1}^3 \bar{E}_{\ell,n^\theta + j} f_{\ell+(k-1)\bar{n}_0}\\
&\stackrel{\bbol{D}^{(0)}}{=} g_{2+j+(k-1)\bar{n}_1} \stackrel{(\bbol{E}^{(1)})^T}{=}g_{2jk}^1
\end{split}
\end{equation*}
and for $i =3, \ldots, n^r - 1$,
\begin{equation*}
\begin{split}
g_{ijk}^1 &\stackrel{\bbol{D}^{(1,0,0)}}{=} f_{(i+1)jk} - f_{ijk}\\
&\stackrel{(\bbol{E}^{(0)})^T}{=}  f_{3+j+(i-2)n^\theta+(k-1)\bar{n}_0} -f_{3+j+(i-3)n^\theta+(k-1)\bar{n}_0}\\
&\stackrel{\bbol{D}^{(0)}}{=} g_{2+j+(i - 2)n^\theta+(k-1)\bar{n}_1} \stackrel{(\bbol{E}^{(1)})^T}{=} g_{ijk}^1.
\end{split}
\end{equation*}
When $\ell = 2$ and $i \leq 2$, we use again that $\bbol{E}^{(0)}$ is DTA-compatible, so that
\begin{equation*}
\begin{split}
g_{ijk}^2 &\stackrel{\bbol{D}^{(0,1,0)}}{=} f_{i(j+1)k} - f_{ijk}\\
&\stackrel{(\bbol{E}^{(0)})^T}{=} \sum_{\ell=1}^3 (\bar{E}_{\ell,(i-1)n^\theta + j + 1} - \bar{E}_{(i-1)n^\theta + j})f_{\ell+(k-1)\bar{n}_0}\\
&= \sum_{\ell=1}^3 (\bar{E}_{\ell,(i-1)n^\theta + j + 1} - \bar{E}_{(i-1)n^\theta + j})f_{\ell+(k-1)\bar{n}_0} + f_{1+(k-1)\bar{n}_0}-f_{1+(k-1)\bar{n}_0}\\
&\stackrel{\text{DTA}}{=} \sum_{\ell=1}^3 (\bar{E}_{\ell,(i-1)n^\theta + j + 1} - \bar{E}_{(i-1)n^\theta + j})f_{\ell+(k-1)\bar{n}_0} + \sum_{\ell=1}^3 \bar{E}_{\ell,(i-1)n^\theta + j} f_{1+(k-1)\bar{n}_0} - \sum_{\ell=1}^3 \bar{E}_{\ell,(i-1)n^\theta + j + 1} f_{1+(k-1)\bar{n}_0}\\ 
&=\sum_{\ell=1}^2 \bar{E}_{\ell,(i-1)n^\theta + j}^{(1), R}(f_{(\ell+1)+(k-1)\bar{n}_0} - f_{1+(k-1)\bar{n}_0})\\
&\stackrel{\bbol{D}^{(0)}}{=} \sum_{\ell=1}^2 \bar{E}_{\ell,(i-1)n^\theta + j}^{(1), R}g_{\ell+(k-1)\bar{n}_1} \stackrel{(\bbol{E}^{(1)})^T}{=} g_{ijk}^2.
\end{split}
\end{equation*}
Instead, for $i = 3, \ldots, n^r$ we simply have
\begin{equation*}
\begin{split}
g_{ijk}^2 &\stackrel{\bbol{D}^{(0,1,0)}}{=} f_{i(j+1)k} - f_{ijk}\\
&\stackrel{(\bbol{E}^{(0)})^T}{=}  f_{3+(j+1)+(i-3)n^\theta+(k-1)\bar{n}_0} - f_{3+j+(i-3)n^\theta+(k-1)\bar{n}_0}\\
&\stackrel{\bbol{D}^{(0)}}{=} g_{2+(n^r-2)n^\theta +j + (i-3)n^\theta + (k-1)\bar{n}_1} \stackrel{(\bbol{E}^{(1)})^T}{=}g_{ijk}^2.
\end{split}
\end{equation*}
Finally, when $\ell = 3$ and $i \leq 2$ it holds
\begin{equation*}
\begin{split}
g_{ijk}^3 &\stackrel{\bbol{D}^{(0,0,1)}}{=} f_{ij(k+1)} - f_{ijk}\\
&\stackrel{(\bbol{E}^{(0)})^T}{=}  \sum_{\ell=1}^3 \bar{E}_{\ell,(i-1)n^\theta + j} (f_{\ell+k\bar{n}_0}-f_{\ell+(k-1)\bar{n}_0})\\
&\stackrel{\bbol{D}^{(0)}}{=} \sum_{\ell=1}^3 \bar{E}_{\ell,(i-1)n^\theta + j}g_{n^\varphi\bar{n}_1 + \ell + (k-1)\bar{n}_0} \stackrel{(\bbol{E}^{(1)})^T}{=} g_{ijk}^3,
\end{split}
\end{equation*}
and for $i \geq 3$ we have
\begin{equation*}
\begin{split}
g_{ijk}^3 &\stackrel{\bbol{D}^{(0,0,1)}}{=} f_{ij(k+1)} - f_{ijk}\\
&\stackrel{(\bbol{E}^{(0)})^T}{=}  f_{3+j + (i-3)n^\theta+k\bar{n}_0} -f_{3+ j +(i-3)n^\theta+(k-1)\bar{n}_0}\\
&\stackrel{\bbol{D}^{(0)}}{=} g_{n^\varphi\bar{n}_1 +3+j+(i-3)n^\theta+(k-1)\bar{n}_0} \stackrel{(\bbol{E}^{(1)})^T}{=}g_{ijk}^3.
\end{split}
\end{equation*}
\end{proof}
We can finally define the basis of $V_1$. First, we introduce the following collections of $n_1$ functions,
\begin{equation}\label{eq:N100}
\m{N}^{(1, 0, 0)} \defeq \bbol{E}^{(1)}\begin{bmatrix}
\m{B}^{(1, 0, 0)}\\ \pmb{0}\\ \pmb{0}
\end{bmatrix}, \qquad
\m{N}^{(0, 1, 0)} \defeq \bbol{E}^{(1)}\begin{bmatrix}
\pmb{0} \\ \m{B}^{(0, 1, 0)}\\ \pmb{0}
\end{bmatrix}, \qquad
\m{N}^{(0, 0, 1)} \defeq \bbol{E}^{(1)}\begin{bmatrix}
\pmb{0}\\ \pmb{0}\\ \m{B}^{(0, 0, 1)}
\end{bmatrix}.
\end{equation}
We highlight that $\m{N}^{(1, 0, 0)}, \m{N}^{(0, 1, 0)}$ and $\m{N}^{(0, 0, 1)}$ contain zero functions. However, there exists no index $\ell$ in $\{1, \ldots, n_1\}$ for which $N_\ell^{(1, 0, 0)}, N_\ell^{(0, 1, 0)}$ and $N_\ell^{(0, 0, 1)}$ are all zero functions. This can be easily seen form the structure of $\bbol{E}^{(1)}$, reported in Equation \eqref{eq:E1}.
\begin{prop}\label{prop:linind1}
The non-zero functions in $\m{N}^{(1, 0, 0)}, \m{N}^{(0, 1, 0)}$ and $\m{N}^{(0, 0, 1)}$ defined in Equation \eqref{eq:N100} are linearly independent.
\end{prop}
\begin{proof}
The functions in $\m{B}^{(1,0,0)}, \m{B}^{(0,1,0)}$ and $\m{B}^{(0,0,1)}$ are linearly independent. Furthermore, one can show, see \cite{deepesh}, that the non-zero rows in $\bbol{E}^{(0)}, \bbol{E}^{(1), L}$ and $\bbol{E}^{(1), R}$ are linearly independent, which implies that matrix $\bbol{E}^{(1)}$ has linearly independent non-zero rows as well, because of Equation \eqref{eq:E1}. By Equation \eqref{eq:N100}, the non-zero functions in $\m{N}^{(1,0,0)},\m{N}^{(0,1,0)}$ and $\m{N}^{(0,0,1)}$ are linearly independent.
\end{proof}
We define the space $V_1$ as the span of the following linearly independent vector functions \begin{equation}\label{N1}
N_\ell^{(1)} := \left(\begin{array}{c}N_\ell^{(1,0,0)}\\\\N_\ell^{(0,1,0)}\\\\N_\ell^{(0,0,1)}\end{array}\right)\quad \text{for }\ell=1,\ldots, n_1.
\end{equation}
Thanks to Remark \ref{linearalgebraremark}, we can rewrite Proposition \ref{commutation1} as follows.
\begin{cor}\label{cor1}
The diagram 
$$
\begin{tikzpicture}
\matrix (m)[matrix of math nodes,column sep=4em,row sep=3em]{
\SSS^{p^r,p^s,p^t} \pgfmatrixnextcell \SSS^{p^r-1,p^s,p^t}\times\SSS^{p^r,p^s-1,p^t}\times\SSS^{p^r,p^s,p^t-1}\\ V_0 \pgfmatrixnextcell V_1\\
};

\draw[-stealth] (m-1-1) --node[above]{$\grad$} (m-1-2);
\draw[-stealth] (m-2-1) --node[below]{$\grad$} (m-2-2);
\draw[-stealth] (m-2-1) --node[left]{$\id$} (m-1-1);
\draw[-stealth] (m-2-2) --node[right]{$\id$} (m-1-2);
\draw[-stealth, dashed] plot [smooth] coordinates {(-3.5, -0.65) (-3, 0.5) (-1.45, 0.65)};
\draw[-stealth, dashed] plot [smooth] coordinates {(-3.25, -0.75) (0.5, -0.65) (1.25, 0.5)};
\end{tikzpicture}
$$
commutes and the pushforward of the functions in $V_1$, namely $\cF^1(N_\ell^{(1)})$ for $\ell=1, \ldots, n_1$, with the operator $\cF^1$ defined as in \eqref{pushforward}, are $C^0$ on $\Omega^{pol}$.
\end{cor}
\begin{proof} The commutation is drawn from Proposition \ref{commutation1} and Remark \ref{linearalgebraremark}. For what concerns the smoothness of the pushforwarded functions, the $C^0$-continuity is implied by their local exactness at the polar curve.
\end{proof}
Therefore, the representation of the gradient of a function $f \in V_0$, $f = \m{N}^{(0)} \cdot \m{f}$, in the basis of $V_1$ is given by
$$
\grad\,f = \m{N}^{(1)} \cdot \bbol{D}^{(0)}\m{f} \in V_1.
$$
\subsection{The curl of functions in $V_1$ and the space $V_2$}\label{curl}
Let $\pmb{g} \in \SSS^{p^r-1,p^\theta,p^\varphi}\times\SSS^{p^r,p^\theta-1,p^\varphi}\times\SSS^{p^r,p^\theta,p^\varphi-1}$, $$
\pmb{g} = \left(\begin{array}{c}g^1\\\\g^2\\\\g^3\end{array}\right) = \left(\begin{array}{c}\m{B}^{(1,0,0)} \cdot \m{g}^1\\\\\m{B}^{(0,1,0)} \cdot \m{g}^2\\\\\m{B}^{(0,0,1)} \cdot \m{g}^3\end{array}\right)
$$
with $\m{g}^1, \m{g}^2, \m{g}^3$ the vectorizations of the coefficients $\{g_{ijk}^1\}_{i,j,k=1}^{n^r-1,n^\theta,n^\varphi}, \{g_{ijk}^2\}_{i,j,k=1}^{n^r,n^\theta,n^\varphi}, \{g_{i,j,k}^3\}_{i,j,k=1}^{n^r,n^\theta,n^\varphi}$ obtained by running the indices first on $j$, then on $i$ and then on $k$. The curl of $\pmb{g}$ has the following expression:
$$
\curl\, \pmb{g} = \left(\begin{array}{c}\displaystyle\sum_{k=1}^{n^\varphi}\sum_{j=1}^{n^\theta} \sum_{i=1}^{n^r} (g_{i(j+1)k}^3-g_{ijk}^3 + g_{ijk}^2 - g_{ij(k+1)}^2) B_i^{p^r}\per{D}_j^{p^\theta} \per{D}_k^{p^\varphi}\\\\
\displaystyle\sum_{k=1}^{n^\varphi}\sum_{j=1}^{n^\theta} \sum_{i=1}^{n^r-1} (g_{ij(k+1)}^1-g_{ijk}^1 + g_{ijk}^3 - g_{(i+1)jk}^3) D_{i}^{p^r}\per{B}_j^{p^\theta} \per{D}_k^{p^\varphi}\\\\
\displaystyle\sum_{k=1}^{n^\varphi}\sum_{j=1}^{n^\theta} \sum_{i=1}^{n^r-1} (g_{(i+1)jk}^2-g_{ijk}^2 + g_{ijk}^1 - g_{i(j+1)k}^1) D_{i}^{p^r}\per{D}_{j}^{p^\theta}\per{B}_k^{p^\varphi}
\end{array}\right),
$$
where the functions $D_i^{p^r}, \per{D}_j^{p^s}, \per{D}_k^{p^t}$, for any $i,j,k$, have been defined in Section \ref{preliminaries}. Let $\m{B}^{(0,1,1)}, \m{B}^{(1,0,1)}$ and $\m{B}^{(1,1,0)}$ be the vectorizations of $\{B_i^{p^r}\per{D}_j^{p^\theta}\per{D}_k^{p^\varphi}\}_{i,j,k=1}^{n^r,n^\theta,n^\varphi}, \{D_i^{p^r}\per{B}_j^{p^\theta}\per{D}_k^{p^\varphi}\}_{i,j,k=1}^{n^r-1,n^\theta,n^\varphi}$ and $\{D_i^{p^r}\per{D}_j^{p^\theta}\per{B}_k^{p^\varphi}\}_{i,j,k=1}^{n^r-1,n^\theta,n^\varphi}$, respectively, running the indices on $j$, then on $i$ and then on $k$. The curl operator $\curl:\SSS^{p^r-1,p^\theta,p^\varphi}\times\SSS^{p^r,p^\theta-1,p^\varphi}\times\SSS^{p^r,p^\theta,p^\varphi-1} \to \SSS^{p^r,p^\theta-1,p^\varphi-1}\times\SSS^{p^r-1,p^\theta,p^\varphi-1}\times\SSS^{p^r-1,p^\theta-1,p^\varphi}$ has the following matrix representation:
$$
\curl\,\pmb{g} = \left(\begin{array}{c}\m{B}^{(0,1,1)} \cdot [\bbol{D}^{(0,1,0)} \m{g}^3-\bbol{D}^{(0,0,1)}\m{g}^2]\\\\\m{B}^{(1,0,1)} \cdot [\bar{\bbol{D}}^{(0,0,1)}\m{g}^1 - \bbol{D}^{(1,0,0)}\m{g}^3]\\\\\m{B}^{(1,1,0)} \cdot [\bbol{D}^{(1,0,0)} \m{g}^2 - \bar{\bbol{D}}^{(0,1,0)}\m{g}^1]\end{array}\right) \quad \forall\,\pmb{g} \in \SSS^{p^r-1,p^\theta,p^\varphi}\times\SSS^{p^r,p^\theta-1,p^\varphi}\times\SSS^{p^r,p^\theta,p^\varphi-1},
$$
where the matrices $\bbol{D}^{(1,0,0)}, \bbol{D}^{(0,1,0)}$ and $\bbol{D}^{(0,0,1)}$ have been defined in Equation \eqref{D100}, while $\bar{\bbol{D}}^{(0, 1, 0)}$ and $\bar{\bbol{D}}^{(0,0,1)}$ have the same definitions of $\bbol{D}^{(0,1,0)}$ and $\bbol{D}^{(0,0,1)}$, respectively, when $\bbol{I}_{n^r}$ is replaced with $\bbol{I}_{n^r-1}$. Hence, the DOFs in $\im(\curl)$ have form 
\begin{equation}\label{eq:tensorcurl}
\left\{\begin{array}{ll}
h_{ijk}^1 \defeq g_{i(j+1)k}^3-g_{ijk}^3 + g_{ijk}^2 - g_{ij(k+1)}^2 & \text{for }i=1, \ldots, n^r;~j=1, \ldots, n^\theta;~k=1, \ldots, n^\varphi;\\\\
h_{ijk}^2 \defeq g_{ij(k+1)}^1-g_{ijk}^1 + g_{ijk}^3 - g_{(i+1)jk}^3 & \text{for }i=1, \ldots, n^r-1;~j=1, \ldots, n^\theta;~k=1, \ldots, n^\varphi;\\\\
h_{ijk}^3 \defeq g_{(i+1)jk}^2-g_{ijk}^2 + g_{ijk}^1 - g_{i(j+1)k}^1 & \text{for }i=1, \ldots, n^r-1;~j=1, \ldots, n^\theta;~k=1, \ldots, n^\varphi;
\end{array}\right.
\end{equation}
for some set of DOFs $\{g_{ijk}^1\}_{i,j,k=1}^{n^r-1,n^\theta,n^\varphi}$, $\{g_{ijk}^2\}_{i,j,k=1}^{n^r,n^\theta,n^\varphi}$, $\{g_{ijk}^3\}_{i,j,k=1}^{n^r,n^\theta,n^\varphi}$ in $\SSS^{p^r-1,p^\theta,p^\varphi}\times\SSS^{p^r,p^\theta-1,p^\varphi}\times\SSS^{p^r,p^\theta,p^\varphi-1}$. We have seen the geometric interpretations of $g_{ijk}^1, g_{ijk}^2$ and $g_{ijk}^3$ as oriented edges in a tensor mesh $\cM$. Furthermore, each collection of four edges invoked in the expressions of $h_{ijk}^1, h_{ijk}^2$ and $h_{ijk}^3$ frames a face of $\cM$. On the other hand, the DOFs in $\SSS^{p^r,p^\theta-1,p^\varphi-1}\times\SSS^{p^r-1,p^\theta,p^\varphi-1}\times\SSS^{p^r-1,p^\theta-1,p^\varphi}$ are the oriented faces of $\cM$, in this geometric interpretation. From this point of view, the curl operator maps the oriented edges framing a face to such face equipped with the orientation induced by the edges, as shown in Figure \ref{dof2}.
\begin{figure}
\centering
\begin{tikzpicture}[scale=3.5]
\foreach \x in{0,...,4}
{\foreach \y in{0,...,2}{
   \draw[red,->-] (\y*.25,\x*.25,1) -- (\y*.25+.25,\x*.25 ,1);
   \draw[red,->-] (\y*.25,1,\x*.25 ) -- (\y*.25+.25,1,\x*.25 );
}
\draw[red,dashed] (.75,\x*.25 ,1)--(1,\x*.25 ,1);
\draw[red,dashed] (.75,1,\x*.25)--(1,1,\x*.25);
}
\foreach \x in{0,...,4}
{\foreach \y in{0,...,2}{
\draw[green!50!black,->-] (\x*.25 ,\y*.25,1) -- (\x*.25 ,\y*.25+.25,1);
\draw[green!50!black,->-] (1,\y*.25,\x*.25 ) -- (1,\y*.25+.25,\x*.25 );
\draw[green!50!black,->-] (\x*.25,1,\y*.25 ) -- (\x*.25,1.25,\y*.25 );
\draw[green!50!black,->-] (\x*.25,1,\y*.25+.5) -- (\x*.25,1.25,\y*.25+.5);
}
\draw[green!50!black,dashed] (\x*.25 ,.75,1)--(\x*.25 ,1,1);
\draw[green!50!black,dashed] (1 ,.75,\x*.25)--(1,1,\x*.25);
}
\foreach \x in{0,...,4}
{\foreach \y in{-1,1,2,3}{
    \draw[blue,->-] (1,\x*.25 ,\y*.25+.25) -- (1,\x*.25 ,\y*.25);
    \draw[blue,->-] (\x*.25 ,1,\y*.25+.25) -- (\x*.25 ,1,\y*.25);
    }
\draw[blue,dashed] (\x*.25,1,0) -- (\x*.25,1,.25);    
\draw[blue,dashed] (1,\x*.25,0) -- (1,\x*.25,.25);
    }
\draw[red] (.125,0,1) node[below]{{\scriptsize $1,1,1$}};
\draw[red] (.375,0,1) node[below]{{\scriptsize $2,1,1$}};
\draw[red] (.625,0,1) node[below]{{\scriptsize $3,1,1$}};

\draw[green!50!black] (0,.125,1) node[left]{{\scriptsize $\hspace{-.05cm}1,1,1$}};
\draw[green!50!black] (0,.375,1) node[left]{{\scriptsize $\hspace{-.05cm}1,2,1$}};
\draw[green!50!black] (0,.625,1) node[left]{{\scriptsize $\hspace{-.05cm}1,3,1$}};
\draw[green!50!black] (0,1.125,1) node[left]{{\scriptsize $\hspace{-.05cm}1,n^\theta,1$}};

\draw[blue] (1,0,.825) node[right]{{\scriptsize $n^r,1,1~$}};
\draw[blue] (1,0,.6) node[right]{{\scriptsize $n^r,1,2~$}};
\draw[blue] (1,0,.325) node[right]{{\scriptsize $n^r,1,3~$}};
\draw[blue] (1,0,-.125) node[right]{{\scriptsize $n^r,1,n^\varphi$}};
\draw[-stealth] (1.2,.5) --node[above]{\scalebox{.65}{$\left[\hspace{-.15cm}\begin{array}{ccc} \textcolor{red}{0} &\hspace{-.15cm} \textcolor{green!50!black}{-\bbol{D}^{(0,0,1)}} &\hspace{-.15cm} \textcolor{blue}{\bbol{D}^{(0,1,0)}}\\ 
\textcolor{red}{\bar{\bbol{D}}^{(0,0,1)}} &\hspace{-.15cm} \textcolor{green!50!black}{0} &\hspace{-.15cm} \textcolor{blue}{-\bbol{D}^{(1,0,0)}}\\
\textcolor{red}{-\bar{\bbol{D}}^{(0,1,0)}} &\hspace{-.15cm} \textcolor{green!50!black}{\bbol{D}^{(1,0,0)}} &\hspace{-.15cm} \textcolor{blue}{0}
\end{array}\hspace{-.15cm}\right]\hspace{-.15cm} \begin{array}{c}\textcolor{cyan!90!black}{\blacksquare}\\\textcolor{magenta!90!black}{\blacksquare}\\\textcolor{orange!90!black}{\blacksquare}\end{array}$}} (2.1,.5);
\draw[-stealth] (2,0,1) --node[below]{$r$} (2.3,0,1);
\draw[-stealth] (2,0,1) --node[left]{$\theta$} (2,.3,1);
\draw[-stealth](2,0,1) -- (2,0,.5) node[above]{$\varphi$};
\end{tikzpicture}\hspace{-.25cm}
\begin{tikzpicture}[scale=3.5]
\foreach \x in{0,...,4}{
\draw (0,\x*.25,1) -- (.75,\x*.25 ,1);
\draw[dashed] (.75,\x*.25 ,1)--(1,\x*.25 ,1);
\draw (\x*.25 ,0,1) -- (\x*.25 ,.75,1);
\draw[dashed] (\x*.25 ,.75,1)--(\x*.25 ,1,1);
\draw (1,\x*.25 ,1) -- (1,\x*.25 ,.25);
\draw[dashed] (1,\x*.25 ,.25)--(1,\x*.25 ,0);
\draw (\x*.25 ,1,1) -- (\x*.25 ,1,.25);
\draw[dashed] (\x*.25 ,1,.25)--(\x*.25 ,1,0);
\draw (1,0,\x*.25 ) -- (1,.75,\x*.25 );
\draw[dashed] (1,.75,\x*.25 )--(1,1,\x*.25 );
\draw (0,1,\x*.25 ) -- (.75,1,\x*.25 );
\draw[dashed] (.75,1,\x*.25 )--(1,1,\x*.25 );
}
\foreach \x in {0,1,2}{
\foreach \y in {0,1,2}{
\node at (\x*.25+.125,\y*.25+.125,1) {\textcolor{orange!90!black}{\acpatch}};
\node at (1,\x*.25+.125,1-\y*.25 -.075) {{\scriptsize\rotslant{45}{45}{\textcolor{cyan!90!black}{\cpatch}}}};
\node at (\x*.25+.1,1,1-\y*.25-.13) {{\scriptsize \rotslant{-30}{45}{\textcolor{magenta!90!black}{\cpatch}}}};
}}

\draw[orange!90!black] (.375,0,1) node[below]{{\scriptsize $2,1,1$}};
\draw[orange!90!black] (.625,0,1) node[below]{{\scriptsize $3,1,1$}};
\draw[orange!90!black] (-.125,.175,1) node[below]{{\scriptsize $1,1,1$}};
\draw[orange!90!black] (-.125,.425,1) node[below]{{\scriptsize $1,2,1$}};
\draw[orange!90!black] (-.125,.675,1) node[below]{{\scriptsize $1,3,1$}};
\draw[magenta!90!black] (-.25,1,.1) node[below]{{\scriptsize $1,n^\theta$,3}};
\draw[magenta!90!black] (-.25,1,.4) node[below]{{\scriptsize $1,n^\theta$,2}};
\draw[magenta!90!black] (-.25,1,.7) node[below]{{\scriptsize $1,n^\theta$,1}};
\draw[magenta!90!black] (.35,.99,1) node[below]{{\scriptsize $2,n^\theta$,1}};
\draw[magenta!90!black] (.65,.99,1) node[below]{{\scriptsize $3,n^\theta$,1}};
\draw[cyan!90!black] (.875,.125,1) node{{\scriptsize $n^r,1,1$}};
\draw[cyan!90!black] (1.175,0,.675) node{{\scriptsize $n^r,1,2$}};
\draw[cyan!90!black] (1.175,0,.475) node{{\scriptsize $n^r,1,3$}};
\draw[cyan!90!black] (.875,.375,1) node{{\scriptsize $n^r,2,1$}};
\draw[cyan!90!black] (.875,.6,1) node{{\scriptsize $n^r,3,1$}};
\end{tikzpicture}
\caption{The action of the curl operator when associating the DOFs in $\SSS^{p^r-1,p^\theta,p^\varphi}\times\SSS^{p^r,p^\theta-1,p^\varphi}\times\SSS^{p^r,p^\theta,p^\varphi-1}$ to the oriented edges of a tensor mesh $\cM$, as shown on the left side of the figure. In $\SSS^{p^r,p^\theta-1,p^\varphi-1}\times\SSS^{p^r-1,p^\theta,p^\varphi-1}\times\SSS^{p^r-1,p^\theta-1,p^\varphi}$ the DOFs become oriented faces of $\cM$ and the curl action on the DOFs consists in mapping the edges framing a face to such face, in this geometric interpretation. In particular, the first components of the DOFs in $\SSS^{p^r,p^\theta-1,p^\varphi-1}\times\SSS^{p^r-1,p^\theta,p^\varphi-1}\times\SSS^{p^r-1,p^\theta-1,p^\varphi}$ correspond to faces parallel to the $(\theta,\varphi)$-plane, the seconds to faces parallel to the $(r,\varphi)$-plane and the third components to faces parallel to the $(r,\theta)$-plane. Because of the periodicity of the spaces in the $\theta$ and $\varphi$ directions we should have outwards pointing faces on the top and back sides of the tensor mesh $\cM$, as we do for the edges. We have chosen to omit such faces for the readability of the figure.}\label{dof2}
\end{figure}

Let now $\pmb{g} \in V_1$, $\pmb{g}=\m{N}^{(1)} \cdot \m{g}$ for a set of coefficients $\m{g} = (g_1, \ldots, g_{n_1})^T$. In order to simplify the notation in what follows, let us introduce $\bar{n}_2 \defeq \bar{n}_0 -3 = n^\theta(n^r - 2)$. We now define the action of the curl on the DOFs in $V_1$. We want it to mimic what we had for the DOFs in $\SSS^{p^r-1,p^\theta,p^\varphi}\times\SSS^{p^r,p^\theta-1,p^\varphi}\times\SSS^{p^r,p^\theta,p^\varphi-1}$. If in the latter, the geometric interpretation of such action consisted in mapping the edges to the faces of a tensor mesh $\cM$, this time the curl should map the edges to the faces of the polygonal ring $\cR$ introduced in Section \ref{geometricinterpretation}, as shown in Figure \ref{curldof2}. The total number of faces in $\cR$ is $n_2 \defeq n^\varphi(\bar{n}_1 + \bar{n}_2) = n^\varphi(3n^\theta(n^r - 2) + 2)$, as the number of faces in a joint of $\cR$ is $\bar{n}_2$ and the number of faces in each side is $\bar{n}_1$, which is equal to the number of edges in a joint. 
\begin{figure}
\centering
\begin{tikzpicture}
\node at (0,0) {\includegraphics[width=.5\textwidth]{dof11}};
\node at (5.75,0) {\includegraphics[width=.1\textwidth]{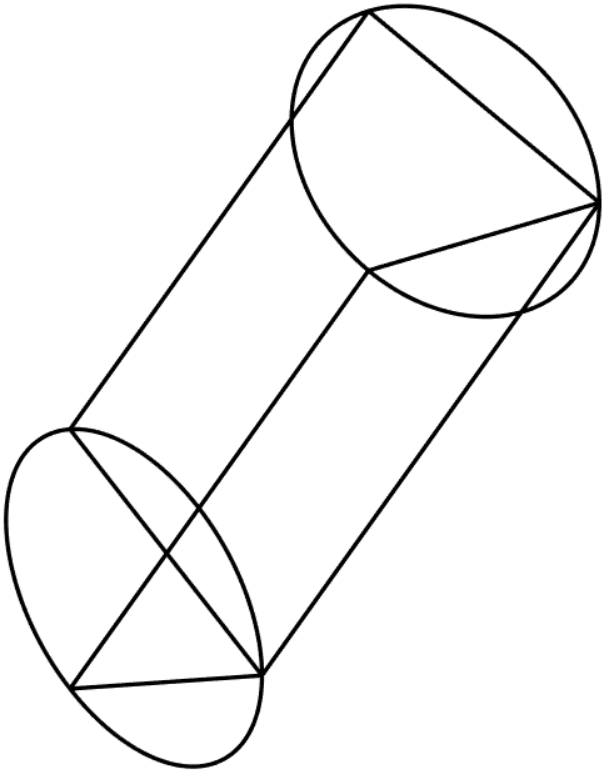}};
\node at (7.75,0) {\includegraphics[width=.1\textwidth]{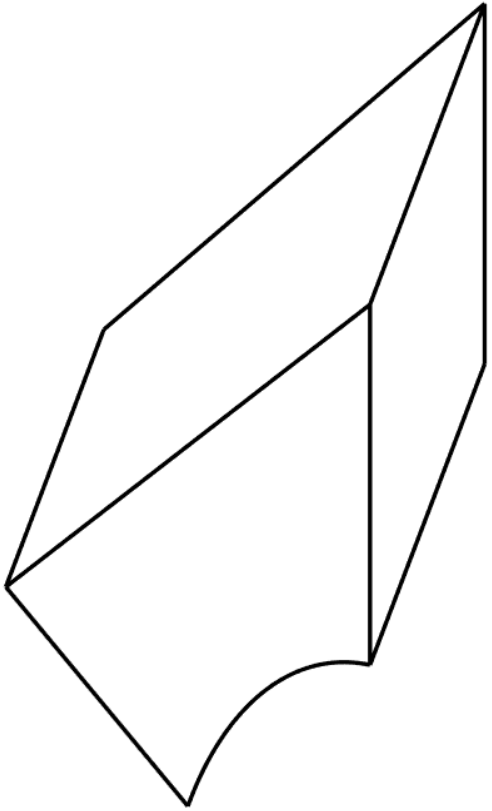}};
\node at (9.75,0) {\includegraphics[width=.1\textwidth]{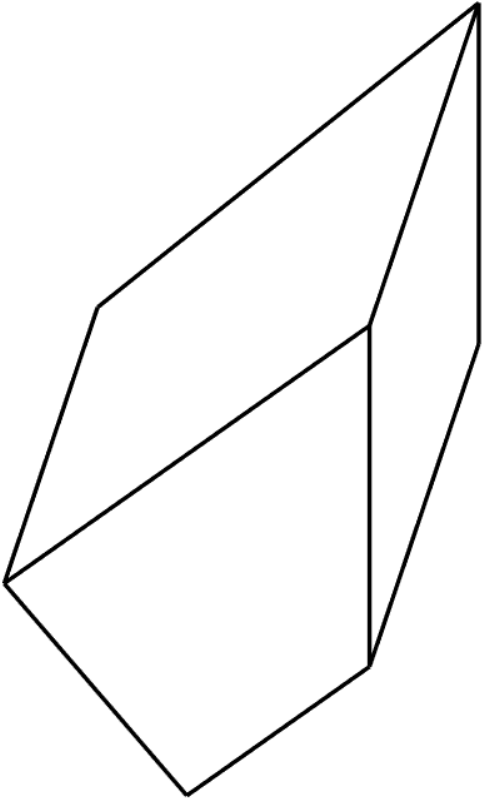}};
\draw[-stealth] (4,0) --node[above]{$\bbol{D}^{(1)}$} (4.75,0);
\node at (5.7,-.2) {\scalebox{1.25}{\rotslant{2}{50}{$\pmb{\circlearrowright}$}}};
\node at (5.7,.2) {\scalebox{1.35}{\rotslant{-90}{20}{$\pmb{\circlearrowright}$}}};
\node at (7.7,-.65) {\scalebox{1.75}{$\pmb{\circlearrowleft}$}}; 
\node at (9.7,-.65) {\scalebox{1.75}{$\pmb{\circlearrowleft}$}}; 
\node at (8.3,.1) {\rotslant{55}{52}{$\pmb{\circlearrowright}$}};
\node at (10.3,.1) {\rotslant{55}{52}{$\pmb{\circlearrowright}$}};
\node at (7.5,.25) {\scalebox{1.25}{\rotslant{15}{50}{$\pmb{\circlearrowright}$}}};
\node at (9.5,.25) {\scalebox{1.25}{\rotslant{15}{50}{$\pmb{\circlearrowright}$}}};
\node at (2.75,-5) {
\begin{tikzpicture}[scale=4]
\draw (.135,0) -- (-.067,.116);
\draw[dashed] (.135,0) -- (-.067,.116);
\draw (.135,0) -- (-.067,-.116);
\draw[dashed] (.135,0) -- (-.067,-.116);
\draw (0,.04) node{{\scriptsize 1}};
\draw (0,.-.04) node{{\scriptsize 2}};
\draw (0,0) circle (.135);
\draw (-.303,.175) -- (0,.35);
\draw (0,.35) -- (.303,.175);
\draw (-.303,-.175) -- (-.303,.175);
\draw[dashed] (-.303,-.175) -- (0,-.35) -- (.303,-.175) -- (.303,.175);
\draw (-.519,.3) -- (0,.6);
\draw (0,.6) -- (.519,.3);
\draw (-.519,-.3) -- (-.519,.3);
\draw[dashed] (-.519,-.3) -- (0,-.6) -- (.519,-.3) -- (.519,.3);
\draw (0,.135) --node[right]{{\scriptsize 3}} (0,.35);
\draw (0,.35) --node[right, yshift = -0.15cm]{{\scriptsize $n^\theta + 3$}} (0,.6);
\draw[dashed] (0,.7) -- (0,.6);
\draw[dashed] (0,-.135) -- (0,-.7);
\draw (-.116,.067) --node[yshift = 0.15cm, xshift = 0.15cm]{{\scriptsize 4}} (-.303,.175);
\draw (-.303,.175) --node[yshift = 0.15cm, xshift = 0.5cm]{{\scriptsize $n^\theta + 4$}} (-.519,.3);
\draw[dashed] (-.596,.345)--(-.519,.3);
\draw (.116,.067) --node[yshift = -0.15cm, xshift = 0.25cm]{{\scriptsize $n^\theta + 2$}} (.303,.175);
\draw (.303,.175) --node[yshift = -0.25cm, xshift = 0.25cm]{{\scriptsize $2n^\theta + 2$}}  (.519,.3);
\draw[dashed] (.596,.345)--(.519,.3);
\draw (-.116,-.067) --node[yshift = 0.25cm]{{\scriptsize 5}} (-.303,-.175);
\draw (-.303,-.175) --node[yshift = 0.35cm]{{\scriptsize $n^\theta + 5$}} (-.519,-.3);
\draw[dashed] (-.596,-.345)--(-.519,-.3);
\draw[dashed] (.596,-.345) -- (.116,-.067);
\node at (.425,-.65) {{\scriptsize $+(k-1)\bar{n}_1$}};
\end{tikzpicture}$\quad$
\begin{tikzpicture}[scale=4]
\draw (0,0) circle (.135);
\draw (0,.35) -- node[xshift = -0.25, yshift = 0.25cm]{{\scriptsize\rotatebox{30}{2}}}(-.303,.175);
\draw (.303,.175) --node[xshift = 0.25, yshift =  0.25cm]{{\scriptsize \rotatebox{-30}{1}}} (0,.35);
\draw (-.303,.175) --node[left]{{\scriptsize \rotatebox{90}{3}}} (-.303,-.175);
\draw[dashed] (-.303,-.175) -- (0,-.35) -- (.303,-.175) -- (.303,.175);
\draw  (0,.6) --node[xshift = -0.25, yshift = 0.25cm]{{\scriptsize\rotatebox{30}{$n^\theta + 2$}}} (-.519,.3);
\draw  (.519,.3) --node[xshift = 0.25, yshift =  0.25cm]{{\scriptsize \rotatebox{-30}{$n^\theta + 1$}}} (0,.6);
\draw (-.519,.3) --node[left]{{\scriptsize \rotatebox{90}{$n^\theta + 3$}}} (-.519,-.3);
\draw[dashed] (-.519,-.3) -- (0,-.6) -- (.519,-.3) -- (.519,.3);
\draw (0,.135) -- (0,.35);
\draw (0,.35) -- (0,.6);
\draw[dashed] (0,.7) -- (0,.6);
\draw[dashed] (0,-.135) -- (0,-.7);
\draw (-.116,.067) -- (-.303,.175);
\draw (-.303,.175) -- (-.519,.3);
\draw[dashed] (-.596,.345)--(-.519,.3);
\draw (.116,.067) -- (.303,.175);
\draw (.303,.175) --  (.519,.3);
\draw[dashed] (.596,.345)--(.519,.3);
\draw (-.116,-.067) -- (-.303,-.175);
\draw (-.303,-.175) -- (-.519,-.3);
\draw[dashed] (-.596,-.345)--(-.519,-.3);
\draw[dashed] (.596,-.345) -- (.116,-.067);
\node at (.425,-.55) {{\scriptsize $+(k-1)\bar{n}_1$}};
\node at (.425,-.65) {{\scriptsize $+2 + (n^r - 2)n^\theta$}};
\end{tikzpicture}$\quad$
\begin{tikzpicture}[scale=4]
\draw (-.303,.175) -- (0,.35);
\draw (0,.35) -- (.303,.175);
\draw (-.303,-.175) -- (-.303,.175);
\draw[dashed] (-.303,-.175) -- (0,-.35) -- (.303,-.175) -- (.303,.175);
\draw (0,0) circle (.135);
\draw (-.519,.3) -- (0,.6);
\draw (0,.6) -- (.519,.3);
\draw (-.519,-.3) -- (-.519,.3);
\draw[dashed] (-.519,-.3) -- (0,-.6) -- (.519,-.3) -- (.519,.3);
\draw (0,.135) -- (0,.35);
\draw (0,.35) -- (0,.6);
\draw[dashed] (0,.7) -- (0,.6);
\draw[dashed] (0,-.135) -- (0,-.7);
\draw (-.116,.067) -- (-.303,.175);
\draw (-.303,.175) -- (-.519,.3);
\draw[dashed] (-.596,.345)--(-.519,.3);
\draw (.116,.067) -- (.303,.175);
\draw (.303,.175) -- (.519,.3);
\draw[dashed] (.596,.345)--(.519,.3);
\draw (-.116,-.067) -- (-.303,-.175);
\draw (-.303,-.175) -- (-.519,-.3);
\draw[dashed] (-.596,-.345)--(-.519,-.3);
\draw[dashed] (.596,-.345) -- (.116,-.067);
\node at (-.1,.2) {{\scriptsize $1$}};
\node at (.1,.2) {{\scriptsize $n^\theta$}};
\node at (-.2,0) {{\scriptsize $2$}};
\node at (-.2,.35) {{\scriptsize $1+n^\theta$}};
\node at (.2,.35) {{\scriptsize $2n^\theta$}};
\node at (-.4,0) {{\scriptsize $2+n^\theta$}};
\node at (.5,-.5) {{\scriptsize $+n^\varphi\bar{n}_1$}};
\node at (.5,-.6) {{\scriptsize $+(k-1)\bar{n}_2$}};
\end{tikzpicture}
};
\end{tikzpicture}
\caption{The action of the curl operator when associating the DOFs in $V_1$ to the oriented edges of the polygonal ring $\cR$. In the curl space the DOFs correspond to oriented faces of $\cR$. On the top of the figure we see part of $\cR$. In particular on the right we show sampled faces of $\cR$. The left-most sampled faces are those connecting the centers of two consecutive joints of $\cR$. $\bbol{D}^{(1)}$ is the matrix encoding the action of the curl in this geometric interpretation and maps the edges to the faces. On the bottom of the figure we see the index numbering of the DOFs in the curl space associated to the faces from the point of view of the $k$th joint of $\cR$, for $k \in \{1, \ldots, n^\varphi\}$. In particular the faces between the $k$th joint and the $(k+1)$th joint are seen as radial and poloidal edges from this perspective in the left-most and central figures.}\label{curldof2}
\end{figure}
We order the degrees of freedom in $V_2$ as follows (see the bottom of Figure \ref{curldof2} for reference). First we run over all those associated to faces on the sides of $\cR$, following the same ordering used in $V_1$ for the DOFs linked to edges in a joint of $\cR$. Then we number those associated to faces in a joint of $\cR$. For each joint, these faces are ordered anticlockwise from the innermost polygon, in the net composing the joint, towards the outermost.
The action on the DOFs operated by the curl is expressed by the matrix $\bbol{D}^{(1)}$ of size $n_2 \times n_1$ which has the following structure and entries
\begin{equation}\label{eq:D1}
\bbol{D}^{(1)} = \begin{bmatrix}
\per{\bbol{D}}_{n^\varphi} \otimes -\bbol{I}_{\bar{n}_1} & \bbol{I}_{n^\varphi} \otimes \bbol{D}_J^{(0)}\\
\bbol{I}_{n^\varphi} \otimes \bbol{D}_J^{(1)} & \raisebox{-0.05cm}{\scalebox{2}{$\circ$}}  
\end{bmatrix}
\end{equation}
where $\bbol{D}_J^{(1)}$ is the following $\bar{n}_2 \times \bar{n}_1$ matrix, identified in \cite[Equations (94)--(97)]{deepesh},  providing the DOFs related to a joint in $\cR$:
\begin{equation}\label{eq:DJ1}
\bbol{D}_J^{(1)} \defeq \raisebox{-1.5cm}{\begin{tikzpicture}[scale = 2.25]
\draw (0, 0) rectangle (3.5, 1.5);
\draw (0, 1) -- (3.5, 1);
\draw (1, 0) -- (1, 1.5);
\draw (2, 0) -- (2, 1.5);
\draw (0.5, 1) -- (0.5, 1.5);
\draw (2.5, 1) -- (2.5, 1.5);
\node at (0.25, 1.25) {\scalebox{0.75}{$-\bar{\bbol{E}}_J^{(1), R, R}$}};
\node at (0.75, 1.25) {$-\per{\bbol{D}}_{n^\theta}$};
\node at (1.5, 0.5) {$\bbol{I}_{n^r-3} \otimes -\per{\bbol{D}}_{n^\theta}$};
\node at (2.75, 0.5) {$\bbol{D}_{n^r-2} \otimes \bbol{I}_{n^\theta}$};
\node at (2.25, 1.25) {$\bbol{I}_{n^\theta}$};
\node at (0.5, 0.5) {\scalebox{3}{$\circ$}};
\node at (1.5, 1.25) {\scalebox{3}{$\circ$}};
\node at (3, 1.25) {\scalebox{3}{$\circ$}};
\draw[|-|] (0, 1.6) --node[above]{$2$} (0.5, 1.6);
\draw[-|] (0.5, 1.6) --node[above]{$(n^r-2)n^\theta$} (2, 1.6);
\draw[-|] (2, 1.6) --node[above]{$(n^r-2)n^\theta$} (3.5, 1.6);
\draw[|-|] (3.6, 1.5) --node[right]{$n^\theta$} (3.6, 1);
\draw[-|] (3.6, 1) --node[right]{$(n^r-3)n^\theta$} (3.6, 0);
\end{tikzpicture}}
\end{equation}
with $\bar{\bbol{E}}_J^{(1), R, R}$ the right block of $\bar{\bbol{E}}_J^{(1), R}$ reported in Equation \eqref{eq:EJ1}, that is, the matrix whose elements are $\bar{E}_{\ell, n^\theta + j}^{(1), R}$ for $\ell = 1, 2$ and $j = 1, \ldots, n^\theta$. 

Let now $\bbol{E}^{(2)}$ be the following matrix:
\begin{equation}\label{eq:E2}
\bbol{E}^{(2)} \defeq \begin{bmatrix}
\bbol{I}_{n^\varphi} \otimes \bbol{E}_J^{(1), R} & -\bbol{I}_{n^\varphi} \otimes \bbol{E}_J^{(1), L} &\raisebox{-0.15cm}{\scalebox{3}{$\circ$}} \\
\raisebox{-0.15cm}{\scalebox{3}{$\circ$}} & \raisebox{-0.15cm}{\scalebox{3}{$\circ$}} & \bbol{I}_{n^\varphi} \otimes \bbol{E}^{(2)}_J
\end{bmatrix}
\end{equation}
with $\bbol{E}^{(2)}_J$ the matrix introduced in \cite[Equations (94)--(97)]{deepesh} which relates the DOFs in $V_2$ associated to faces in a joint of $\cR$ to those in $\SSS^{p^r, p^\theta - 1, p^\varphi - 1} \times \SSS^{p^r-1, p^\theta, p^\varphi - 1} \times \SSS^{p^r-1, p^\theta - 1, p^\varphi}$ associated to faces in $\cM$ parallel to the $(r, \theta)$ plane. $\bbol{E}^{(2)}_J$ has size $\bar{n}^2 \times n^\theta(n^r-1)$ and the following simple stricture
\begin{equation}\label{eq:E2J}
\bbol{E}^{(2)}_J = \begin{bmatrix} \raisebox{-0.075cm}{\scalebox{2}{$\circ$}} & \bbol{I}_{n^\theta(n^r - 2)} \end{bmatrix}.
\end{equation}
Furthermore, note that the blocks involving $\bbol{E}_J^{(1), L}$ and $\bbol{E}_J^{(1), R}$ are inverted compared to the expression of $\bbol{E}^{(1)}$.
Given a set of DOFs in $V_2$, $\m{h}=\{h_\ell\}_{\ell=1}^{n_2}$, we will have that
\begin{equation}\label{dof2relation}
\begin{bmatrix}
\m{h}^1 \\\\ \m{h}^2 \\\\ \m{h}^3
\end{bmatrix} = (\bbol{E}^{(2)})^T \m{h}
\end{equation}
where $\m{h}^1, \m{h}^2, \m{h}^3$ are the vectorizations of the DOFs in the first, second and third components of an element $\pmb{h} \in \SSS^{p^r,p^\theta-1,p^\varphi-1}\times\SSS^{p^r-1,p^\theta,p^\varphi-1}\times\SSS^{p^r-1,p^\theta-1,p^\varphi}$, i.e.,
$$
\pmb{h} = \begin{pmatrix}\m{B}^{(0,1,1)}\cdot \m{h}^1\\\\\m{B}^{(1,0,1)}\cdot \m{h}^2\\\\\m{B}^{(1,1,0)}\cdot \m{h}^3\end{pmatrix} = \begin{pmatrix} h^1\\\\h^2\\\\h^3\end{pmatrix}  \in \SSS^{p^r,p^\theta-1,p^\varphi-1}\times\SSS^{p^r-1,p^\theta,p^\varphi-1}\times\SSS^{p^r-1,p^\theta-1,p^\varphi}.
$$
The geometric interpretation of this link between the DOFs in $V_2$ and in $\SSS^{p^r,p^\theta-1,p^\varphi-1}\times\SSS^{p^r-1,p^\theta,p^\varphi-1}\times\SSS^{p^r-1,p^\theta-1,p^\varphi}$ is visualized in Figure \ref{dof22}.
\begin{figure}
\centering
\scalebox{.9}{
\begin{tikzpicture}[scale=3.5]
\node at (2,1.25) {\includegraphics[width=.1\textwidth]{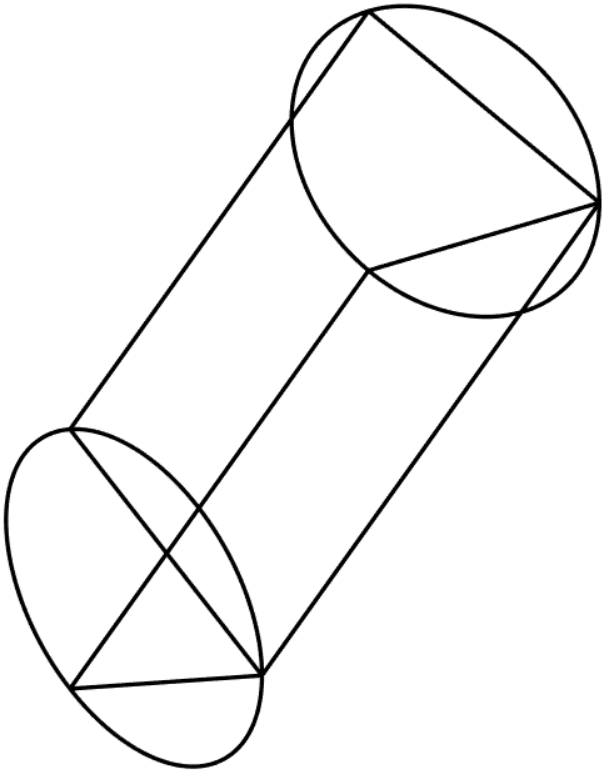}};
\node at (2.65,1.25) {\includegraphics[width=.1\textwidth]{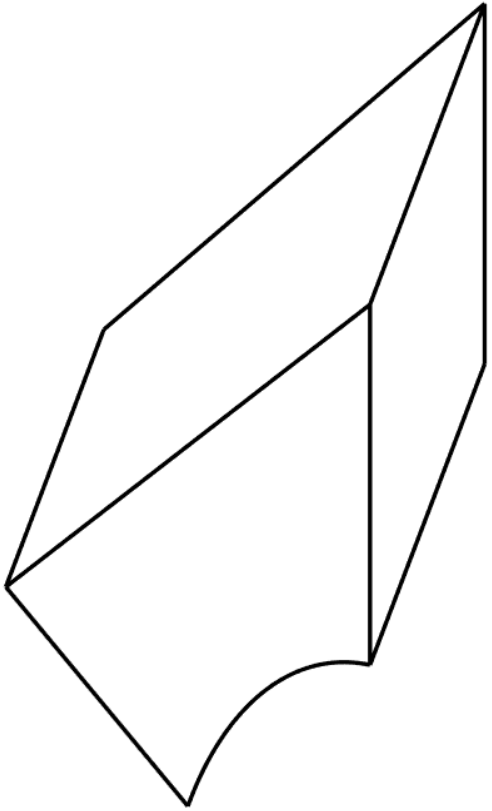}};
\node at (3.4,1.25) {\includegraphics[width=.1\textwidth]{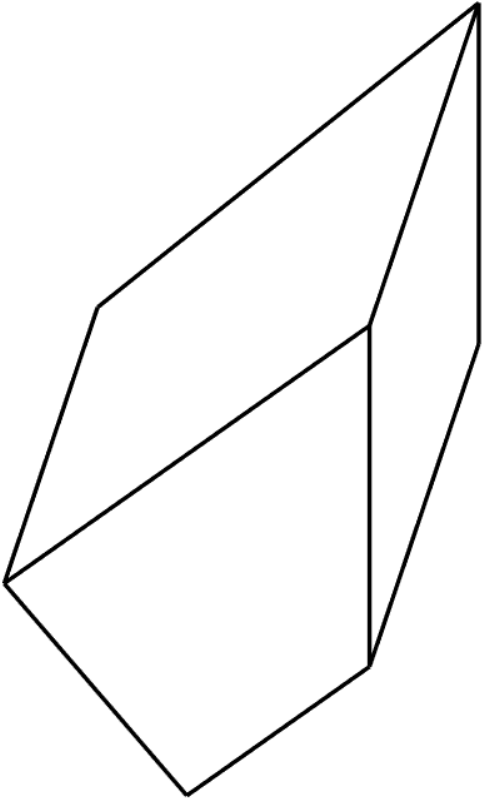}};
\node at (1.95,1.15) {\scalebox{1.25}{\rotslant{2}{50}{\textcolor{cyan!50!magenta}{$\pmb{\circlearrowright}$}}}};
\node at (2,1.3) {\scalebox{1.35}{\rotslant{-90}{20}{\textcolor{cyan!50!magenta}{$\pmb{\circlearrowright}$}}}};
\node at (2.65,1.1) {\scalebox{1.75}{\textcolor{orange!90!black}{$\pmb{\circlearrowleft}$}}}; 
\node at (3.4,1.05) {\scalebox{1.75}{\textcolor{orange!90!black}{$\pmb{\circlearrowleft}$}}};
\node at (2.8,1.3) {\rotslant{55}{52}{\textcolor{cyan!90!black}{$\pmb{\circlearrowright}$}}};
\node at (3.55,1.25) {\rotslant{55}{52}{\textcolor{cyan!90!black}{$\pmb{\circlearrowright}$}}};
\node at (2.6,1.3) {\scalebox{1.25}{\rotslant{15}{50}{\textcolor{magenta!90!black}{$\pmb{\circlearrowleft}$}}}};
\node at (3.35,1.3) {\scalebox{1.25}{\rotslant{15}{50}{\textcolor{magenta!90!black}{$\pmb{\circlearrowleft}$}}}};

\draw[rounded corners] (1.65,.7) -- (3.8,.7) -- (3.8,1.75) -- (1.65,1.75) -- cycle;
\foreach \x in{0,...,4}{
\draw (0,\x*.25,1) -- (.75,\x*.25 ,1);
\draw[dashed] (.75,\x*.25 ,1)--(1,\x*.25 ,1);
\draw (\x*.25 ,0,1) -- (\x*.25 ,.75,1);
\draw[dashed] (\x*.25 ,.75,1)--(\x*.25 ,1,1);
\draw (1,\x*.25 ,1) -- (1,\x*.25 ,.25);
\draw[dashed] (1,\x*.25 ,.25)--(1,\x*.25 ,0);
\draw (\x*.25 ,1,1) -- (\x*.25 ,1,.25);
\draw[dashed] (\x*.25 ,1,.25)--(\x*.25 ,1,0);
\draw (1,0,\x*.25 ) -- (1,.75,\x*.25 );
\draw[dashed] (1,.75,\x*.25 )--(1,1,\x*.25 );
\draw (0,1,\x*.25 ) -- (.75,1,\x*.25 );
\draw[dashed] (.75,1,\x*.25 )--(1,1,\x*.25 );
}
\foreach \x in {0,1,2}{
\foreach \y in {0,1,2}{
\node at (\x*.25+.125,\y*.25+.125,1) {\textcolor{orange!90!black}{\acpatch}};
\node at (1,\x*.25+.125,1-\y*.25 -.075) {{\scriptsize\rotslant{45}{45}{\textcolor{cyan!90!black}{\cpatch}}}};
\node at (\x*.25+.1,1,1-\y*.25-.13) {{\scriptsize \rotslant{-30}{45}{\textcolor{magenta!90!black}{\cpatch}}}};
}}

\draw[orange!90!black] (.375,0,1) node[below]{{\scriptsize $2,1,1$}};
\draw[orange!90!black] (.625,0,1) node[below]{{\scriptsize $3,1,1$}};
\draw[orange!90!black] (-.125,.175,1) node[below]{{\scriptsize $1,1,1$}};
\draw[orange!90!black] (-.125,.425,1) node[below]{{\scriptsize $1,2,1$}};
\draw[orange!90!black] (-.125,.675,1) node[below]{{\scriptsize $1,3,1$}};
\draw[magenta!90!black] (-.25,1,.1) node[below]{{\scriptsize $1,n^\theta,3$}};
\draw[magenta!90!black] (-.25,1,.4) node[below]{{\scriptsize $1,n^\theta,2$}};
\draw[magenta!90!black] (-.25,1,.7) node[below]{{\scriptsize $1,n^\theta,1$}};
\draw[magenta!90!black] (.35,.99,1) node[below]{{\scriptsize $2,n^\theta,1$}};
\draw[magenta!90!black] (.65,.99,1) node[below]{{\scriptsize $3,n^\theta,1$}};
\draw[cyan!90!black] (.875,.125,1) node{{\scriptsize $n^r,1,1$}};
\draw[cyan!90!black] (1.175,0,.675) node{{\scriptsize $n^r,1,2$}};
\draw[cyan!90!black] (1.175,0,.475) node{{\scriptsize $n^r,1,3$}};
\draw[cyan!90!black] (.875,.375,1) node{{\scriptsize $n^r,2,1$}};
\draw[cyan!90!black] (.875,.6,1) node{{\scriptsize $n^r,3,1$}};

\draw[-stealth]  (1.65,1.35) to [in=60,out=180] (.5,1.15);
\node at (.75,1.5) {$(\bbol{E}^{(2)})^T$};
\draw[-stealth] (2.6,.7) to [out=270,in=90] (2.1,.5);
\draw[-stealth] (2.6,.7) to [out=270,in=90] (3.1,.5);
\node at (2.75,-.25) {
\begin{tikzpicture}[scale=4]
\draw (.135,0) -- (-.067,.116);
\draw[dashed] (.135,0) -- (-.067,.116);
\draw (.135,0) -- (-.067,-.116);
\draw[dashed] (.135,0) -- (-.067,-.116);
\draw (0,.04) node{{\scriptsize \textcolor{cyan!50!magenta}{$1$}}};
\draw (0,.-.04) node{{\scriptsize \textcolor{cyan!50!magenta}{$2$}}};
\draw (0,0) circle (.135);
\draw (-.303,.175) -- (0,.35);
\draw (0,.35) -- (.303,.175);
\draw (-.303,-.175) -- (-.303,.175);
\draw[dashed] (-.303,-.175) -- (0,-.35) -- (.303,-.175) -- (.303,.175);
\draw (-.519,.3) -- (0,.6);
\draw (0,.6) -- (.519,.3);
\draw (-.519,-.3) -- (-.519,.3);
\draw[dashed] (-.519,-.3) -- (0,-.6) -- (.519,-.3) -- (.519,.3);
\draw (0,.135) --node[right]{{\scriptsize \textcolor{cyan!90!black}{$3$}}} (0,.35);
\draw (0,.35) --node[right, yshift = -0.15cm]{{\scriptsize \textcolor{cyan!90!black}{$n^\theta + 3$}}} (0,.6);
\draw[dashed] (0,.7) -- (0,.6);
\draw[dashed] (0,-.135) -- (0,-.7);
\draw (-.116,.067) --node[yshift = 0.15cm, xshift = 0.15cm]{{\scriptsize \textcolor{cyan!90!black}{$4$}}} (-.303,.175);
\draw (-.303,.175) --node[yshift = 0.15cm, xshift = 0.5cm]{{\scriptsize \textcolor{cyan!90!black}{$n^\theta + 4$}}} (-.519,.3);
\draw[dashed] (-.596,.345)--(-.519,.3);
\draw (.116,.067) --node[yshift = -0.15cm, xshift = 0.25cm]{{\scriptsize \textcolor{cyan!90!black}{$n^\theta + 2$}}} (.303,.175);
\draw (.303,.175) --node[yshift = -0.25cm, xshift = 0.25cm]{{\scriptsize \textcolor{cyan!90!black}{$2n^\theta + 2$}}}  (.519,.3);
\draw[dashed] (.596,.345)--(.519,.3);
\draw (-.116,-.067) --node[yshift = 0.25cm]{{\scriptsize \textcolor{cyan!90!black}{$5$}}} (-.303,-.175);
\draw (-.303,-.175) --node[yshift = 0.35cm]{{\scriptsize \textcolor{cyan!90!black}{$n^\theta + 5$}}} (-.519,-.3);
\draw[dashed] (-.596,-.345)--(-.519,-.3);
\draw[dashed] (.596,-.345) -- (.116,-.067);
\node at (.425,-.55) {{\scriptsize $+(k-1)\bar{n}_1$}};
\end{tikzpicture}
\begin{tikzpicture}[scale=4]
\draw (0,0) circle (.135);
\draw (0,.35) -- node[xshift = -0.25, yshift = 0.25cm]{{\scriptsize\rotatebox{30}{\textcolor{magenta!90!black}{$2$}}}}(-.303,.175);
\draw (.303,.175) --node[xshift = 0.25, yshift =  0.25cm]{{\scriptsize \rotatebox{-30}{\textcolor{magenta!90!black}{$1$}}}} (0,.35);
\draw (-.303,.175) --node[left]{{\scriptsize \rotatebox{90}{\textcolor{magenta!90!black}{$3$}}}} (-.303,-.175);
\draw[dashed] (-.303,-.175) -- (0,-.35) -- (.303,-.175) -- (.303,.175);
\draw  (0,.6) --node[xshift = -0.25, yshift = 0.25cm]{{\scriptsize\rotatebox{30}{\textcolor{magenta!90!black}{$n^\theta + 2$}}}} (-.519,.3);
\draw  (.519,.3) --node[xshift = 0.25, yshift =  0.25cm]{{\scriptsize \rotatebox{-30}{\textcolor{magenta!90!black}{$n^\theta + 1$}}}} (0,.6);
\draw (-.519,.3) --node[left]{{\scriptsize \rotatebox{90}{\textcolor{magenta!90!black}{$n^\theta + 3$}}}} (-.519,-.3);
\draw[dashed] (-.519,-.3) -- (0,-.6) -- (.519,-.3) -- (.519,.3);
\draw (0,.135) -- (0,.35);
\draw (0,.35) -- (0,.6);
\draw[dashed] (0,.7) -- (0,.6);
\draw[dashed] (0,-.135) -- (0,-.7);
\draw (-.116,.067) -- (-.303,.175);
\draw (-.303,.175) -- (-.519,.3);
\draw[dashed] (-.596,.345)--(-.519,.3);
\draw (.116,.067) -- (.303,.175);
\draw (.303,.175) --  (.519,.3);
\draw[dashed] (.596,.345)--(.519,.3);
\draw (-.116,-.067) -- (-.303,-.175);
\draw (-.303,-.175) -- (-.519,-.3);
\draw[dashed] (-.596,-.345)--(-.519,-.3);
\draw[dashed] (.596,-.345) -- (.116,-.067);
\node at (.425,-.55) {{\scriptsize $+(k-1)\bar{n}_1$}};
\node at (.425,-.65) {{\scriptsize $+2 + (n^r - 2)n^\theta$}};
\end{tikzpicture}
};
\node at (2.75, -1.75) {
\begin{tikzpicture}[scale=4]
\draw (-.303,.175) -- (0,.35);
\draw (0,.35) -- (.303,.175);
\draw (-.303,-.175) -- (-.303,.175);
\draw[dashed] (-.303,-.175) -- (0,-.35) -- (.303,-.175) -- (.303,.175);
\draw (0,0) circle (.135);
\draw (-.519,.3) -- (0,.6);
\draw (0,.6) -- (.519,.3);
\draw (-.519,-.3) -- (-.519,.3);
\draw[dashed] (-.519,-.3) -- (0,-.6) -- (.519,-.3) -- (.519,.3);
\draw (0,.135) -- (0,.35);
\draw (0,.35) -- (0,.6);
\draw[dashed] (0,.7) -- (0,.6);
\draw[dashed] (0,-.135) -- (0,-.7);
\draw (-.116,.067) -- (-.303,.175);
\draw (-.303,.175) -- (-.519,.3);
\draw[dashed] (-.596,.345)--(-.519,.3);
\draw (.116,.067) -- (.303,.175);
\draw (.303,.175) -- (.519,.3);
\draw[dashed] (.596,.345)--(.519,.3);
\draw (-.116,-.067) -- (-.303,-.175);
\draw (-.303,-.175) -- (-.519,-.3);
\draw[dashed] (-.596,-.345)--(-.519,-.3);
\draw[dashed] (.596,-.345) -- (.116,-.067);
\node at (-.1,.2) {{\scriptsize \textcolor{orange!90!black}{$1$}}};
\node at (.1,.2) {{\scriptsize \textcolor{orange!90!black}{$n^\theta$}}};
\node at (-.2,0) {{\scriptsize \textcolor{orange!90!black}{$2$}}};
\node at (-.2,.35) {{\scriptsize \textcolor{orange!90!black}{$1+n^\theta$}}};
\node at (.2,.35) {{\scriptsize \textcolor{orange!90!black}{$2n^\theta$}}};
\node at (-.4,0) {{\scriptsize \textcolor{orange!90!black}{$2+n^\theta$}}};
\node at (.5,-.5) {{\scriptsize $+n^\varphi\bar{n}_1$}};
\node at (.5,-.6) {{\scriptsize $+(k-1)\bar{n}_2$}};
\end{tikzpicture}
};
\draw (2.6, 0.7) to[out = -90, in = 90] (2.69, .35);
\draw[-stealth] (2.69, .35) -- (2.69, -0.85);
\end{tikzpicture}}
\caption{Geometric interpretation of the DOFs in the curl spaces $\SSS^{p^r,p^\theta-1,p^\varphi-1}\times\SSS^{p^r-1,p^\theta,p^\varphi-1}\times\SSS^{p^r-1,p^\theta-1,p^\varphi}$ (left) and $V_2$ (right). In the former space, the DOFs are associated to the oriented faces of a tensor mesh $\cM$. As we have already pointed out in Figure \ref{dof2}, there should be faces going outward of $\cM$ in the left and back sides because of the periodicity of the spaces. However we have omitted them for the sake of readability. In $V_2$ instead, the DOFs are represented as oriented faces of the polygonal ring $\cR$ defined in Section \ref{geometricinterpretation}. In the top-right part of the figure we show in a box some sampled  faces of $\cR$. In the middle figures we show the numbering of the faces corresponding to the DOFs in $V_2$ from the point of view of the $k$th joint of $\cR$, with $k \in \{1, \ldots, n^\varphi\}$. In particular the faces between the $k$th joint and the $(k+1)$th joint are seen as edges from this perspective. When considering a function in $V_2$, one can get its representation in terms of the basis functions of $\SSS^{p^r,p^\theta-1,p^\varphi-1}\times\SSS^{p^r-1,p^\theta,p^\varphi-1}\times\SSS^{p^r-1,p^\theta-1,p^\varphi}$ by applying the transpose of the matrix $\bbol{E}^{(2)}$ defined in \eqref{eq:E2}.}\label{dof22} 
\end{figure}
\begin{prop}\label{commutation2}
Let $\{h_{ijk}^1\}_{i,j,k=1}^{n^r,n^\theta,n^\varphi},\{h_{ijk}^2\}_{i,j,k=1}^{n^r-1,n^\theta,n^\varphi},\{h_{ijk}^3\}_{i,j,k=1}^{n^r-1,n^\theta,n^\varphi}$ be a general set of DOFs in $\im(\curl) \subseteq\SSS^{p^r,p^\theta-1,p^\varphi-1}\times\SSS^{p^r-1,p^\theta,p^\varphi-1}\times\SSS^{p^r-1,p^\theta-1,p^t}$. Then there exists a set of DOFs $\{g_{ijk}^1\}_{i,j,k=1}^{n^r-1,n^\theta,n^\varphi}$, $\{g_{ijk}^2\}_{i,j,k=1}^{n^r,n^\theta,n^\varphi}$, $\{g_{ijk}^3\}_{i,j,k=1}^{n^r,n^\theta,n^\varphi}$ in $\SSS^{p^r-1,p^\theta,p^\varphi}\times\SSS^{p^r,p^\theta-1,p^\varphi}\times\SSS^{p^r,p^\theta,p^\varphi-1}$, a set of DOFs $\{g_\ell\}_{\ell=1}^{n_1}$ a general set of DOFs in $V_1$ and a collection of numbers $\{h_\ell\}_{\ell=1}^{n_2}$ such that Equation \eqref{dof2relation} holds true, for which the following diagram commutes:
$$
\begin{tikzpicture}
\matrix (m)[matrix of math nodes,column sep=14em,row sep=3em]{
\left\{\begin{array}{c}g_{ijk}^1\\\\g_{ijk}^2\\\\g_{ijk}^3\end{array}\right\}_{i,j,k} \pgfmatrixnextcell \left\{\begin{array}{c}h_{ijk}^1\\\\h_{ijk}^2\\\\h_{ijk}^3\end{array}\right\}_{i,j,k}\\ \{g_\ell\}_{\ell} \pgfmatrixnextcell \{h_\ell\}_{\ell}\\
};
\draw[-stealth] (m-1-1) --node[above]{{\scriptsize$\begin{bmatrix} 0 & -\bbol{D}^{(0,0,1)} & \bbol{D}^{(0,1,0)}\\ \bar{\bbol{D}}^{(0,0,1)} & 0 & -\bbol{D}^{(1, 0, 0)}\\ -\bar{\bbol{D}}^{(0,0,1)} & \bbol{D}^{(1, 0, 0)} & 0\end{bmatrix}$}} (m-1-2);
\draw[-stealth] (m-2-1) --node[below]{$\bbol{D}^{(1)}$} (m-2-2);
\draw[-stealth] (m-2-1) --node[left]{$(\bbol{E}^{(1)})^T$} (m-1-1);
\draw[-stealth] (m-2-2) --node[right]{$(\bbol{E}^{(2)})^T$} (m-1-2);
\draw[-stealth,dashed] plot [smooth] coordinates {(-3.5, -1.5) (-2.65, 0.5) (2.5, 0.7)};
\draw[-stealth,dashed] plot [smooth] coordinates {(-3.25, -1.65) (2.5, -1.5) (3.35, -0.35)};
\end{tikzpicture}
$$
\end{prop}
One can prove Proposition \ref{commutation2} using an analogous argument of the proof of Proposition \ref{commutation1}. The interested reader is referred to Appendix \ref{app:commutation12}.

We can move to the definitions of $V_2$. We introduce the following collections of $n_2$ functions, \begin{equation}\label{N011}
\m{N}^{(0, 1, 1)} \defeq \bbol{E}^{(2)}\begin{bmatrix}
\m{B}^{(0, 1, 1)}\\ \pmb{0} \\ \pmb{0}
\end{bmatrix}, \qquad \m{N}^{(1, 0, 1)} \defeq \bbol{E}^{(2)} \begin{bmatrix}
\pmb{0}\\ \m{B}^{(1, 0, 1)} \\ \pmb{0}
\end{bmatrix}, \qquad \m{N}^{(1, 1, 0)} \defeq \bbol{E}^{(2)} \begin{bmatrix}
\pmb{0}\\ \pmb{0} \\ \m{B}^{(1, 1, 0)}\end{bmatrix}.
\end{equation}
By looking at the structure of $\bbol{E}^{(2)}$, reported in Equation \eqref{eq:E2}, one easily recognizes that there are zero functions in $\m{N}^{(0,1,1)}, \m{N}^{(1,0,1)}$ and $\m{N}^{(1,1,0)}$. However, there exists no index $\ell$ in $\{1, \ldots, n_2\}$ such that $N_\ell^{(0,1,1)}, N_\ell^{(1,0,1)}$ and $N_\ell^{(1,1,0)}$ are all zero. Furthermore, by using an analogous argument of the proof of Proposition \ref{prop:linind1}, one shows the independence of the non-zero functions.
\begin{prop}\label{prop:linind2}
The non-zero functions in $\m{N}^{(0,1,1)}, \m{N}^{(1,0,1)}$ and $\m{N}^{(1,1,0)}$, defined in Equation \eqref{N011}, are linearly independent.
\end{prop} 
We define the space $V_2$ as the span of the following linearly independent vector functions
\begin{equation}\label{N2}
N_\ell^{(2)} := \left(\begin{array}{c}
N_\ell^{(0,1,1)}\\\\N_\ell^{(1,0,1)}\\\\N_\ell^{(1,1,0)}
\end{array}\right)\quad \text{for }\ell=1, \ldots, n_2.
\end{equation}
Thanks to Remark \ref{linearalgebraremark}, Proposition \ref{commutation2} proves commutation in the following diagram:
\begin{cor}\label{cor2}
The diagram 
$$
\begin{tikzpicture}
\matrix (m)[matrix of math nodes,column sep=3em,row sep=3em]{
\SSS^{p^r-1,p^s,p^t}\times\SSS^{p^r,p^s-1,p^t}\times\SSS^{p^r,p^s,p^t-1} \pgfmatrixnextcell \SSS^{p^r,p^s-1,p^t-1}\times\SSS^{p^r-1,p^s,p^t-1}\times\SSS^{p^r-1,p^s-1,p^t} \\ V_1 \pgfmatrixnextcell V_2\\
};

\draw[-stealth] (m-1-1) --node[above]{$\curl$} (m-1-2);
\draw[-stealth] (m-2-1) --node[below]{$\curl$} (m-2-2);
\draw[-stealth] (m-2-1) --node[left]{$\id$} (m-1-1);
\draw[-stealth] (m-2-2) --node[right]{$\id$} (m-1-2);
\draw[-stealth, dashed] plot [smooth] coordinates {(-3.85, -0.5) (-3.25, 0.5) (0, 0.65)};
\draw[-stealth, dashed] plot [smooth] coordinates {(-3.75, -0.75) (2.25, -0.65) (3.25, 0.5)};
\end{tikzpicture}
$$
commutes and the pushforward of the functions in $V_2$, namely $\cF^2(N_\ell^{(2)})$ for $\ell=1, \ldots, n_2$, with the operator $\cF^2$ defined as in \eqref{pushforward}, are $C^0$ on $\Omega^{pol}$.
\end{cor}
Corollary \ref{cor2} can be proved adapting the arguments used in the proof of Corollary \ref{cor1}.

Hence, given a function $\pmb{g}\in V_1$, $\pmb{g} = \m{N}^{(1)}\cdot \m{g}$, we have the following representation of $\curl\,\pmb{g}$ in terms of the basis of $V_2$: $$
\curl\,\pmb{g} = \m{N}^{(2)} \cdot \bbol{D}^{(1)}\m{g} = \left(\begin{array}{c}
\m{N}^{(0,1,1)} \cdot \bbol{D}^{(1)}\m{g}\\\\
\m{N}^{(1,0,1)} \cdot \bbol{D}^{(1)}\m{g}\\\\
\m{N}^{(1,1,0)} \cdot \bbol{D}^{(1)}\m{g}
\end{array}\right) \in V_2.
$$
\subsection{The divergence of functions in $V_2$ and the space $V_3$}
Let $\pmb{h} \in \SSS^{p^r,p^\theta-1,p^\varphi-1}\times\SSS^{p^r-1,p^\theta,p^\varphi-1}\times\SSS^{p^r-1,p^\theta-1,p^\varphi}$, $$
\pmb{h} = \left(\begin{array}{c}
h^1\\\\h^2\\\\h^3
\end{array}\right) = \left(\begin{array}{c}
\m{B}^{(0,1,1)}\cdot \m{h}^1\\\\\m{B}^{(1,0,1)}\cdot \m{h}^2\\\\\m{B}^{(1,1,0)}\cdot \m{h}^3
\end{array}\right)
$$
with $\m{h}^1, \m{h}^2, \m{h}^3$ the vectorizations of the coefficients $\{h_{ijk}^1\}_{i,j,k=1}^{n^r,n^\theta,n^\varphi}, \{h_{ijk}^2\}_{i,j,k=1}^{n^r-1,n^\theta,n^\varphi}$ and $\{h_{ijk}^3\}_{i,j,k=1}^{n^r-1,n^\theta,n^\varphi}$ obtained by running the indices first on $j$, then on $i$ and then on $k$. The divergence of $\pmb{h}$ has the following expression:
$$
\div \pmb{h} = \sum_{k=1}^{n^\varphi}\sum_{j=1}^{n^\theta}\sum_{i=1}^{n^r-1} (h_{(i+1)jk}^1-h_{ijk}^1 + h_{i(j+1)k}^2 - h_{ijk}^2 + h_{ij(k+1)}^3- h_{ijk}^3)D_i^{p^r}\per{D}_j^{p^\theta}\per{D}_k^{p^\varphi}
$$
where the functions $D_i^{p^r}, \per{D}_j^{p^\theta}$ and $\per{D}_k^{p^\varphi}$ for any $i,j,k$ have been introduced in Section \ref{preliminaries}. Let $\m{B}^{(1,1,1)}$ be the vectorization of $\{D_i^{p^r}\per{D}_j^{p^\theta}\per{D}_k^{p^\varphi}\}_{i,j,k=1}^{n^r-1,n^\theta,n^\varphi}$. The divergence operator $\div:\SSS^{p^r,p^\theta-1,p^\theta-1}\times\SSS^{p^r-1,p^\theta,p^\varphi-1}\times\SSS^{p^r-1,p^\theta-1,p^\varphi} \to \SSS^{p^r-1,p^\theta-1,p^\varphi-1}$ has the following matrix representation:
$$
\div \pmb{h} =  \m{B}^{(1,1,1)} \cdot [\bbol{D}^{(1,0,0)}\m{h}^1 + \bbol{D}^{(0,1,0)}\m{h}^2 + \bbol{D}^{(0,0,1)}\m{h}^3] \quad \forall\, \pmb{h} \in \SSS^{p^r,p^\theta-1,p^\varphi-1}\times\SSS^{p^r-1,p^\theta,p^\varphi-1}\times\SSS^{p^r-1,p^\theta-1,p^\varphi},
$$
where the matrices $\bbol{D}^{(1,0,0)},\bbol{D}^{(0,1,0)}$ and $\bbol{D}^{(0,0,1)}$ have been defined in Equation \eqref{D100}. The DOFs in $\im(\div)$ have form
\begin{equation}\label{eq:divh}
m_{ijk} := h_{(i+1)jk}^1-h_{ijk}^1 + h_{i(j+1)k}^2 - h_{ijk}^2 + h_{ij(k+1)}^3- h_{ijk}^3
\end{equation}
for $i=1, \ldots, n^r-1$, $j=1, \ldots, n^\theta$ and $k=1, \ldots, n^\varphi$, for some set of DOFs $\{h_{ijk}^1\}_{i,j,k=1}^{n^r,n^\theta,n^\varphi}$, $\{h_{ijk}^2\}_{i,j,k=1}^{n^r-1,n^\theta,n^\varphi}$, $\{h_{ijk}^3\}_{i,j,k=1}^{n^r-1,n^\theta,n^\varphi}$ in $\SSS^{p^r,p^\theta-1,p^\varphi-1}\times\SSS^{p^r-1,p^\theta,p^\varphi-1}\times\SSS^{p^r-1,p^\theta-1,p^\varphi}$. In the geometric interpretation proposed in Section \ref{geometricinterpretation}, $h_{ijk}^1, h_{(i+1)jk}^1, h_{ijk}^2, h_{i(j+1)k}^2, h_{ijk}^3, h_{ij(k+1)}^3$ correspond to the faces of a tensor mesh $\cM$ enclosing the volume related to $m_{ijk}$. Therefore, from this geometric point of view, the divergence operator maps oriented faces enclosing a volume to such volume, as shown in Figure \ref{divtensor}.
\begin{figure}
\centering
\begin{tikzpicture}[scale=3]
\foreach \x in{0,...,4}{
\draw (0,\x*.25,1) -- (.75,\x*.25 ,1);
\draw[dashed] (.75,\x*.25 ,1)--(1,\x*.25 ,1);
\draw (\x*.25 ,0,1) -- (\x*.25 ,.75,1);
\draw[dashed] (\x*.25 ,.75,1)--(\x*.25 ,1,1);
\draw (1,\x*.25 ,1) -- (1,\x*.25 ,.25);
\draw[dashed] (1,\x*.25 ,.25)--(1,\x*.25 ,0);
\draw (\x*.25 ,1,1) -- (\x*.25 ,1,.25);
\draw[dashed] (\x*.25 ,1,.25)--(\x*.25 ,1,0);
\draw (1,0,\x*.25 ) -- (1,.75,\x*.25 );
\draw[dashed] (1,.75,\x*.25 )--(1,1,\x*.25 );
\draw (0,1,\x*.25 ) -- (.75,1,\x*.25 );
\draw[dashed] (.75,1,\x*.25 )--(1,1,\x*.25 );
}
\foreach \x in {0,1,2}{
\foreach \y in {0,1,2}{
\node at (\x*.25+.125,\y*.25+.125,1) {\textcolor{orange!90!black}{\acpatch}};
\node at (1,\x*.25+.125,1-\y*.25 -.075) {{\scriptsize\rotslant{45}{45}{\textcolor{cyan!90!black}{\cpatch}}}};
\node at (\x*.25+.1,1,1-\y*.25-.13) {{\scriptsize \rotslant{-30}{45}{\textcolor{magenta!90!black}{\cpatch}}}};
}}

\draw[orange!90!black] (.375,0,1) node[below]{{\scriptsize $2,1,1$}};
\draw[orange!90!black] (.625,0,1) node[below]{{\scriptsize $3,1,1$}};
\draw[orange!90!black] (-.125,.175,1) node[below]{{\scriptsize $1,1,1$}};
\draw[orange!90!black] (-.125,.425,1) node[below]{{\scriptsize $1,2,1$}};
\draw[orange!90!black] (-.125,.675,1) node[below]{{\scriptsize $1,3,1$}};
\draw[magenta!90!black] (-.25,1,.1) node[below]{{\scriptsize $1,n^\theta$,3}};
\draw[magenta!90!black] (-.25,1,.4) node[below]{{\scriptsize $1,n^\theta$,2}};
\draw[magenta!90!black] (-.25,1,.7) node[below]{{\scriptsize $1,n^\theta$,1}};
\draw[magenta!90!black] (.35,.99,1) node[below]{{\scriptsize $2,n^\theta$,1}};
\draw[magenta!90!black] (.65,.99,1) node[below]{{\scriptsize $3,n^\theta$,1}};
\draw[cyan!90!black] (.875,.125,1) node{{\scriptsize $n^r,1,1$}};
\draw[cyan!90!black] (1.175,0,.675) node{{\scriptsize $n^r,1,2$}};
\draw[cyan!90!black] (1.175,0,.475) node{{\scriptsize $n^r,1,3$}};
\draw[cyan!90!black] (.875,.375,1) node{{\scriptsize $n^r,2,1$}};
\draw[cyan!90!black] (.875,.6,1) node{{\scriptsize $n^r,3,1$}};
\draw[-stealth] (1.15,.5) --node[above]{$\left[\textcolor{cyan!90!black}{\bbol{D}^{(1,0,0)}}~\textcolor{magenta!90!black}{\bbol{D}^{(0,1,0)}}~\textcolor{orange!90!black}{\bbol{D}^{(0,0,1)}}\right]$} (2.65,.5);

\foreach \x in{0,...,4}{
\draw (3.15,\x*.25,1) -- (3.15+.75,\x*.25 ,1);
\draw[dashed] (3.15+.75,\x*.25 ,1)--(4.15,\x*.25 ,1);
\draw (3.15+\x*.25 ,0,1) -- (3.15+\x*.25 ,.75,1);
\draw[dashed] (3.15+\x*.25 ,.75,1)--(3.15+\x*.25 ,1,1);
\draw (3.15+1,\x*.25 ,1) -- (3.15+1,\x*.25 ,.25);
\draw[dashed] (3.15+1,\x*.25 ,.25)--(3.15+1,\x*.25 ,0);
\draw (3.15+\x*.25 ,1,1) -- (3.15+\x*.25 ,1,.25);
\draw[dashed] (3.15+\x*.25 ,1,.25)--(3.15+\x*.25 ,1,0);
\draw (3.15+1,0,\x*.25 ) -- (3.15+1,.75,\x*.25 );
\draw[dashed] (3.15+1,.75,\x*.25 )--(3.15+1,1,\x*.25 );
\draw (3.15,1,\x*.25 ) -- (3.15+.75,1,\x*.25 );
\draw[dashed] (3.15+.75,1,\x*.25 )--(3.15+1,1,\x*.25 );
}
\foreach \x in {0,...,3}{
\foreach \y in {0,...,3}{
\draw (3.15+\x*.25,\y*.25,1) -- (3.15+\x*.25,\y*.25,.75);
}
\draw (3.15+\x*.25,0,.75) -- (3.15+\x*.25,.75,.75);
\draw[dashed] (3.15+\x*.25,.75,.75) -- (3.15+\x*.25,1,.75);
\draw (3.15,\x*.25,.75) -- (3.15+.75,\x*.25,.75);
\draw[dashed] (3.15+.75,\x*.25,.75) -- (3.15+1,\x*.25,.75);
}
\foreach \x in {0,...,2}{
\foreach \y in {0,...,2}{
\fill[shading=ball,ball color=black] (3.15+\x*.25+.125,\y*.25+.125,.825) circle (.05);
}}
\draw (3.15+.375,0,1) node[below]{{\scriptsize $2,1,1$}};
\draw (3.15+.625,0,1) node[below]{{\scriptsize $3,1,1$}};
\draw (3.15-.125,.175,1) node[below]{{\scriptsize $1,1,1$}};
\draw (3.15-.125,.425,1) node[below]{{\scriptsize $1,2,1$}};
\draw (3.15-.125,.675,1) node[below]{{\scriptsize $1,3,1$}};
\draw (3.15-.25,1,.1) node[below]{{\scriptsize $1,n^\theta$,3}};
\draw (3.15-.25,1,.4) node[below]{{\scriptsize $1,n^\theta$,2}};
\draw (3.15+1.175,0,.675) node{{\scriptsize $n^r,1,2$}};
\draw (3.15+1.175,0,.475) node{{\scriptsize $n^r,1,3$}};
\end{tikzpicture}
\caption{The action of the divergence operator when associating the DOFs in $\SSS^{p^r,p^\theta-1,p^\varphi-1}\times\SSS^{p^r-1,p^\theta,p^\varphi-1}\times\SSS^{p^r-1,p^\theta-1,p^\varphi}$ to the oriented faces of a tensor mesh $\cM$, as shown on the left side of the figure. In this geometric interpretation, the divergence maps the six faces enclosing a volume in $\cM$ to such volume. For the periodicity of the spaces in the $\theta$ and $\varphi$ directions, on the left and back sides of the tensor meshes shown in the figure there should be faces going outward, in the mesh on the left, and an extra block of volume in each of these sides, in the mesh on the right. We have chosen to omit them for the sake of readability of the figure.}\label{divtensor}
\end{figure} 

We want the same action onto the DOFs of $V_2$. This time the latter are associated to the oriented faces of the polygonal ring $\cR$ defined in Section \ref{geometricinterpretation}. When applying the divergence operator, we should map such faces to the volumes they enclose as shown in Figure \ref{divV}. The total number of volumes in $\cR$ is $n_3 \defeq n^\varphi\bar{n}_2 = n^\varphi n^\theta(n^r - 2)$, as each face in a joint of $\cR$ corresponds to a volume in the side between such joint and the next joint. 
\begin{figure}
\centering
\begin{tikzpicture}
\node at (0,0) {\includegraphics[width=.1\textwidth]{dof21}};
\node at (2,0) {\includegraphics[width=.1\textwidth]{dof22}};
\node at (4,0) {\includegraphics[width=.1\textwidth]{dof23}};
\node at (6.5,0) {\includegraphics[width=.1\textwidth]{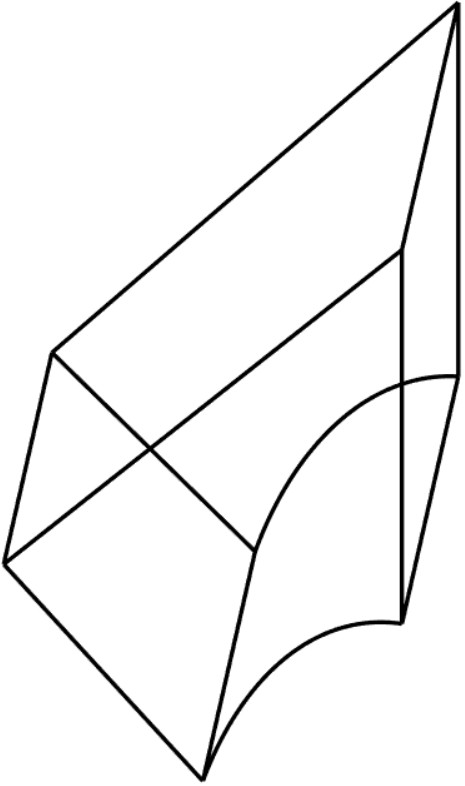}};
\node at (8.5,0) {\includegraphics[width=.1\textwidth]{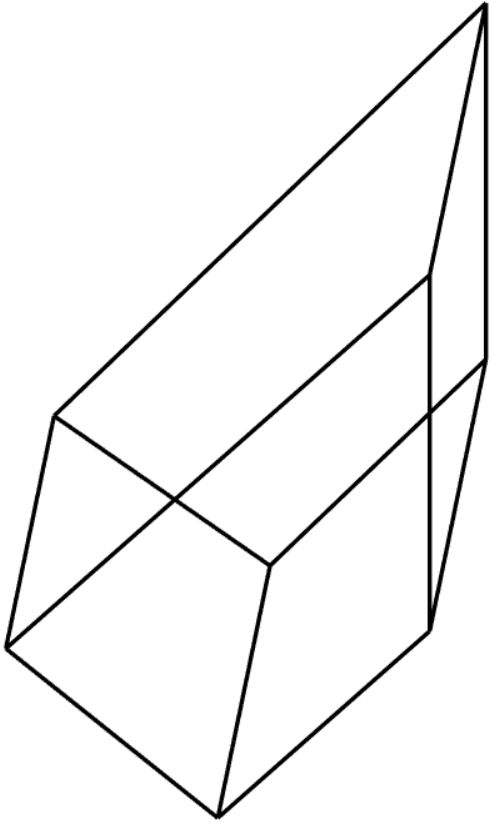}};
\draw[-stealth] (5,0) --node[above]{$\bbol{D}^{(2)}$} (5.75,0);
\node at (.0,-.2) {\scalebox{1.25}{\rotslant{2}{50}{$\pmb{\circlearrowright}$}}};
\node at (.05,.2) {\scalebox{1.35}{\rotslant{-90}{20}{$\pmb{\circlearrowright}$}}};
\node at (2,-.65) {\scalebox{1.75}{$\pmb{\circlearrowleft}$}}; 
\node at (4,-.65) {\scalebox{1.75}{$\pmb{\circlearrowleft}$}}; 
\node at (2.6,.1) {\rotslant{55}{52}{$\pmb{\circlearrowright}$}};
\node at (4.6,.1) {\rotslant{55}{52}{$\pmb{\circlearrowright}$}};
\node at (1.85,.25) {\scalebox{1.25}{\rotslant{15}{50}{$\pmb{\circlearrowleft}$}}};
\node at (3.85,.25) {\scalebox{1.25}{\rotslant{15}{50}{$\pmb{\circlearrowleft}$}}};
\fill[shading=ball, ball color=black] (6.65,-.35) circle (.25); 
\fill[shading=ball, ball color=black] (8.75,-.35) circle (.25); 
\node at (13,0) {
\begin{tikzpicture}[scale=3.5]
\draw (-.303,.175) -- (0,.35);
\draw (0,.35) -- (.303,.175);
\draw (-.303,-.175) -- (-.303,.175);
\draw[dashed] (-.303,-.175) -- (0,-.35) -- (.303,-.175) -- (.303,.175);
\draw (0,0) circle (.135);
\draw (-.519,.3) -- (0,.6);
\draw (0,.6) -- (.519,.3);
\draw (-.519,-.3) -- (-.519,.3);
\draw[dashed] (-.519,-.3) -- (0,-.6) -- (.519,-.3) -- (.519,.3);
\draw (0,.135) -- (0,.35);
\draw (0,.35) -- (0,.6);
\draw[dashed] (0,.7) -- (0,.6);
\draw[dashed] (0,-.135) -- (0,-.7);
\draw (-.116,.067) -- (-.303,.175);
\draw (-.303,.175) -- (-.519,.3);
\draw[dashed] (-.596,.345)--(-.519,.3);
\draw (.116,.067) -- (.303,.175);
\draw (.303,.175) -- (.519,.3);
\draw[dashed] (.596,.345)--(.519,.3);
\draw (-.116,-.067) -- (-.303,-.175);
\draw (-.303,-.175) -- (-.519,-.3);
\draw[dashed] (-.596,-.345)--(-.519,-.3);
\draw[dashed] (.596,-.345) -- (.116,-.067);
\draw (.135,0) -- (-.067,.116);
\draw (.135,0) -- (-.067,-.116);
\node at (-.1,.2) {{\scriptsize $1$}};
\node at (.1,.2) {{\scriptsize $n^\theta$}};
\node at (-.2,0) {{\scriptsize $2$}};
\node at (-.2,.35) {{\scriptsize $1+n^\theta$}};
\node at (.2,.35) {{\scriptsize $2n^\theta$}};
\node at (-.4,0) {{\fontsize{6}{5}\selectfont $2 + n^\theta$}};
\node at (.5,-.6) {{\scriptsize $+(k-1)\bar{n}_2$}};
\end{tikzpicture}};
\end{tikzpicture}
\caption{The action of the divergence operator when associating the DOFs in $V_2$ to the oriented faces of the polygonal ring $\cR$. In the divergence space the DOFs correspond to the volumes in $\cR$. For the sake of readability, in the left part of the figure we only show the different types of faces and volumes. $\bbol{D}^{(2)}$ is the matrix encoding the action of the divergence on the DOFs and it maps the faces enclosing a volume to such volume, in this geometric interpretation. On the right part of the figure we see the index numbering of the DOFs in the divergence space, associated to the volumes, from the point of view of the $k$th joint of $\cR$, for $k \in \{1, \ldots, n^\varphi\}$.  Each volume is seen as a face from this perspective.}\label{divV}
\end{figure}
Given $\pmb{h} \in V_2, \pmb{h}=\m{N}^{(2)}\cdot \m{h}$ for a set of coefficients $\m{h}=(h_1, \ldots, h_{n_2})^T$, the action of the divergence on the DOFs is expressed by the matrix $\bbol{D}^{(2)}$ of size $n_3\times n_2$ which has the following structure
\begin{equation}
\bbol{D}^{(2)} \defeq \begin{bmatrix}
\bbol{I}_{n^\varphi} \otimes \bbol{D}_J^{(1)} & \per{\bbol{D}}_{n^\varphi} \otimes \bbol{I}_{\bar{n}_0}
\end{bmatrix}
\end{equation}
Let now 
\begin{equation}\label{E3}
\bbol{E}^{(3)} = \bbol{I}_{n^\varphi} \otimes \bbol{E}_J^{(2)},
\end{equation} where $\bbol{E}_J^{(2)}$ is the matrix reported in Equation \eqref{eq:E2J}. $\bbol{E}^{(3)}$ has dimension $n_3 \times n^\varphi n^\theta(n^r-1)$ and it will provide the relation between the DOFs in $V_3$ and $\SSS^{p^r-1,p^\theta-1,p^\varphi-1}$, accordingly to Remark \ref{linearalgebraremark}, that is, given a set of DOFs in $V_3$, $\m{m} = \{m_\ell\}_{\ell=1}^{n_3}$, we have that
\begin{equation}\label{m}
\bar{\m{m}} = (\bbol{E}^{(3)})^T\m{m},
\end{equation}
where $\bar{\m{m}}$ is the vectorization of the DOFs of an element $\bar{m} \in \SSS^{p^r-1,p^\theta-1,p^\varphi-1}$, i.e., $\bar{m} = \m{B}^{(1,1,1)}\cdot \bar{\m{m}}$.
The geometric interpretation of this link between the DOFs in $V_3$ and in $\SSS^{p^r-1,p^\theta-1,p^\varphi-1}$ is shown in Figure \ref{dof3}.
\begin{figure}
\centering
\begin{tikzpicture}
\node at (0,0) {
\begin{tikzpicture}[scale=2.5]
\foreach \x in{0,...,4}{
\draw (3.15,\x*.25,1) -- (3.15+.75,\x*.25 ,1);
\draw[dashed] (3.15+.75,\x*.25 ,1)--(4.15,\x*.25 ,1);
\draw (3.15+\x*.25 ,0,1) -- (3.15+\x*.25 ,.75,1);
\draw[dashed] (3.15+\x*.25 ,.75,1)--(3.15+\x*.25 ,1,1);
\draw (3.15+1,\x*.25 ,1) -- (3.15+1,\x*.25 ,.25);
\draw[dashed] (3.15+1,\x*.25 ,.25)--(3.15+1,\x*.25 ,0);
\draw (3.15+\x*.25 ,1,1) -- (3.15+\x*.25 ,1,.25);
\draw[dashed] (3.15+\x*.25 ,1,.25)--(3.15+\x*.25 ,1,0);
\draw (3.15+1,0,\x*.25 ) -- (3.15+1,.75,\x*.25 );
\draw[dashed] (3.15+1,.75,\x*.25 )--(3.15+1,1,\x*.25 );
\draw (3.15,1,\x*.25 ) -- (3.15+.75,1,\x*.25 );
\draw[dashed] (3.15+.75,1,\x*.25 )--(3.15+1,1,\x*.25 );
}
\foreach \x in {0,...,3}{
\foreach \y in {0,...,3}{
\draw (3.15+\x*.25,\y*.25,1) -- (3.15+\x*.25,\y*.25,.75);
}
\draw (3.15+\x*.25,0,.75) -- (3.15+\x*.25,.75,.75);
\draw[dashed] (3.15+\x*.25,.75,.75) -- (3.15+\x*.25,1,.75);
\draw (3.15,\x*.25,.75) -- (3.15+.75,\x*.25,.75);
\draw[dashed] (3.15+.75,\x*.25,.75) -- (3.15+1,\x*.25,.75);
}
\foreach \x in {0,...,2}{
\foreach \y in {0,...,2}{
\fill[shading=ball,ball color=black] (3.15+\x*.25+.125,\y*.25+.125,.825) circle (.05);
}}
\end{tikzpicture}
};
\draw[-stealth] (3, 0) --node[above]{$(\bbol{E}^{(3)})^T$} (2, 0);
\node at (3.75, 0) {\includegraphics[width=.075\textwidth]{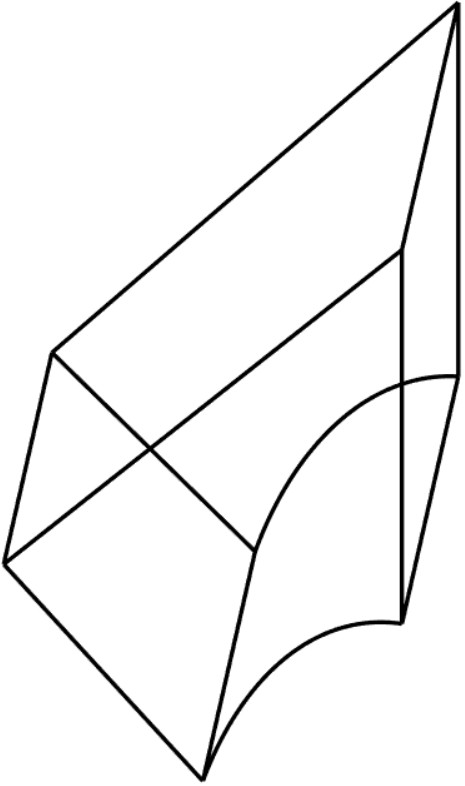}};
\node at (5.25, 0) {\includegraphics[width=.075\textwidth]{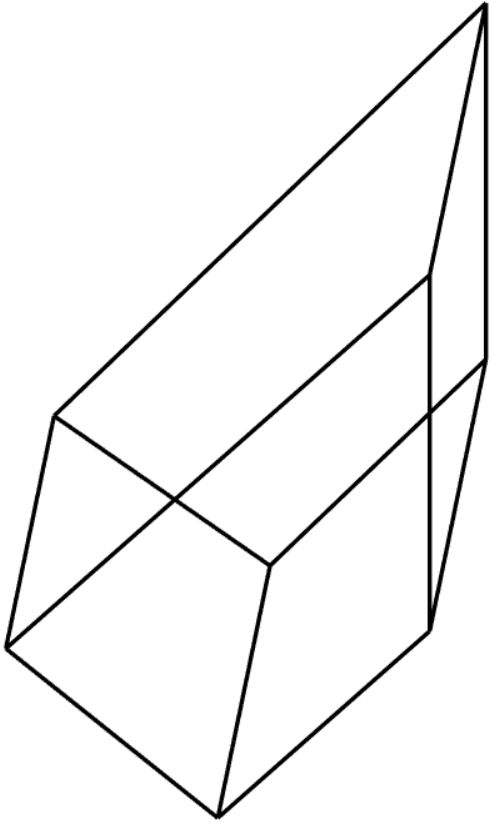}};
\fill[shading=ball,ball color=black] (3.9, -.35) circle (.2); 
\fill[shading=ball,ball color=black] (5.4,-.35) circle (.2); 
\node at (12.5, 0) {
\begin{tikzpicture}[scale=3.5]
\draw (-.303,.175) -- (0,.35);
\draw (0,.35) -- (.303,.175);
\draw (-.303,-.175) -- (-.303,.175);
\draw[dashed] (-.303,-.175) -- (0,-.35) -- (.303,-.175) -- (.303,.175);
\draw (0,0) circle (.135);
\draw (-.519,.3) -- (0,.6);
\draw (0,.6) -- (.519,.3);
\draw (-.519,-.3) -- (-.519,.3);
\draw[dashed] (-.519,-.3) -- (0,-.6) -- (.519,-.3) -- (.519,.3);
\draw (0,.135) -- (0,.35);
\draw (0,.35) -- (0,.6);
\draw[dashed] (0,.7) -- (0,.6);
\draw[dashed] (0,-.135) -- (0,-.7);
\draw (-.116,.067) -- (-.303,.175);
\draw (-.303,.175) -- (-.519,.3);
\draw[dashed] (-.596,.345)--(-.519,.3);
\draw (.116,.067) -- (.303,.175);
\draw (.303,.175) -- (.519,.3);
\draw[dashed] (.596,.345)--(.519,.3);
\draw (-.116,-.067) -- (-.303,-.175);
\draw (-.303,-.175) -- (-.519,-.3);
\draw[dashed] (-.596,-.345)--(-.519,-.3);
\draw[dashed] (.596,-.345) -- (.116,-.067);
\draw (.135,0) -- (-.067,.116);
\draw (.135,0) -- (-.067,-.116);
\node at (-.1,.2) {{\scriptsize $1$}};
\node at (.1,.2) {{\scriptsize $n^\theta$}};
\node at (-.2,0) {{\scriptsize $2$}};
\node at (-.2,.35) {{\scriptsize $1+n^\theta$}};
\node at (.2,.35) {{\scriptsize $2n^\theta$}};
\node at (-.4,0) {{\fontsize{6}{5}\selectfont $2 + n^\theta$}};
\node at (.375,-.6) {{\scriptsize $+(k-1)\bar{n}_2$}};
\end{tikzpicture}
};
\draw[-stealth] (10.35, 0) --node[above]{$(\bbol{E}_J^{(2)})^T$} (9.35, 0);
\node at (7.5, 0) {
\begin{tikzpicture}[scale=3]
\draw[step=.25] (0,0) grid (.75,.75);
\foreach \t in {0,...,3}{
\draw[dashed] (.75,\t*.25) -- (1,\t*.25);
\draw[dashed] (\t*.25,.75) -- (\t*.25,1);
}
\draw[dashed] (.75,1) -- (1,1);
\draw[dashed] (1,.75) -- (1,1);
\draw (0,1) -- (.75,1);
\draw (1,0) -- (1,.75);
\node at (.125,.125) {{\scriptsize $1,1,k$}};
\node at (.125,.375) {{\scriptsize $1,2,k$}};
\node at (.125,.625) {{\scriptsize $1,3,k$}};
\node at (.375,.125) {{\scriptsize $2,1,k$}};
\node at (.625,.125) {{\scriptsize $3,1,k$}};
\node at (.375,.375) {{\scriptsize $2,2,k$}};
\node at (.375,.625) {{\scriptsize $2,3,k$}};
\node at (.625,.625) {{\scriptsize $3,3,k$}};
\node at (.625,.375) {{\scriptsize $2,3,k$}};
\end{tikzpicture}
};
\end{tikzpicture}
\caption{Geometric interpretation of the DOFs in $\SSS^{p^r-1,p^\theta-1,p^\varphi-1}$ and $V_3$ as volumes enclosed, respectively, in a tensor mesh $\cM$ and a polygonal ring $\cR$, as introduced in Section \ref{geometricinterpretation}. The two sets are related through the transpose of matrix $\bbol{E}^{(3)}$, as explained in Remark \ref{linearalgebraremark}. $\bbol{E}^{(3)}$ is a block diagonal matrix with $\bbol{E}_J^{(2)}$ as diagonal blocks. On the right-most part of the figure we show the indices of the DOFs corresponding to the volumes in the $k$th side of $\cR$, seen from the point of view of the $k$th joint of $\cR$, for $k \in \{1, \ldots, n^\varphi\}$. These indices correspond to the indices $(i, j, k)$ for fixed $k$ via $\bbol{E}^{(2)}_J$. In this geometric interpretation, the DOFs related to such $(i, j, k)$ indices form a slice of the mesh $\cM$ obtained by cutting along the $(r, \theta)$-plane.}\label{dof3}
\end{figure}
\begin{prop}\label{commutation3}
Let $\{m_{ijk}\}_{i,j,k=1}^{n^r-1,n^\theta,n^\varphi}$ be a general set of DOFs in $\im(\div) \subseteq\SSS^{p^r-1,p^\theta-1,p^\varphi-1}$. Then there exist a set of DOFs $\{h_{ijk}^1\}_{i,j,k=1}^{n^r,n^\theta,n^\varphi},\{h_{ijk}^2\}_{i,j,k=1}^{n^r-1,n^\theta,n^\varphi},\{h_{ijk}^3\}_{i,j,k=1}^{n^r-1,n^\theta,n^\varphi}$ in $\SSS^{p^r,p^\theta-1,p^\varphi-1}\times \SSS^{p^r-1,p^\theta,p^\varphi-1}\times \SSS^{p^r-1,p^\theta-1,p^\varphi}$, a set of DOFs $\{h_\ell\}_{\ell=1}^{n_2}$ in $V_2$ and a collection of numbers $\{m_\ell\}_{\ell=1}^{n_3}$ such that Equation \eqref{m} holds true, for which the following diagram commutes:
$$
\begin{tikzpicture}
\matrix (m)[matrix of math nodes,column sep=14em,row sep=3em]{
\left\{\begin{array}{c}h_{ijk}^1\\\\h_{ijk}^2\\\\h_{ijk}^3\end{array}\right\}_{i,j,k} \pgfmatrixnextcell \{m_{ijk}\}_{i,j,k}\\ \{h_\ell\}_{\ell} \pgfmatrixnextcell \{m_\ell\}_{\ell}\\
};
\draw[-stealth] (m-1-1) --node[above]{$\left[\bbol{D}^{(1, 0, 0)}~\bbol{D}^{(0, 1, 0)}~\bbol{D}^{(0, 0, 1)}\right]$} (m-1-2);
\draw[-stealth] (m-2-1) --node[below]{$\bbol{D}^{(2)}$} (m-2-2);
\draw[-stealth] (m-2-1) --node[left]{$(\bbol{E}^{(2)})^T$} (m-1-1);
\draw[-stealth] (m-2-2) --node[right]{$(\bbol{E}^{(3)})^T$} (m-1-2);
\draw[-stealth,dashed] plot [smooth] coordinates {(-3.35, -1.5) (-2.35, 0.5) (2.77, 0.7)};
\draw[-stealth,dashed] plot [smooth] coordinates {(-3.15, -1.65) (2.5, -1.5) (3.5, 0.55)};
\end{tikzpicture}
$$
\end{prop}
The proof of Proposition \ref{commutation3} follows the same argument of the proof of Proposition \ref{commutation1}. Hence it is postponed to Appendix \ref{app:commutation12}.

Let now $\m{N}^{(3)} \defeq \bbol{E}^{(3)}\m{B}^{(1,1,1)}$. By looking at Equation \eqref{E3}, we note that there are only non-zero functions in $\m{N}^{(3)}$ as $\bbol{E}^{(3)}$ has only non-zero rows. Furthermore, such functions are linearly independent as the functions in $\m{B}^{(1,1,1)}$ are linearly independent and $\bbol{E}^{(3)}$ has full rank (equal to the number of rows). We define the space $V_3$ as the span of the functions in $\m{N}^{(3)}$. Thanks to Remark \ref{linearalgebraremark} and Proposition \ref{commutation3}, we are able to prove the following corollary.
\begin{cor}\label{cor3}
The diagram
$$
\begin{tikzpicture}
\matrix (m)[matrix of math nodes,column sep=3em,row sep=3em]{
\SSS^{p^r,p^s-1,p^t-1}\times \SSS^{p^r-1,p^s,p^t-1}\times \SSS^{p^r-1,p^s-1,p^t}\pgfmatrixnextcell\SSS^{p^r-1,p^s-1,p^t-1}\\ V_2 \pgfmatrixnextcell V_3\\
};
\draw[-stealth] (m-1-1) --node[above]{$\div$} (m-1-2);
\draw[-stealth] (m-2-1) --node[below]{$\div$} (m-2-2);
\draw[-stealth] (m-2-1) --node[left]{$\id$} (m-1-1);
\draw[-stealth] (m-2-2) --node[right]{$\id$} (m-1-2);
\draw[-stealth, dashed] plot [smooth] coordinates {(-1.75, -0.5) (-1, 0.5) (2.65, 0.65)};
\draw[-stealth, dashed] plot [smooth] coordinates {(-1.5, -0.75) (2.75, -0.65) (3.75, 0.5)};
\end{tikzpicture}
$$
commutes and the pushforward of the functions in $V_3$, namely $\cF^3(N_\ell^{(3)})$ for $\ell=1, \ldots, n_3$, with the operator $\cF^3$ defined as in \eqref{pushforward}, are bounded on $\Omega^{pol}$.
\end{cor}
Corollary \ref{cor3} can be proved adapting the arguments used in the proof of Corollary \ref{cor1}.

Therefore, given a function $\pmb{h} \in V_2, \pmb{h}=\m{N}^{(2)}\cdot \m{h}$, then $\div\pmb{h}=\m{N}^{(3)}\cdot \bbol{D}^{(2)}\m{h}$.
\section{Preservation of the cohomology dimensions}\label{preservation}
Let us define the following spline complex on $\Omega$:
\begin{equation}\label{splinecomplex}
\gS: 0 \xrightarrow{\id} V_0 \xrightarrow{\grad} V_1 \xrightarrow{\curl} V_2 \xrightarrow{\div} V_3 \xrightarrow{0} 0.
\end{equation}
We now prove that the cohomological structure of $\gS$ is equal to the cohomological structure of \eqref{continuousderham}. This preservation property will be immediately transferred to the corresponding complex on the physical domain, as we will show after the theorem.
\begin{thm}\label{preservationthm}
The cohomology spaces of the spline complex $\gS$ defined in Equation \eqref{splinecomplex} have the following dimensions:
$$
\dim \cH^0(\gS) = 1, \quad \dim \cH^1(\gS) = 1, \quad \dim \cH^2(\gS) = 0, \quad \dim \cH^3(\gS) = 0.
$$
\end{thm}
\begin{proof}
Let us start by computing $\dim \cH^0(\gS)$. We know that $\cH^0(\gS) = \ker(\grad)$. Let $f \in V_0$ be such that $\grad f = 0$. Since $\grad f = \sum_{\ell=1}^{n_1} g_\ell N_\ell^{(1)}$, this means that the coefficients $g_\ell=0$ for all $\ell$. With reference to Equation \eqref{eq:gl}, we have 
$$
\left\{
\begin{array}{l}
g_{1 +(k-1)(\bar{n}_0 + \bar{n}_1)}= f_{2+(k-1)\bar{n}_0} - f_{1+(k-1)\bar{n}_0} = 0\\\\
g_{2 +(k-1)(\bar{n}_0 + \bar{n}_1)}= f_{3+(k-1)\bar{n}_0} - f_{1+(k-1)\bar{n}_0} = 0
\end{array}
\right. \R  f_{1+(k-1)\bar{n}_0} = f_{2+(k-1)\bar{n}_0} = f_{3+(k-1)\bar{n}_0}=\alpha_k
$$
for some $\alpha_k \in \RR$. Then, for any $j \in \{1, \ldots, n^\theta\}$ we have
$$
g_{2+j + (k-1)(\bar{n}_0 + \bar{n}_1)} = f_{3+j+(k-1)\bar{n}_0} - \sum_{\ell=1}^3 \bar{E}_{\ell, n^\theta + j}f_{\ell+(k-1)\bar{n}_0} = 0\quad \R \quad f_{3+j+(k-1)\bar{n}_0} = \alpha_k\sum_{\ell=1}^3 \bar{E}_{\ell,n^\theta + j} \stackrel{\text{DTA}}{=} \alpha_k,
$$
where we have used that the matrix $\bbol{E}^{(0)}$ of Equation \eqref{E0} is DTA-compatible, so that in particular its columns sum to one. Next, if we impose
$$
g_{2 + (n^r-2)n^\theta + j + (k-1)\bar{n}_1} = f_{3+(j+1)+(i-3)n^\theta+(k-1)\bar{n}_0} - f_{3+j+(i-3)n^r+(k-1)\bar{n}_0} = 0
$$
for any $i \in \{3, \ldots, n^r\}$ and $j \in \{1, \ldots, n^\theta\}$ and 
$$
g_{2+j + (i-2)n^\theta+(k-1)\bar{n}_1} = f_{3+j+(i-2)n^\theta+(k-1)\bar{n}_0} - f_{3+j+(i-3)n^\theta+(k-1)\bar{n}_0} =0
$$
for any $i \in \{3, \ldots, n^r-1\}$ and $j \in \{1, \ldots, n^\theta\}$, we get
$$
\left\{
\begin{array}{l}
f_{3+(j+1)+(i-3)n^\theta+(k-1)\bar{n}_0} = f_{3+j+(i-3)n^\theta+(k-1)\bar{n}_0}\\\\
f_{3+j+(i-2)n^\theta+(k-1)\bar{n}_0} = f_{3+j+(i-3)n^\theta+(k-1)\bar{n}_0}.
\end{array}
\right.
$$ Hence, we have that all the coefficients of the form $f_{\ell+(k-1)\bar{n}_0}$ for a fixed $k$ and all $\ell \in \{1, \ldots, \bar{n}_0\}$ are equal to $\alpha_k$. However, we further have 
$$
g_{n^\varphi\bar{n}_1 + \ell + (k-1)\bar{n}_0} = f_{\ell+k\bar{n}_0} - f_{\ell+(k-1)\bar{n}_0} = 0 \quad \R \quad  f_{\ell+k\bar{n}_0} = f_{\ell+(k-1)\bar{n}_0},
$$
for $\ell \in \{1, \ldots, \bar{n}_0\}$, that is, $\alpha_{k+1} = \alpha_k = \alpha \in \RR$. We conclude that $1=\dim \ker(\grad)= \dim \cH^0(\gS)$. Let us now move to the computation of $\dim \cH^3(\gS)$. We show that $\div:V_2 \to V_3$ is surjective, so that $\im(\div) = V_3$ and $\dim \cH^3(\gS) = \dim V_3 - \dim \im(\div) = 0$. Therefore, given $m \in V_3$, $m = \m{N}^{(3)}\cdot \m{m}$, we look for a $\pmb{h} \in V_2$, $\pmb{h} = \m{N}^{(2)} \cdot \m{h}$, such that $\div\pmb{h} = m$, that is, such that $\bbol{D}^{(2)}\m{h} = \m{m}$. By choosing,
\begin{center}
\begin{algorithm}[H]
\For{$k=1, \ldots, n^\varphi$}{
$h_{1+(k-1)\bar{n}_1} = 0$\;
$h_{2+(k-1)\bar{n}_1} = 0$\;
\For{$i=1, \ldots, n^r-2$}{
\For{$j=1, \ldots, n^\theta$}{
$h_{2+j+(i-1)n^\theta+(k-1)\bar{n}_1} = 0$\;
$h_{2+(n^r-2)n^\theta + j +(i-1)n^\theta+(k-1)\bar{n}_1} = \displaystyle\sum_{s = 1}^i m_{j + (s-1)n^\theta + (k-1)\bar{n}_2}$\;
}}
\For{$\ell=1, \ldots, \bar{n}_2$}{
$h_{n^\varphi\bar{n}_1 + \ell+(k-1)\bar{n}_2} = \beta \in \RR$\;
}}
\end{algorithm}
\end{center}
one easily can check, looking at the action of $\bbol{D}^{(2)}$ reported in Equation \eqref{eq:divhpol}, that $\bbol{D}^{(2)}\m{h} = \m{m}$. Thus $\dim \cH^3(\gS) = 0$. We move to $\dim \cH^2(\gS) = \dim \ker(\div) - \dim \im(\curl)$.  Let us start by calculating $\dim \ker(\div)$. We have shown that the divergence is surjective, that is, $\dim \im(\div) = \dim V_3 = n_3= n^\varphi\bar{n}_2$. Furthermore, it holds $\dim \ker(\div) + \dim\im(\div) = \dim V_2 = n_2 = n^\varphi(\bar{n}_1+\bar{n}_2)$. Therefore, 
$$
\dim \ker(\div) = n^\varphi(\bar{n}_1 + \bar{n}_2) - n^\varphi\bar{n}_2 = n^\varphi\bar{n}_1.
$$
For $\dim \im(\curl)$, we claim that it is equal to 
\begin{equation}\label{eq:dimimcurl}
\dim \im(\curl) = n^\varphi(\bar{n}_2 + \bar{n}_0 -1).
\end{equation} 
We shall prove this fact later, at the end of the proof, for the sake of readability. Therefore, assuming that \eqref{eq:dimimcurl} holds, we have
\begin{equation*}
\begin{split}
\dim \cH^2(\gS) &= \dim \ker(\div) - \dim \im(\curl)\\
&= n^\varphi(\bar{n}_1- \bar{n}_2-\bar{n}_0+1)\\
&= n^\varphi(2n^\theta(n^r-2)+2-n^\theta(n^r-2) - n^\theta(n^r-2)-3+1)=0.
\end{split}
\end{equation*}
With the dimensions of $\cH^0(\gS),\cH^2(\gS)$ and $\cH^3(\gS)$ at hand, we can easily compute $\dim \cH^1(\gS)$ by using the following identity:
\begin{equation}\label{eq:ranknullity}
\dim \cH^0(\gS) - \dim \cH^1(\gS) + \dim \cH^2(\gS) - \dim \cH^3(\gS) = \dim V_0 - \dim V_1 + \dim V_2 - \dim V_3.
\end{equation}
Equation \eqref{eq:ranknullity} holds because the spaces $V_\ell$ for $\ell = 0, \ldots, 3$ are finite-dimensional and the differential operators act as linear transformations between them. Therefore, we have $\dim V_0 = \dim \im(\grad) + \dim \ker(\grad)$, $\dim V_1 = \dim \im(\curl) + \dim \ker(\curl)$, and $\dim V_2 = \dim \im(\div) + \dim \ker(\div) = \dim V_3 + \dim \ker(\div)$. The finite dimensionality of the spaces also implies that the dimensions of the quotient spaces $\cH^\ell$ for $\ell = 1, 2, 3$ are determined by the following differences: $\dim \cH^1 = \dim \ker(\curl) - \dim \im(\grad)$, $\dim \cH^2 = \dim \ker(\div) - \dim \im(\curl)$, and $\dim \cH^3 = \dim V_3 - \dim \im(\div)$.

On one hand we have 
$$
\dim V_0 - \dim V_1 + \dim V_2 - \dim V_3 = n^\varphi(\bar{n}_0-\bar{n}_0-\bar{n}_1+\bar{n}_1+\bar{n}_2-\bar{n}_2)=0,
$$
and on the other hand we have proved that
$$
\dim \cH^0(\gS) =1, \quad \dim \cH^2(\gS) = 0, \quad \dim \cH^3(\gS) = 0,
$$
so that we must have $$
\dim \cH^1(\gS) = 1.
$$
This completes the proof. It remains to show Equation \eqref{eq:dimimcurl}. This is equivalent to finding the rank of matrix $\bbol{D}^{(1)}$, $\rk(\bbol{D}^{(1)})$, reported in Equation \eqref{eq:D1}. We prove that there exists a Gaussian elimination that trasforms $\bbol{D}^{(1)}$ in 
\begin{equation}\label{eq:D1gauss}
\begin{bmatrix}
\raisebox{-0.05cm}{\scalebox{2}{$\circ$}} & \bbol{I}_{n^\varphi} \otimes \bbol{D}_J^{(0)}\\
\bbol{I}_{n^\varphi} \otimes \bbol{D}_J^{(1)} & \raisebox{-0.05cm}{\scalebox{2}{$\circ$}}  
\end{bmatrix}
\end{equation}
and therefore the rank of $\bbol{D}^{(1)}$ shall be 
$$\rk(\bbol{D}^{(1)}) = \rk(\bbol{I}_{n^\varphi} \otimes \bbol{D}_J^{(0)}) +  \rk(\bbol{I}_{n^\varphi} \otimes \bbol{D}_J^{(1)})= n^\varphi(\rk(\bbol{D}_J^{(0)}) + \rk(\bbol{D}_J^{(1)})).$$
The ranks of $\bbol{D}_J^{(0)}$ and $\bbol{D}_J^{(1)}$ can be found in \cite[proof of Theorem 5.10, first and second cohomologies]{deepesh} and are equal to $\bar{n}_0 - 1$ and $\bar{n}_2$. Therefore, in particular, there exist Gaussian eliminations that trasform the $\ell$th diagonal block of $\bbol{I}_{n^\varphi} \otimes \bbol{D}_J^{(0)}$, that is, the  $\bar{n}_1 \times \bar{n}_0$ matrix $\bbol{D}_J^{(0)}$, into
$$
\bbol{U}^{\ell, k} \defeq
\raisebox{-2cm}{
\begin{tikzpicture}[scale = 1.75]
\draw (0, -1) rectangle (1.5, 1);
\draw[|-|] (0, 1.1) --node[above]{{\scriptsize$\bar{n}_0 - 1$}} (1, 1.1);
\draw[-|] (1, 1.1) --node[above]{{\scriptsize $1$}} (1.5, 1.1);
\draw[|-|] (1.6, 1) -- node[right]{{\scriptsize $k$}} (1.6, 0.5);
\draw[-|] (1.6, 0.5) --node[right]{{\scriptsize \rotatebox{-90}{$\bar{n}_0 - 1$}}} (1.6, -0.5);
\draw[-|] (1.6, 0) -- (1.6, -1);
\draw (1.6, -0.75) node[right]{{\scriptsize $\bar{n}_1 - \bar{n}_0 + 1 - k$}};
\draw (1, -1) -- (1, 1);
\draw (0, -0.5) -- (1, -0.5);
\draw (0, 0.5) -- (1, 0.5);
\node at (1.25, 0) {$\begin{array}{c} 0\\\vdots\\\vdots\\\vdots\\0 \end{array}$};
\node at (0.5, 0) {$\bbol{I}_{\bar{n}_0 - 1}$};
\node at (0.5, -0.75) {\raisebox{-0.1cm}{\scalebox{2}{$\circ$}}};
\node at (0.5, 0.75) {\raisebox{-0.1cm}{\scalebox{2}{$\circ$}}};
\end{tikzpicture}}\hspace{-1cm}\forall\, k = 0, \ldots, \bar{n}_1 - \bar{n}_0 + 1.
$$
Let then split $\per{\bbol{D}}_{n^\varphi} \otimes -\bbol{I}_{\bar{n}_1}$ by columns into $n^\varphi$ consecutive matrices $\bbol{A}^\ell$ of size $n^\varphi\bar{n}_1 \times \bar{n}_1$:
$$\per{\bbol{D}}_{n^\varphi} \otimes -\bbol{I}_{\bar{n}_1} = \raisebox{-1.75cm}{
\begin{tikzpicture}[scale = 3.5]
\fill[gray!20!white] (0, 0) rectangle (0.2, 1);
\fill[gray!40!white] (0.2, 0) rectangle (0.4, 1);
\fill[gray!60!white] (0.4, 0) rectangle (0.6, 1);
\fill[gray!80!white] (0.8, 0) rectangle (1, 1);
\draw (0, 0) rectangle (1, 1);
\draw (0, 0.8) -- (0.6, 0.8);
\draw (0.2, 0.6) -- (0.6, 0.6);
\draw (0.2, 0.6) -- (0.2, 1);
\draw (0.4, 0.4) -- (0.4, 1);
\draw (0.4, 0.4) -- (0.6, 0.4);
\draw (0.6, 0.4) -- (0.6, 0.8);
\draw (0.8, 0) -- (0.8, 0.4) -- (1, 0.4);
\draw (0.8, 0.2) -- (1, 0.2);
\draw (0, 0) rectangle (0.2, 0.2);
\draw[dotted] (0.61, 0.39) -- (0.79, 0.21);
\draw[dotted] (0.61, 0.59) -- (0.79, 0.41);
\node at (0.1, 0.9) {$\bbol{I}_{\bar{n}_1}$};
\node at (0.3, 0.9) {$-\bbol{I}_{\bar{n}_1}$};
\node at (0.3, 0.7) {$\bbol{I}_{\bar{n}_1}$};
\node at (0.5, 0.7) {$-\bbol{I}_{\bar{n}_1}$};
\node at (0.5, 0.5) {$\bbol{I}_{\bar{n}_1}$};
\node at (0.9, 0.3) {$-\bbol{I}_{\bar{n}_1}$};
\node at (0.9, 0.1) {$\bbol{I}_{\bar{n}_1}$};
\node at (0.1, 0.1) {$-\bbol{I}_{\bar{n}_1}$};
\draw[-stealth] (0.1, 1.15) node[above, xshift=0.05]{$\bbol{A}^1$} -- (0.1, 1.05);
\draw[-stealth] (0.3, 1.15) node[above, xshift=0.05]{$\bbol{A}^2$} -- (0.3, 1.05);
\draw[-stealth] (0.5, 1.15) node[above, xshift=0.05]{$\bbol{A}^3$} -- (0.5, 1.05);
\draw[-stealth] (0.9, 1.15) node[above, xshift=0.05]{$\bbol{A}^{n^\varphi}$} -- (0.9, 1.05);
\end{tikzpicture}}.
$$

$\bbol{A}^\ell$ is basically a vertical band in $\per{\bbol{D}}_{n^\varphi} \otimes -\bbol{I}_{\bar{n}_1}$. Furthermore, each $\bbol{A}^\ell$ has only two non-zero blocks of size $\bar{n}_1 \times \bar{n}_1$, equal to $\pm\bbol{I}_{\bar{n}_1}$ in position $\ell -1, \ell$ vertically, respectively. Given the $j$th column of $\bbol{A}^\ell$, $\m{a}_j^\ell$, then it holds
\begin{equation}\label{eq:columnaj}
\left\{\begin{array}{ll}
\m{a}_j^\ell = \m{u}_j^{\ell, 0} - \m{u}_j^{\ell - 1, 0} & \text{if } j \leq \bar{n}_0 - 1 \\\\
\m{a}_j^\ell = \m{u}_{\bar{n}_0-1}^{\ell, j - \bar{n}_0 + 1} - \m{u}_{\bar{n}_0-1}^{\ell - 1, j - \bar{n}_0 + 1} & \text{if } j \geq \bar{n}_0
\end{array}\right.
\end{equation}
where $\m{u}_j^{\ell, 0}$ and $\m{u}_{\bar{n}_0-1}^{\ell, j - \bar{n}_0 + 1}$ are the $j$th and $(\bar{n}_0-1)$th columns of the Gaussian trasformations of $\bbol{I}_{n^\varphi} \otimes \bbol{D}_J^{(0)}$ with $\ell$th block equal to $\bbol{U}^{\ell, 0}$ and $\bbol{U}^{\ell, j - \bar{n}_0+1}$ respectively. As summing columns of the right side of $\bbol{D}^{(1)}$ to the left side of it has no effect to the lower block $\bbol{I}_{n^\varphi} \otimes \bbol{D}_J^{(1)}$, we get the matrix in Equation \eqref{eq:D1gauss} from \eqref{eq:columnaj}. Therefore, $\dim\im(\curl) = \rk(\bbol{D}^{(1)}) = n^\varphi(\bar{n}_2 + \bar{n}_0 -1)$.
\end{proof}
Since the pushforward operators defined in \eqref{pushforward} commute with the differential operators $\grad,\curl$ and $\div$, see e.g. \cite[Theorem 19.8]{tu}, the spline complex on $\Omega^{pol}$  
$$
\gS^{pol}:0 \xrightarrow{\id} V_0^{pol} \xrightarrow{\grad} V_1^{pol} \xrightarrow{\curl} V_2^{pol} \xrightarrow{\div} V_3^{pol} \xrightarrow{0} 0,
$$
where $V_i^{pol}$ are the spaces spanned by the functions $N_\ell^{(i),pol} := \cF^i(N_\ell^{(i)})$ for $i=1,2,3$, preserves the cohomological structure of the de Rham complex \eqref{continuousderham}, because of Theorem \ref{preservationthm}.

\section{Numerical experiments}\label{numerical}
We now present some numerical experiments to check the approximation power of the polar spline spaces constructed. In particular, we perform $L^2$ projections of smooth scalar and vector fields onto the polar spline spaces. The approximations indeed converge at the expected rates, even in the presence of the singularity introduced by the underlying toroidal geometry. This supports our claim that the singularity of the parametric map is properly managed by our construction and that the spaces are suitable for the intended discretization of differential forms in toroidal solids.

We consider solid toroidal domains $\Omega^{pol}$ (see Figure \ref{fig:geometries_l2projection}) that are defined at the coarsest level by uniformly dividing the parametric domain $\Omega = [0,1] \times [0, 2\pi] \times [0, 2\pi]$ into $2, 5$ and $5$ elements in the $r, \theta$ and $\varphi$ directions, respectively.
On this mesh of $\Omega$, we build the tensor-product spline space $\SSS^{p, p, p}$, $p \in \{2, 3\}$, using maximally smooth splines, i.e., $C^{p-1}$ smooth splines, in each parametric direction; recall that the splines in $\theta$ and $\varphi$ directions are also periodic.
This space of splines is used for building $V_0^{pol}$ and is also used to describe the geometries; the remaining polar spline spaces $V_k^{pol}$, $k = 1, 2, 3$, are built as described in Section \ref{ssec:discretizationoverview}.

Note that the above setup means that, at the coarsest level, the tensor-product space $\SSS^{p, p, p}$ is defined by $n^r = p+2, n^\theta = 5, n^\varphi = 5$ basis functions for $p \in \{2, 3\}$.
When refining the spaces for performing the convergence study, we refine the parametric domain uniformly by bisection and maintain maximal smoothness in each parametric direction.

\begin{figure}
  \centering
  \subfloat[$p = 2$]{
    \includegraphics[width=0.49\textwidth]{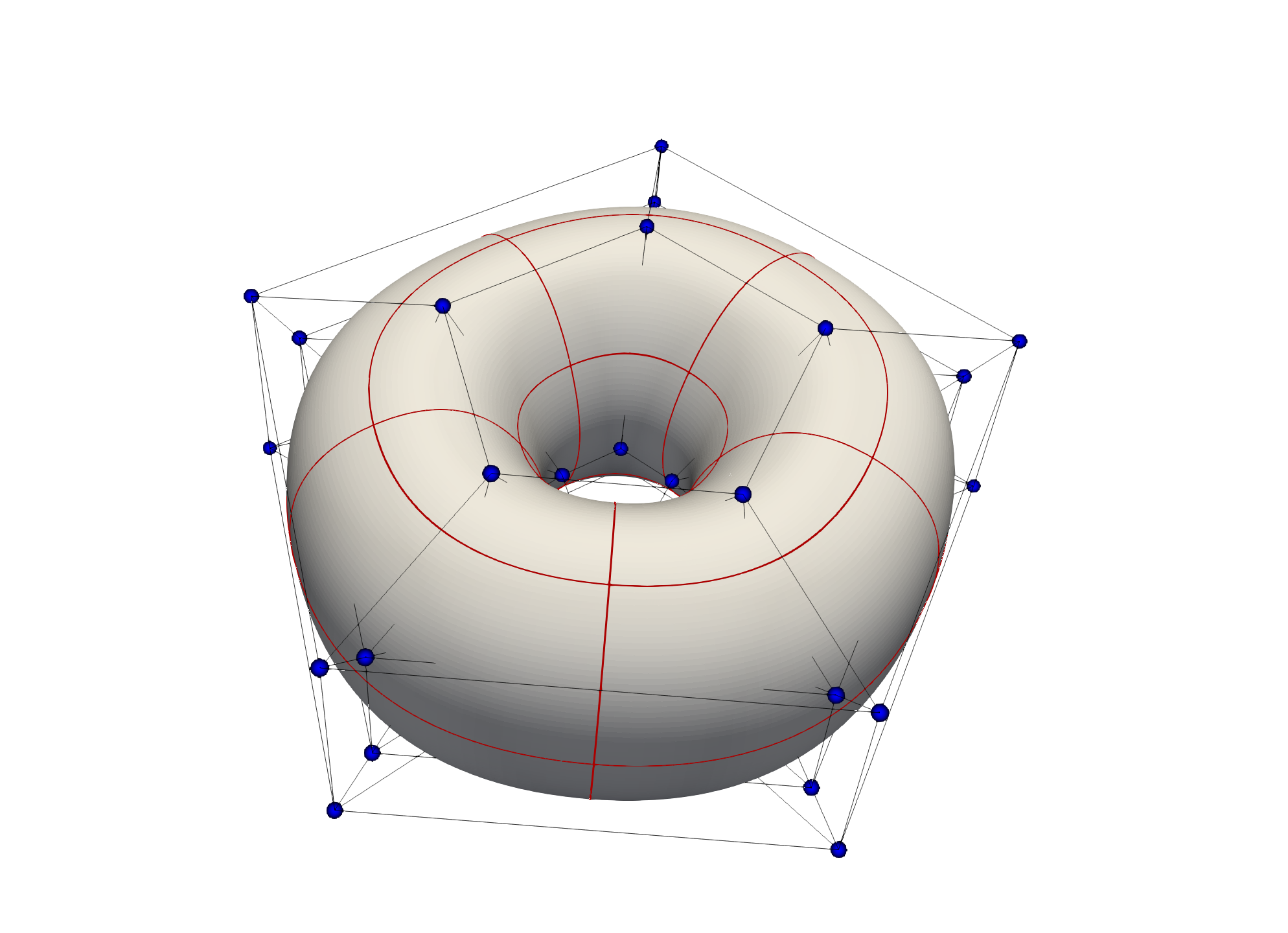}
  }
  \subfloat[$p = 3$]{
    \includegraphics[width=0.49\textwidth]{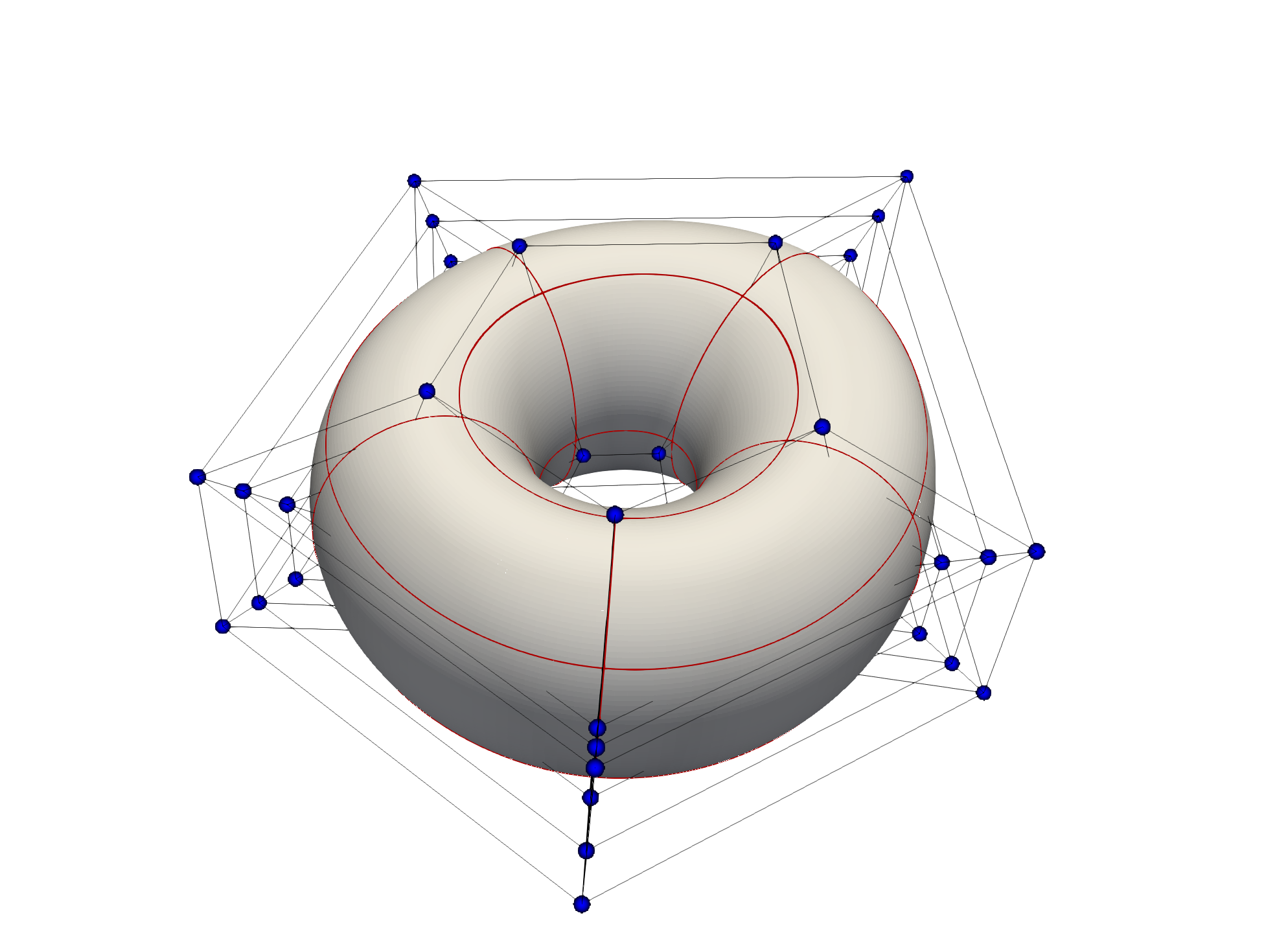}
  }
  \caption{The coarsest geometries used in the $L^2$ projection tests for degrees 2 and 3 are shown above. The mesh element boundaries are displayed in red, and the control points used for building the geometries are displayed as blue spheres.}\label{fig:geometries_l2projection}
\end{figure}

For the $L^2$ convergence study, we solve the following problems:
\begin{itemize}
  \item Given $f \in L^2(\Omega^{pol})$, find $f_h \in V^{pol}_k$, $k \in \{0, 3\}$, such that
  \begin{equation}
    \int_{\Omega^{pol}} (w_h - f_h)f \,d\Omega^{pol} = 0\;,\qquad w_h \in V^{pol}_k\;.
  \end{equation}
  \item Given $\pmb{f} \in [L^2(\Omega^{pol})]^3$, find $\pmb{f}_h \in V^{pol}_k$, $k \in \{1, 2\}$, such that
  \begin{equation}
    \int_{\Omega^{pol}} (\pmb{w}_h - \pmb{f}_h)\cdot \pmb{f} \,d\Omega^{pol} = 0\;,\qquad \pmb{w}_h \in V^{pol}_k\;.
  \end{equation}
\end{itemize}
Specifically, we consider the following smooth functions:
\begin{align}
f(x,y,z) &:= \sin(x/10)\sin(y/10)\sin(z/10)\;,\label{eq:l2-projection-scalar-target}\\
\pmb{f}(x,y,z) &:= (f(x,y,z), f(x,y,z), f(x,y,z))\;.\label{eq:l2-projection-vector-target}
\end{align}
The error convergence plots are shown in Figure \ref{fig:convergence}, where the $L^2$-norm of error is plotted against the number of degrees of freedom (denoted $\text{NDOF}_k$ for the space $V^{pol}_k$).
The plots show that the error converges with the expected rate, that is, the error decreases as $\text{NDOF}_0^{-\frac{p+1}{3}}$ for $k = 0$ and as $\text{NDOF}_k^{-\frac{p}{3}}$ for $k = 1, 2, 3$.

\begin{figure}
  \centering
  \subfloat[]{\includegraphics[page = 1]{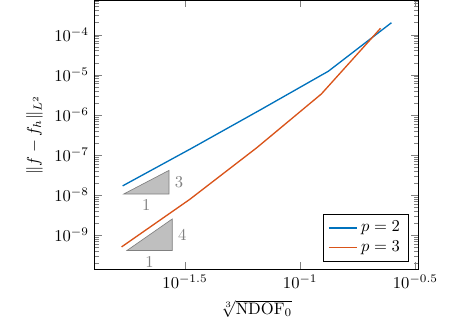}}
 \subfloat[]{ \includegraphics[page = 2]{convPlot}}\\
 \subfloat[]{\includegraphics[page = 3]{convPlot}}
 \subfloat[]{\includegraphics[page = 4]{convPlot}}
  \caption{$L^2$-norm of the projection error of the smooth functions in Equations \eqref{eq:l2-projection-scalar-target}--\eqref{eq:l2-projection-vector-target}, defined in the solids of Figure \ref{fig:geometries_l2projection}, onto the polar spaces. 
  The spaces are refined uniformly by bisection; see Section \ref{numerical} for details.
  We perform 4 such refinement steps for all spaces, except for the combination of $p = 3$ and $k = 1,2$ where we stop at 3 steps because of computational constraints.
  In (a) the error decays for the projection on the space $V_0^{pol}$ as $\text{NDOF}_0^{-(p+1)/3}$. In (b)--(d) the error decays for the projections on the spaces $V_{\ell}^{pol}$ for $\ell = 1, 2, 3$ as $\text{NDOF}_k^{-p/3}$.  The rates of convergence are the same that one would expect for non-singular geometric mappings, showing that the hurdle of the polar singularity is overcome.}\label{fig:convergence}
\end{figure}

\section{Conclusions}\label{conclusion}
We have presented a discretization of the de Rham complex \eqref{continuousderham} on a 3D solid toroidal domain $\Omega^{pol}$ by means of pushed-forward restricted spline spaces. Such discretization forms a spline complex preserving the cohomological structure of the continuous complex. 

Several extension of this work are possible. Our future research will be concentrated in developing de Rham complex discretizations using locally refinable B-spline spaces, such as LR B-spline spaces or THB-spline spaces, and establishing extraction processes of suitable reduced spaces to gain local adaptivity in polar domains. These aspects would then be combined to extend the discretization defined in this paper to spline spaces in which local adaptivity is allowed. 

\section*{Acknowledgments}
Francesco Patrizi is supported by the National Recovery and Resilience Plan, Mission 4 Component 2 – Investment 1.4 – CN\_00000013 ``CENTRO NAZIONALE HPC, BIG DATA E QUANTUM COMPUTING'', spoke 6 (CUP B83C22002830001) and he is member of Gruppo Nazionale per il Calcolo Scientifico, Istituto Nazionale di Alta Matematica.
Deepesh Toshniwal is supported by project number 212.150
awarded through the Veni research programme by the Dutch Research Council (NWO).
\begin{appendices}
\section{Proofs of commutations}\label{app:commutation12}
Given an element $\pmb{g} \in V^1$ with associated set of DOFs $\m{g}$, the application of matrix $\bbol{D}^{(1)}$ to $\m{g}$, defined in Equation \eqref{eq:D1}, produces a vector $\m{h} = \bbol{D}^{(1)}\m{g}$ with the entries $\{h_\ell\}_{\ell=1}^{n_2}$ given by
\begin{equation}\label{hl}
\hspace{-1cm}\begin{minipage}{\textwidth}
\begin{algorithm}[H]
\For{$k = 1, \ldots, n^\varphi$}{
$h_{1 + (k-1)\bar{n}_1} = g_{1 + (k-1)\bar{n}_1} - g_{1 + k\bar{n}_1} + g_{n^\varphi \bar{n}_1 + 2 + (k-1) \bar{n}_0} - g_{n^\varphi \bar{n}_1 + 1 + (k-1)\bar{n}_0}$\;
$h_{2 + (k-1)\bar{n}_1} = g_{2 + (k-1)\bar{n}_1} - g_{2 + k\bar{n}_1} + g_{n^\varphi \bar{n}_1 + 3 + (k-1) \bar{n}_0} - g_{n^\varphi \bar{n}_1 + 1 + (k-1)\bar{n}_0}$\;
\For{$j = 1, \ldots, n^\theta$}{
$\displaystyle h_{2 + j + (k-1)\bar{n}_1} = g_{2+j+(k-1)\bar{n}_1} - g_{2+j+k\bar{n}_1} +$\\$\displaystyle\qquad\qquad\qquad\qquad\qquad g_{n^\varphi \bar{n}_1 + 3+ j + (k-1)\bar{n}_0} - \sum_{\ell=1}^3 \bar{E}_{\ell, n^\theta + j} g_{n^\varphi\bar{n}_1 + \ell + (k-1)\bar{n}_0}$\;
$h_{2 + (n^r-2)n^\theta + j + (k-1)\bar{n}_1} = g_{2+(n^r-2)n^\theta + j  + (k-1)\bar{n}_1} - g_{2+(n^r-2)n^\theta + j + k\bar{n}_1} +$\\$\qquad\qquad\qquad\qquad\qquad\qquad\qquad g_{n^\varphi \bar{n}_1 + 3 + (j + 1)  + (k-1)\bar{n}_0} - g_{n^\varphi \bar{n}_1 + 3 + j + (k-1)\bar{n}_0}$\;
$h_{n^\varphi\bar{n}_1 + j + (k-1)\bar{n}_2} = g_{2 + j + (k-1)\bar{n}_1} - g_{2 + (j+1) + (k-1)\bar{n}_1} +$\\ $\qquad\qquad\qquad\qquad\qquad\qquad\qquad\qquad \displaystyle g_{2 +(n^r-2)n^\theta + j + (k-1)\bar{n}_1} - \sum_{\ell=1}^2 \bar{E}_{\ell, n^\theta + j}^{(1), R} g_{\ell + (k-1)\bar{n}_1}$\;
\For{$i = 1, \ldots, n^r-3$}{
$h_{2 + j + in^\theta + (k-1)\bar{n}_1} = g_{2+j+in^\theta + (k-1)\bar{n}_1} - g_{2+j+in^\theta + k\bar{n}_1}+$\\$\qquad\qquad\qquad\qquad\qquad\qquad g_{n^\varphi\bar{n}_1 + 3 + j + (i-1)n^\theta + (k-1)\bar{n}_0} - g_{n^\varphi \bar{n}_1 + 3+ j + in^\theta + (k-1)\bar{n}_0}$\; 
$h_{2 + (n^r-2)n^\theta + j + in^\theta + (k-1)\bar{n}_1} = g_{2+(n^r-2)n^\theta + j + in^\theta + (k-1)\bar{n}_1} - g_{2+(n^r-2)n^\theta + j + in^\theta + k\bar{n}_1} +$\\$\qquad\qquad\qquad\qquad\qquad\qquad\qquad\qquad g_{n^\varphi \bar{n}_1 + 3 + (j + 1) + in^\theta + (k-1)\bar{n}_0} - g_{n^\varphi \bar{n}_1 + 3 + j + in^\theta + (k-1)\bar{n}_0}$\;
$h_{n^\varphi\bar{n}_1 + j + in^\theta+(k-1)\bar{n}_2} = g_{2 + j +in^\theta+ (k-1)\bar{n}_1} - g_{2 + (j+1) + in^\theta + (k-1)\bar{n}_1} +$\\ $\qquad\qquad\qquad\qquad\qquad\qquad g_{2 + (n^r-2)n^\theta + j + in^\theta+ (k-1)\bar{n}_1} - g_{2 + (n^r-2)n^\theta + j +(i-1)n^\theta + (k-1)\bar{n}_1}$\;
}}}
\end{algorithm}

\end{minipage}
\end{equation}
Each entry of $\m{h}$ is a linear combination of four DOFs of $V_1$ associated to edges of $\cR$ framing a face. The only exceptions are those entries corresponding to faces with an edge on the cylinders at the center of the sides of $\cR$, for which we need a suitable linear combination of the edges inside such cylinders. 

Matrix $\bbol{E}^{(2)}$ defined in Equation \eqref{eq:E2} provides the representation of $\pmb{h} = \m{N}^{(2)} \cdot \m{h} \in V^2$ in terms of the tensor product basis of $\SSS^{p^r,p^\theta-1,p^\varphi-1}\times\SSS^{p^r-1,p^\theta,p^\varphi-1}\times\SSS^{p^r-1,p^\theta-1,p^\varphi}$. More precisely, a general element of the latter space is expressed using three sets of coefficients $\{h^1_{ijk}\}_{i,j,k=1}^{n^r, n^\theta, n^\varphi}, \{h^2_{ijk}\}_{i,j,k=1}^{n^r-1, n^\theta, n^\varphi}, \{h^2_{ijk}\}_{i,j,k=1}^{n^r, n^\theta, n^\varphi}$ for the three components and matrix $\bbol{E}^{(2)}$ describes how to choose them as follows
\begin{equation}\label{eq:hijk}
\hspace{-1cm}\begin{minipage}{\textwidth}
\begin{algorithm}[H]
\For{$k = 1, \ldots, n^\varphi$}{
\For{$j = 1, \ldots, n^\theta$}{
$h_{1jk}^1 = \displaystyle\sum_{\ell=1}^2 \bar{E}_{\ell, j}^{(1), R} h_{\ell + (k - 1)\bar{n}_1} = 0; \quad h_{2jk}^1 = \displaystyle\sum_{\ell = 1}^2 \bar{E}_{\ell, n^\theta + j}^{(1), R}h_{\ell + (k-1)\bar{n}_1}$\;
$h_{1jk}^2 = \displaystyle-\sum_{\ell=1}^2 \bar{E}_{\ell, j}^{(1), L} h_{\ell + (k-1)\bar{n}_1}$\;
$h_{1jk}^3 = 0$\;
\For{$i = 3, \ldots, n^r$}{
$h_{ijk}^1 = h_{2 + (n^r - 2)n^\theta + j + (i-3)n^\theta + (k-1)\bar{n}_1}$\;
$h_{(i-1)jk}^2 = -h_{2 + j + (i-3)n^\theta + (k-1)\bar{n}_1}$\;
$h_{(i-1)jk}^3 = h_{n^\varphi \bar{n}_1 + j + (i-3)n^\theta + (k - 1)\bar{n}_2}$\;
}}}
\end{algorithm}
\end{minipage}
\end{equation}
We can now prove Proposition \ref{commutation2}.
\begin{proof}
We split the cases $i = 1, 2$ and $i \geq 3$.  We shall use Equations \eqref{eq:tensorcurl}, \eqref{eq:gijk}, \eqref{hl} and \eqref{eq:hijk}, which provide the explicit relations for each DOF, and sometimes the DTA-compatible property of matrix $\bar{\bbol{E}}$, defined in Equation \eqref{eq:Ebar}. For $i = 1$ we have
\begin{equation*}
\begin{array}{l}
\begin{split}
\bullet~h_{1jk}^1 \stackrel{[ 0~-\bbol{D}^{(0,0,1)}~\bbol{D}^{(0,1,0)}]}{=}& g_{1(j+1)k}^3 - g_{1jk}^3 + g_{1jk}^2 - g_{1j(k+1)}^2 \stackrel{(\bbol{E}^{(1)})^T}{=} \sum_{\ell=1}^3 (\bar{E}_{\ell,(j+1)} - \bar{E}_{\ell,j})g_{n^\varphi\bar{n}_1 + \ell + (k-1)\bar{n}_0} + 0\\ &=0 \stackrel{\bbol{D}^{(1)}}{=} 0 \stackrel{(\bbol{E}^{(2)})^T}{=} h_{1jk}^1
\end{split}\\
\begin{split}
\bullet~h_{1jk}^2 \stackrel{[\bar{\bbol{D}}^{(0,0,1)}~0~-\bbol{D}^{(1, 0, 0)}]}{=} &g_{1j(k+1)}^1 - g_{1jk}^1 + g_{1jk}^3 - g_{2jk}^3 \stackrel{(\bbol{E}^{(1)})^T}{=} \sum_{\ell=1}^2 \bar{\bbol{E}}_{\ell, j}^{(1), L} (g_{\ell + k\bar{n}_1} - g_{\ell + (k-1)\bar{n}_1}) +\\
&- \sum_{\ell=1}^2 \bar{E}_{\ell, j}^{(1), L} g_{n^\varphi \bar{n}_1 + \ell + 1 + (k-1)\bar{n}_0} - (\bar{E}_{1, n^\theta + j} - \bar{E}_{1, j}) g_{n^\varphi\bar{n}_1 + 1 + (k-1)\bar{n}_0} +\\&+ \sum_{\ell = 1}^2 \bar{E}_{\ell, j}^{(1), L} g_{n^\varphi\bar{n}_1 + 1 + (k-1)\bar{n}_0} - \sum_{\ell = 1}^2 \bar{E}_{\ell, j}^{(1), L} g_{n^\varphi\bar{n}_1 + 1+ (k-1)\bar{n}_0}\\ &\stackrel{\text{DTA}}{=} \sum_{\ell =1}^2 \bar{E}_{\ell, j}^{(1), L} (g_{\ell + k\bar{n}_1} - g_{\ell + (k-1)\bar{n}_1} + g_{n^\varphi\bar{n}_1 + 1+ (k-1)\bar{n}_0} - g_{n^\varphi\bar{n}_1 + \ell + 1 + (k-1)\bar{n}_0})\\ &\stackrel{\bbol{D}^{(1)}}{=} -\sum_{\ell=1}^2 \bar{E}_{\ell, j}^{(1), L} h_{\ell + (k-1)\bar{n}_1} \stackrel{(\bbol{E}^{(2)})^T}{=} h_{1jk}^2 
\end{split}\\
\begin{split}
\bullet~h_{1jk}^3 \stackrel{[-\bar{\bbol{D}}^{(0, 1, 0)}~\bbol{D}^{(1, 0, 0)}~0]}{=} &g_{2jk}^2 - g_{1jk}^2 + g_{1jk}^1 - g_{1(j+1)k}^1 \stackrel{(\bbol{E}^{(1)})^T}{=} \sum_{\ell=1}^2 \left(\bar{E}_{\ell, n^\theta + j}^{(1), R} + \bar{E}_{\ell, j}^{(1), L} - \bar{E}_{\ell, j+1}^{(1), L}\right) g_{\ell + (k-1)\bar{n}_1} =\\ &=\sum_{\ell=1}^2 \bar{E}_{\ell, j}^{(1), R} g_{\ell + (k-1)\bar{n}_1} = 0 \stackrel{\bbol{D}^{(1)}}{=} 0 \stackrel{(\bbol{E}^{(2)})^T}{=} h_{1jk}^3.
\end{split}
\end{array}
\end{equation*}
Similarly, for $i = 2$ it holds
\begin{equation*}
\begin{array}{l}
\begin{split}
\bullet~h_{2jk}^1 \stackrel{[ 0~-\bbol{D}^{(0,0,1)}~\bbol{D}^{(0,1,0)}]}{=} &g_{2(j+1)k}^3 - g_{2jk}^3 + g_{2jk}^2 - g_{2j(k+1)}^2\\ &\stackrel{(\bbol{E}^{(1)})^T}{=} \sum_{\ell=1}^3 \left(\bar{E}_{\ell, n^\theta + j + 1} - \bar{E}_{\ell, n^\theta + j}\right) g_{n^\varphi\bar{n}_1 + \ell + (k-1)\bar{n}_0} + \sum_{\ell=1}^2 \bar{E}_{\ell, n^\theta + j}^{(1), R} (g_{\ell + (k-1)\bar{n}_1} - g_{\ell + k\bar{n}_1}) =\\
&= \sum_{\ell = 1}^2 \bar{E}_{\ell, n^\theta + j}^{(1), R} (g_{n^\varphi\bar{n}_1 + \ell + 1 + (k-1)\bar{n}_0} + g_{\ell + (k-1)\bar{n}_1} - g_{\ell + k\bar{n}_1}) +\\ &+\left(\bar{E}_{1, n^\theta + j + 1} - \bar{E}_{1, n^\theta + j}\right)g_{n^\varphi \bar{n}_1 + 1 + (k-1)\bar{n}_0}+\\&+ \sum_{\ell=1}^2 \bar{E}_{\ell, n^\theta + j}^{(1), R} g_{n^\varphi\bar{n}_1 + 1+(k-1)\bar{n}_0} - \sum_{\ell=1}^2 \bar{E}_{\ell, n^\theta + j}^{(1), R} g_{n^\varphi\bar{n}_1 + 1+(k-1)\bar{n}_0}\\ &\stackrel{\text{DTA}}{=} \sum_{\ell=1}^2 \bar{E}_{\ell, n^\theta + j}^{(1), R} (g_{n^\varphi\bar{n}_1 + \ell + 1 + (k-1)\bar{n}_0} + g_{\ell + (k-1)\bar{n}_1} - g_{\ell + k\bar{n}_1} - g_{n^\varphi\bar{n}_1 + 1+ (k-1)\bar{n}_0}) \\&\stackrel{\bbol{D}^{(1)}}{=} \sum_{\ell=1}^2 \bar{E}_{\ell, n^\theta + j}^{(1), R} h_{\ell + (k-1)\bar{n}_1} \stackrel{(\bbol{E}^{(2)})^T}{=} h_{2jk}^1,
\end{split}\\
\begin{split}
\bullet~h_{2jk}^2 \stackrel{[\bar{\bbol{D}}^{(0,0,1)}~0~-\bbol{D}^{(1, 0, 0)}]}{=} &g_{2j(k+1)}^1 - g_{2jk}^1 + g_{2jk}^3 - g_{3jk}^3 \\&\stackrel{(\bbol{E}^{(1)})^T}{=} g_{2+j+k\bar{n}_1} - g_{2+j+(k-1)\bar{n}_1} + \sum_{\ell=1}^3 \bar{E}_{\ell, n^\theta + j} g_{n^\varphi \bar{n}_1 + \ell + (k-1)\bar{n}_0} - g_{n^\varphi\bar{n}_1 + 3+j + (k-1)\bar{n}_0} \\ &\stackrel{\bbol{D}^{(1)}}{=} - h_{2+j+(k-1)\bar{n}_1} \stackrel{(\bbol{E}^{(2)})^T}{=} h_{2jk}^2,
\end{split}\\
\begin{split}
\bullet~h_{2jk}^3 \stackrel{[-\bar{\bbol{D}}^{(0, 1, 0)}~\bbol{D}^{(1, 0, 0)}~0]}{=} & g_{3jk}^2 - g_{2jk}^2 + g_{2jk}^1 - g_{2(j+1)k}^1\\ &\stackrel{(\bbol{E}^{(1)})^T}{=} g_{2 + (n^r - 2)n^\theta +j+ (k-1)\bar{n}_1} - \sum_{\ell=1}^2 \bar{E}_{\ell, n^\theta + j}^{(1), R} g_{\ell + (k-1)\bar{n}_1} + g_{2 + j+(k-1)\bar{n}_1} - g_{2 + j+1+(k-1)\bar{n}_1}\\ &\stackrel{\bbol{D}^{(1)}}{=} h_{n^\varphi\bar{n}_1 + j + (k-1)\bar{n}_2} \stackrel{(\bbol{E}^{(2)})^T}{=} h_{2jk}^3.
\end{split}
\end{array}
\end{equation*}
Finally, for $i > 3$ we get
\begin{equation*}
\begin{array}{l}
\begin{split}
\bullet~h_{ijk}^1 \stackrel{[ 0~-\bbol{D}^{(0,0,1)}~\bbol{D}^{(0,1,0)}]}{=} & g_{i(j+1)k}^3 - g_{ijk}^3 + g_{ijk}^2 - g_{ij(k+1)}^2\\ &\stackrel{(\bbol{E}^{(1)})^T}{=} g_{n^\varphi\bar{n}_1 + 3+j+1+(i-3)n^\theta + (k-1)\bar{n}_0} - g_{n^\varphi\bar{n}_1 + 3+j+(i-3)n^\theta + (k-1)\bar{n}_0}+\\ &+ g_{2+(n^r-2)n^\theta +j+(i-3)n^\theta + (k-1)\bar{n}_1} - g_{2+(n^r-2)n^\theta +j+(i-3)n^\theta + k\bar{n}_1}\\ &\stackrel{\bbol{D}^{(1)}}{=} h_{2 + (n^r -2)n^\theta + j+(i-3)n^\theta + (k-1)\bar{n}_1} \stackrel{(\bbol{E}^{(2)})^T}{=} h_{ijk}^1,
\end{split}\\
\begin{split}
\bullet~h_{ijk}^2 \stackrel{[\bar{\bbol{D}}^{(0,0,1)}~0~-\bbol{D}^{(1, 0, 0)}]}{=} & g_{ij(k+1)}^1 - g_{ijk}^1 + g_{ijk}^3 - g_{(i+1)jk}^3 \\ &\stackrel{(\bbol{E}^{(1)})^T}{=} g_{2+j+(i-3)n^\theta + k\bar{n}_1} - g_{2+j+(i-3)n^\theta + (k-1)n^\theta}+\\ &+ g_{n^\varphi\bar{n}_1 + 3 + j + (i-3)n^\theta + (k-1)\bar{n}_0} - g_{n^\varphi\bar{n}_1 + 3 + j + (i-2)n^\theta + (k-1)\bar{n}_0} \\ &\stackrel{\bbol{D}^{(1)}}{=} - h_{2 + j + (i-2)n^\theta + (k-1)\bar{n}_1} \stackrel{(\bbol{E}^{(2)})^T}{=} h_{ijk}^2,
\end{split}\\
\begin{split}
\bullet~h_{ijk}^3 \stackrel{[-\bar{\bbol{D}}^{(0, 1, 0)}~\bbol{D}^{(1, 0, 0)}~0]}{=} & g_{(i+1)jk}^2 - g_{ijk}^2 + g_{ijk}^1 - g_{i(j+1)k}^1 \\&\stackrel{(\bbol{E}^{(1)})^T}{=} g_{2 + (n^r-2)n^\theta + j + (i-2)n^\theta + (k-1)\bar{n}_1} - g_{2 + (n^r-2)n^\theta + j + (i-3)n^\theta + (k-1)\bar{n}_1}+\\ &+ g_{2 + j + (i-2)n^\theta + (k-1)\bar{n}_1} - g_{2 + j + 1 + (i-2)n^\theta + (k-1)\bar{n}_1}\\ &\stackrel{\bbol{D}^{(1)}}{=} h_{n^\varphi \bar{n}_1 + j + (i-2)n^\theta + (k-1)\bar{n}_2} \stackrel{(\bbol{E}^{(2)})^T}{=} h_{ijk}^3.
\end{split}
\end{array}
\end{equation*}
\end{proof}
Let now $\pmb{h} \in V^2$, $\pmb{h} = \m{N}^{(2)} \cdot \m{h}$. When applying $\bbol{D}^{(2)}$ to $\m{h}$, it produces a vector $\m{m} = \bbol{D}^{(2)}\m{h}$ with the entries $\{m_\ell\}_{\ell=1}^{n_3}$ given by
\begin{equation}\label{eq:divhpol}
\hspace{-1cm}\begin{minipage}{\textwidth}
\begin{algorithm}[H]
\For{$k=1, \ldots, n^\varphi$}{
\For{$j = 1, \ldots, n^\theta$}{
$m_{j+(k-1)\bar{n}_2} = h_{2 + (n^r - 2)n^\theta + j + (k - 1)\bar{n}_1} - \displaystyle\sum_{\ell = 1}^2 \bar{E}_{\ell, n^\theta + j}^{(1), R} h_{\ell + (k - 1)\bar{n}_1} +$\\ $\qquad\qquad\qquad + h_{2 + j + (k - 1)\bar{n}_1} - h_{2 + (j+1) + (k- 1)\bar{n}_1} + h_{n^\varphi \bar{n}_1 + j + k\bar{n}_2} - h_{n^\varphi \bar{n}_1 + j + (k-1)\bar{n}_2}$\;
}
\For{$i = 2, \ldots, n^r-2$}{
\For{$j=1, \ldots, n^\theta$}{
$m_{j+(i-1)n^\theta+(k-1)\bar{n}_2} = h_{2 + (n^r -2)n^\theta + j + (i-1)n^\theta + (k - 1)\bar{n}_1} - h_{2 + (n^r -2)n^\theta + j + (i-2)n^\theta + (k - 1)\bar{n}_1}+$\\
$\qquad\qquad\qquad\qquad + h_{2 + j + (i - 1)n^\theta + (k -1) \bar{n}_1} - h_{2 + (j + 1) + (i - 1)n^\theta + (k -1) \bar{n}_1}+$\\ $\qquad\qquad\qquad\qquad + h_{n^\varphi \bar{n}_1 + j + (i - 1)n^\theta + k\bar{n}_2} - h_{n^\varphi \bar{n}_1 + j + (i - 1)n^\theta + (k-1)\bar{n}_2}$\;
}}}
\end{algorithm}
\end{minipage}
\end{equation}
Each entry of $\m{m}$ is expressed as linear combination of six DOFs in $V_2$ associated to faces enclosing a volume in $\cR$. The only exceptions are those entries corresponding to volumes with a boundary face on the cylinders in the center of the sides of $\cR$, i.e., $m_{j+(k-1)\bar{n}_2}$ for $j=1, \ldots, n^\theta$ and $k=1, \ldots, n^\varphi$, for which we need a suitable linear combination of the two faces inside such cylinders, with coefficients taken from matrix $\bbol{E}_J^{(1), R}$ of Equation \eqref{eq:EJ1}.
Instead, matrix $\bbol{E}^{(3)}$ provides the relation between the DOFs $\{m_\ell\}_{\ell = 1}^{n_3}$ in the polar representation and the DOFs $\{m_{ijk}\}_{i,j,k=1}^{n^r-1, n^\theta, n^\varphi}$ in the tensor product space $\SSS^{p^r-1, p^\theta -1, p^\varphi - 1}$. More precisely, according to Equation \eqref{m} it holds:
\begin{equation}\label{eq:mpol}
\hspace{-1cm}\begin{minipage}{\textwidth}
\begin{algorithm}[H]
\For{$k = 1, \ldots, n^\varphi$}{
\For{$j = 1, \ldots, n^\theta$}{
$m_{1jk} = 0$\;
\For{$i = 2, \ldots, n^r-1$}{
$m_{ijk} = m_{j + (i-2)n^\theta + (k - 1)\bar{n}_2}$\;
}}}
\end{algorithm}
\end{minipage}
\end{equation}
We can finally prove Proposition \ref{commutation3}.
\begin{proof}
We use Equations \eqref{eq:divh}, \eqref{eq:hijk}, \eqref{eq:divhpol} and \eqref{eq:mpol} providing the relations between the DOFs involved in the diagram. The scope is to have a series of equality making a loop that is possible only if the commutation property holds. We split the cases $ i = 1, 2$ and $i \geq 3$, corresponding to the following three bullets:
\begin{equation*}
\begin{array}{l}
\begin{split}
\bullet~m_{1jk} \stackrel{\left[\bbol{D}^{(1, 0, 0)}~\bbol{D}^{(0, 1, 0)}~\bbol{D}^{(0, 0, 1)}\right]}{=}& h_{2jk}^1 - h_{1jk}^1 + h_{1(j+1)k}^2 - h_{1jk}^2 + h_{1j(k+1)}^3 - h_{1jk}^3 \\&\stackrel{(\bbol{E}^{(2)})^T}{=} \sum_{\ell=1}^2\left(\bar{E}_{\ell, n^\theta + j}^{(1), R} + \bar{E}_{\ell, j}^{(1), L} - \bar{E}_{\ell, j + 1}^{(1), L}\right)h_{\ell + (k-1)\bar{n}_1}\\ &= \sum_{\ell=1}^2 \bar{E}_{\ell, j}^{(1), R} h_{\ell + (k-1)\bar{n}_2} = 0 \stackrel{\bbol{D}^{(2)}}{=} 0 \stackrel{(\bbol{E}^{(3)})^T}{=} m_{1jk}
\end{split}\\
\begin{split}
\bullet~m_{2jk} \stackrel{\left[\bbol{D}^{(1, 0, 0)}~\bbol{D}^{(0, 1, 0)}~\bbol{D}^{(0, 0, 1)}\right]}{=}&
h_{3jk}^1 - h_{2jk}^1 + h_{2(j+1)k}^2 - h_{2jk}^2 + h_{2j(k+1)}^3 - h_{2jk}^3 \\&\stackrel{(\bbol{E}^{(2)})^T}{=} h_{2 + (n^r-2)n^\theta + j + (k-1)\bar{n}_1} - \sum_{\ell=1}^2 \bar{E}_{\ell, n^\theta + j}^{(1), R} h_{\ell + (k-1)\bar{n}_1}\\ &+ h_{2 + j + (k-1)\bar{n}_1} - h_{2 + (j+1) + (k-1)\bar{n}_1} + h_{n^\varphi \bar{n}_1 + j + k\bar{n}_2} - h_{n^\varphi \bar{n}_1 + j + (k-1)\bar{n}_2}\\ &\stackrel{\bbol{D}^{(2)}}{=} m_{j + (k-1)\bar{n}_2} \stackrel{(\bbol{E}^{(3)})^T}{=} m_{2jk}
\end{split}\\
\begin{split}
\bullet~m_{ijk} \stackrel{\left[\bbol{D}^{(1, 0, 0)}~\bbol{D}^{(0, 1, 0)}~\bbol{D}^{(0, 0, 1)}\right]}{=}&
h_{(i+1)jk}^1 - h_{ijk}^1 + h_{i(j+1)k}^2 - h_{ijk}^2 + h_{ij(k+1)}^3 - h_{ijk}^3 \\&\stackrel{(\bbol{E}^{(2)})^T}{=} h_{2 + (n^r -2)n^\theta + j + (i-2)n^\theta + (k-1)n^\theta} - h_{2 + (n^r -2)n^\theta + j + (i-3)n^\theta + (k-1)n^\theta}+\\&\qquad\qquad+h_{2 + j + (i-3)n^\theta + (k-1)\bar{n}_1} - h_{2 + (j+1) + (i-3)n^\theta + (k-1)\bar{n}_1}+\\&\qquad\qquad+ h_{n^\varphi\bar{n}_1 + j + (i-2)n^\theta + k\bar{n}_2} - h_{n^\varphi\bar{n}_1 + j + (i-2)n^\theta + (k-1)\bar{n}_2}\\&\stackrel{\bbol{D}^{(2)}}{=}m_{j + (i-2)n^\theta + (k-1)\bar{n}_2} \stackrel{(\bbol{E}^{(3)})^T}{=} m_{ijk}.
\end{split}
\end{array}
\end{equation*}
\end{proof}
\end{appendices}
\bibliography{biblio}
\end{document}